\documentclass[12pt,reqno]{amsart}
\usepackage{amsmath,amsthm,amssymb,mathrsfs,stmaryrd,color,mathtools}

\usepackage[all]{xy}
\usepackage{url}
\usepackage{tikz-cd}
\usepackage[margin=1in,footskip=.5in]{geometry}

\usepackage[utf8]{inputenc}
\usepackage[T1]{fontenc}

\usepackage{relsize}
\usepackage[bbgreekl]{mathbbol}
\usepackage{amsfonts}

\DeclareSymbolFontAlphabet{\mathbb}{AMSb}
\DeclareSymbolFontAlphabet{\mathbbl}{bbold}

\usepackage{enumitem}
\usepackage[colorlinks=true,hyperindex, linkcolor=magenta, pagebackref=false, citecolor=cyan,pdfpagelabels]{hyperref}
\usepackage[capitalize]{cleveref}

\theoremstyle{plain}
\newtheorem{thm}{Theorem}[section]
\newtheorem{thm2}{Theorem}

\newtheorem{prop}[thm]{Proposition}
\newtheorem{prop2}[thm2]{Proposition}
\newtheorem{ques}[thm]{Question}
\newtheorem{cor}[thm]{Corollary}
\newtheorem{cor2}[thm2]{Corollary}
\newtheorem{lem}[thm]{Lemma}

\theoremstyle{definition}
\newtheorem{defi}[thm]{Definition}
\newtheorem{rmk}[thm]{Remark}

\newtheorem{ass}[thm]{Assumption}
\newtheorem{ass2}[thm2]{Assumption}

\newtheorem{sett}[thm]{Setting}

\numberwithin{equation}{section}

\newcommand{\bb}[1]{\mathbb{#1}}
\newcommand{\cl}[1]{{\mathcal{#1}}}

\newcommand{\msf}[1]{{\mathsf{#1}}}
\newcommand{\mfr}[1]{{\mathfrak{#1}}}
\newcommand{\mrm}[1]{{\mathrm{#1}}}

\newcommand{\ov}[1]{{\overline{#1}}}

\newcommand{\wtd}[1]{{\widetilde{#1}}}

\newcommand{\Qla}{{\overline{\mathbb{Q}}_\ell}}
\newcommand{\Qlax}{{\overline{\mathbb{Q}}_\ell^\times}}

\newcommand{\Aut}{{\operatorname{Aut}}}

\newcommand{\Frob}{{\operatorname{Frob}}}

\newcommand{\Lie}{{\operatorname{Lie}}}
\newcommand{\Img}{{\operatorname{Im}}}

\newcommand{\Spec}{{\operatorname{Spec}}}

\newcommand{\PerfAlg}{{\operatorname{PerfAlg}}}

\newcommand{\Tr}{{\operatorname{Tr}}}

\newcommand{\Ker}{{\operatorname{Ker}}}

\newcommand{\id}{{\operatorname{id}}}

\newcommand{\pr}{{\operatorname{pr}}}

\newcommand{\ad}{{\operatorname{ad}}}

\newcommand{\Hom}{{\operatorname{Hom}}}

\newcommand{\Ind}{{\operatorname{Ind}}}
\newcommand{\cInd}{{\operatorname{cInd}}}
\newcommand{\tr}{{\operatorname{tr}}}

\newcommand{\der}{{\operatorname{der}}}
\newcommand{\dt}{{\operatorname{det}}}
\newcommand{\IC}{{\operatorname{IC}}}
\newcommand{\Nm}{{\operatorname{Nm}}}
\newcommand{\sgn}{{\operatorname{sgn}}}
\newcommand{\asym}{{\operatorname{asym}}}
\newcommand{\sym}{{\operatorname{sym}}}
\newcommand{\Yu}{{\operatorname{Yu}}}

\newcommand{\Av}{{\operatorname{Av}}}

\newcommand{\Gal}{{\operatorname{Gal}}}

\newcommand{\GL}{{\operatorname{GL}}}

\newcommand{\Stab}{{\operatorname{Stab}}}

\newcommand{\Sp}{{\operatorname{Sp}}}

\usepackage[citestyle=alphabetic,bibstyle=alphabetic,backend=bibtex, maxalphanames=10, maxnames=10, url=false, doi=false]{biblatex}
\title{Characteristic-free approaches around Yu's construction}
\author{Yuta Takaya}
\address{Graduate School of Mathematical Sciences, The University of Tokyo, 3-8-1 Komaba,
Meguro-ku, Tokyo 153-8914, Japan}
\email{takaya@ms.u-tokyo.ac.jp}
\begin{document}
\begin{abstract}
    We give a direct characteristic-free construction of twisted Heisenberg-Weil representations when there are no symmetric and ramified roots. As a consequence, we show that twisted Yu's construction naturally extends to residual characteristic $2$. Moreover, we give a geometric realization of such twisted Heisenberg-Weil representations via the Deligne-Lusztig construction for Heisenberg group schemes. As an application, we give an explicit description of positive-depth parahoric Deligne-Lusztig induction in the generic case. 
\end{abstract}

\maketitle
\setcounter{tocdepth}{2}
\tableofcontents

\section*{Introduction}

\subsection{Background}

In the smooth representation theory of $p$-adic groups, Yu's construction \cite{Yu01} is one of the most general and wide-ranging constructions of supercuspidal representations to date. It is the main source of Kaletha's local Langlands correspondence for regular supercuspidal representations (\cite{Kal19}, \cite{Kal21}). However, there are several technical subtleties in the construction, arising essentially from the use of symplectic Heisenberg-Weil representations developed in \cite{Ger77}. One is the exceptional behavior at $p = 2$, which is studied in detail in \cite{FS25}, and the other is the need to incorporate a quadratic twist as in \cite{FKS23} for the construction of stable $L$-packets. We refer to this twisted version as \textit{twisted} Yu's construction. This quadratic twist is intrinsic also from the geometric viewpoint: it has been observed (e.g. \cite{CO25_JEMS}, \cite{CO25}, \cite{Nie24}, \cite{IN26}) that twisted Yu's construction is innately realized in the cohomology of positive-depth Deligne-Lusztig varieties. 

In this paper, we aim to give direct characteristic-free constructions of \textit{twisted} Heisenberg-Weil representations from representation-theoretic and geometric viewpoints. As a consequence, we show that \textit{twisted} Yu's construction naturally extends to $p = 2$ and provide an explicit description of positive-depth parahoric Deligne-Lusztig induction in the generic case. 

\subsection{Representation-theoretic construction}

The first part is on the representation-theoretic side. Instead of starting from the symplectic Heisenberg-Weil representation and then correcting it by a quadratic character, we decompose the relevant symplectic space into pieces of general linear and unitary type. On these pieces, the required Heisenberg-Weil representations are constructed by G\'{e}rardin \cite{Ger77} uniformly for all $p$. The quadratic character of \cite{FKS23} then appears as the comparison between the resulting tensor product and the classical symplectic Heisenberg-Weil representation when $p \neq 2$.

Let $G$ be a connected reductive group over a non-archimedean local field $F$ and let $M \subset G$ be a tamely ramified twisted Levi subgroup. Let $x \in \cl{\wtd{B}}(M, F)$ be a point of the enlarged Bruhat-Tits building and let $T \subset M$ be a tamely ramified maximal torus such that $x \in \cl{A}(T, F)$. Let $E$ be a tamely ramified Galois extension of $F$ such that $T$ splits over $E$. Throughout the introduction, we put ourselves in the following setting. 

\begin{ass2} \label{ass2:without_symmetric_ramification}
    There are no nontrivial symmetric and ramified $Z_M$-weights in $\Lie(G)$.  
\end{ass2}

In fact, this holds if the ramification index of $E / F$ is odd, so it is automatic when $p = 2$ or $M$ splits over the completed maximal unramified extension $\breve{F}$ of $F$. One benefit of this assumption is that it simplifies the description of the quadratic twist $\epsilon_x^{G / M}$ in \cite{FKS23}. 

Fix a depth $r > 0$ and let $\msf{m}_{x, r} = \Lie(M)_{x, r} / \Lie(M)_{x, r+}$ be the graded piece of the Moy-Prasad filtration. Let $k$ be the residue field of $F$ and take a $k$-linear map $\phi \colon \msf{m}_{x, r} \to k$ that is stable under conjugation by $M(F)_x$ and satisfies a certain genericity condition (see \Cref{defi:r_generic}). Then, $\msf{m}^{\perp}_{x, r/2} = \Lie(M)^\perp_{x, r/2} / \Lie(M)^\perp_{x, r/2+}$ is equipped with a symplectic form $\langle -, - \rangle_\phi$ coming from the commutator pairing. An intrinsic benefit of \Cref{ass2:without_symmetric_ramification} is the following simple but crucial observation. 

To state this observation, let $\Gamma = \Gal(E / F)$ and $\Sigma = \Gamma \times \{ \pm  1\}$. Let $\Phi^M_\Sigma$ be the set of $\Sigma$-orbits of nontrivial $Z_M$-weights in $\Lie(G)$ and consider an orthogonal weight decomposition $\msf{m}^\perp_{x, r/2} = \bigoplus_{C \in \Phi^M_\Sigma} \msf{m}^{\perp, C}_{x, r/2}$ stable under $M(F)_x$. 
For each $Z_M$-weight $\alpha$, let $F_\alpha = E^{\Stab_\Gamma(\alpha)}$ and let $k_\alpha$ be the residue field of $F_\alpha$. 

\begin{prop2} \label{prop2:decomposition_m_perp}
    Let $C \in \Phi_{\Sigma}^M$ and fix a representative $\alpha \in C$. 
    \begin{enumerate}
        \item If $C \neq \Gamma \cdot \alpha$, $\msf{m}^{\perp, C}_{x, r/2} \cong V^\alpha \oplus (V^\alpha)^*$ for some vector space $V^\alpha$ over $k_\alpha$. Moreover, the action of $M(F)_x$ on $\msf{m}^{\perp, C}_{x, r/2}$ factors through $\GL(V_\alpha / k_\alpha)$.
        \item  If $C = \Gamma \cdot \alpha$, $\msf{m}^{\perp, C}_{x, r/2} \cong V^\alpha$ for some vector space $V^\alpha$ equipped with a non-degenerate skew-Hermitian form over $k_\alpha$. Moreover, the action of $M(F)_x$ on $\msf{m}^{\perp, C}_{x, r/2}$ factors through the unitary group $\msf{U}(V_\alpha / k_\alpha)$.
    \end{enumerate}
\end{prop2}

This follows from a detailed study of the weight decomposition of $\msf{m}^\perp_{x, r/2}$ over the algebraic closure $\ov{k}$. We prove this in an axiomatic way in \Cref{prop:orbit_description}. 

Let $\msf{H}_{x, r, \phi} = (G(F)_{x, r, r/2} / G(F)_{x, r+, r/2+})/\Ker(\phi)$ be the associated Heisenberg group. Let 
\[
    \msf{M}_{x, r, \phi} \subset \Sp(\msf{m}^{\perp}_{x, r/2}, \langle - , - \rangle_\phi)
\]
be the image of the adjoint action of $M(F)_x$. It can be seen that the adjoint action of $M(F)_x$ on $\msf{H}_{x, r, \phi}$ factors through $\msf{M}_{x, r, \phi}$ (see \Cref{lem:adjoint_M_xrphi}). Then, we can form a Heisenberg-Weil group $\msf{HW}_{x, r, \phi} = \msf{H}_{x, r, \phi}  \rtimes \msf{M}_{x, r, \phi}$. In this setting, we can construct Heisenberg-Weil representations uniformly and describe their characters explicitly. Fix a nontrivial character $\psi \colon k \to \Qlax$. 

\begin{thm2} \textup{(\Cref{thm:Heisenberg_Weil_representations})}
    \label{thm2:Heisenberg_Weil_representations}
    There is a unique irreducible representation $\kappa_{x, r, \phi}$ of $\msf{HW}_{x, r, \phi}$ such that
    \begin{enumerate}
        \item $\kappa_{x, r, \phi} \vert_{\msf{H}_{x, r, \phi}}$ is the Heisenberg representation with central character $\psi$, and
        \item $\tr(\gamma \mid \kappa_{x, r, \phi}) = (-1)^{d_\gamma} \sqrt{\lvert \msf{m}^{\perp, \gamma}_{x, r/2} \rvert}$ for every $\gamma \in \msf{M}_{x, r, \phi}$ (see \Cref{defi:d_gamma_sign_factor} for $d_\gamma$). 
    \end{enumerate}
    Moreover, the trace function of $\kappa_{x, r, \phi}$ is supported on the elements conjugate to $k \cdot \msf{M}_{x, r, \phi}$. 
\end{thm2}

Here, we utilize the decomposition in \Cref{prop2:decomposition_m_perp} and construct $\kappa_{x, r, \phi}$ as a product of Heisenberg-Weil representations for general linear groups and unitary groups, as constructed in \cite{Ger77} in arbitrary characteristic. When $p \neq 2$, we show in \Cref{prop:comparison_epsilon_twist} that $\kappa_{x, r, \phi}$ differs from the symplectic Heisenberg-Weil representation $\kappa_\Yu$ exactly by $\epsilon_x^{G / M}$. In other words, $\epsilon_x^{G / M}$ measures the difference between polarized or unitary Heisenberg-Weil representations and symplectic ones. 

\subsubsection{Application to Yu's construction}

By replacing $\kappa_\Yu$ with $\kappa_{x, r, \phi}$, we get twisted Yu's construction in residual characteristic $2$. Thanks to a description of the restriction of $\kappa_{x, r, \phi}$ to parabolic subgroups (see \Cref{prop:restrction_as_induction}), the existing proofs in \cite{Yu01}, \cite{Fin21} and \cite{FKS23} work even in residual characteristic $2$. In this paper, we directly show the preservation of irreducibility under the positive-depth induction step of Yu's construction. 


Let $\phi \colon M(F) \to \Qlax$ be a $(G, M)$-supergeneric character of depth $r$ (see \Cref{defi:genericity}) and let $\rho$ be an $M(F)_x / M(F)_{x, r}$-representation. When $Z_M / Z_G$ is anisotropic, we say that the $G(F)_x / G(F)_{x, r+}$-representation
\[
    \Ind_{r, \phi}(\rho) = \Ind^{G(F)_x}_{G(F)_{x, r/2} M(F)_x} (\kappa_{x, r, \phi} \otimes \rho)
\]
is the positive-depth induction of $\rho$ twisted by $\phi$ (see \Cref{defi:twisted_positive_depth_induction}). In fact, we will see later that $\Ind_{r, \phi}(\rho)$ arises as the positive-depth Deligne-Lusztig induction of $\rho \otimes \phi$. Here, the main representation-theoretic property of $\Ind_{r, \phi}$ is as follows. 

\begin{thm2} \textup{(\Cref{thm:intertwining_induction})}
    Let $\rho$ be an irreducible representation of $M(F)_x / M(F)_{x, r}$. 
    \begin{enumerate}
        \item $\Ind_{r, \phi}(\rho)$ is an irreducible representation of $G(F)_x / G(F)_{x, r+}$. 
        \item If every intertwiner of $\rho$ in $M(F)$ lies in $M(F)_x$, then every intertwiner of $\Ind_{r, \phi}(\rho)$ in $G(F)$ lies in $G(F)_x$. 
    \end{enumerate}
    In particular, if $\cInd_{M(F)_x}^{M(F)} \rho$ is irreducible and supercuspidal, then so is $\cInd_{G(F)_x}^{G(F)} \Ind_{r, \phi}(\rho)$. 
\end{thm2}

When $p \neq 2$, the same results hold without \Cref{ass2:without_symmetric_ramification} by setting $\kappa_{x, r, \phi} = \kappa_\Yu \otimes \epsilon^{G / M}_x$. Then, twisted Yu's construction can be regarded as an iteration of positive-depth induction. 

\begin{cor2}\textup{(\Cref{thm:twisted_Yu_construction})} \label{thm2:twisted_Yu_construction}
    For every generic datum $((G_i)_{1 \leq i \leq n + 1}, x, (r_i)_{1\leq i \leq n}, \rho, (\phi)_{1 \leq i \leq n})$, $\cInd_{\wtd{K}}^{G(F)} \wtd{\rho}$ is irreducible and supercuspidal. 
\end{cor2}

Here, $\wtd{K}$ and $\wtd{\rho}$ are defined as in the usual Yu's construction. In particular, our construction does not involve any additional choices for $p = 2$ in contrast to \cite{FS25}. Nevertheless, it would be interesting to study the behavior of $\Ind_{r, \phi}$ when $\phi$ does not satisfy {\bf GE2} (see \Cref{rmk:behavior_without_GE2}). 

\vspace{7.5pt}
\noindent
{\bf Remark.} We explain the improvement of our work from \cite{FS25} in the case of $p = 2$, where the authors construct irreducible supercuspidal representations of the form $\cInd_K^{G(F)} \rho$ from generic data without {\bf GE2} when $q = \lvert k \rvert \geq 4$. The main difference is that the Heisenberg-Weil representations constructed in \cite[Section 3.4]{FS25} are not defined entirely over Heisenberg-Weil groups $\msf{HW}_{x, r, \phi}$. They remedy this issue with Clifford theory at the cost of making choices of extensions of Heisenberg-Weil representations and leaving $\rho$ and $K$ inexplicit. Instead, we construct entire Heisenberg-Weil representations in the twisted setting and make the construction canonical when we impose {\bf GE2} as in \cite{Yu01}. Our method also allows $q = 2$. 


\subsection{Geometric realization}

In the second part, we study the geometric side. A standard approach in the geometric realization of twisted Yu's construction (e.g. \cite{CO25_JEMS}, \cite{CO25}, \cite{Nie24}, \cite{IN26}) is via the computation of trace characters, but it suffers from technical largeness conditions on $p$ and $q$ due to the limitation of elements whose trace can be computed. We will avoid this issue by realizing $\kappa_{x, r, \phi}$ geometrically in arbitrary residual characteristic via a more explicit study of the geometry.


Let $\Phi^M_e$ be the set of inertial orbits of nontrivial $Z_M$-weights in $\Lie(G)$. Fix a decomposition $\Phi_e^M = \Phi_e^{M+} \sqcup \Phi_e^{M-}$ such that $\Phi_e^{M-} = - \Phi_e^{M+}$ and let 
\[
    I_0 = \{ \alpha \in \Phi_e^{M+} \mid \Gal(\ov{k} / k) \cdot \alpha \subset \Phi_e^{M+} \}, \quad 
    I_1 = \Phi_e^{M+} \cap \sigma(\Phi_e^{M-}), \quad 
    I = I_0 \sqcup I_1.  
\]
Let $\msf{u}_{I, r/2} \subset \msf{m}^{\perp}_{x, r/2} \otimes \ov{k}$ be the direct summand coming from the $Z_M$-weight decomposition. Since $\msf{u}_{I, r/2}$ is isotropic, the inclusion can be lifted to $H_{x, r, \phi}$ (see \Cref{prop:lifting_to_Heisenberg}). 

\begin{thm2}\textup{(\Cref{cor:Heisenberg_realization} and \Cref{thm:extended_HW_rep})} \label{thm2:coh_W_Ir}
    Let $\msf{W}_{I, r} \to \msf{u}_{I, r}$ be the finite \'{e}tale $\msf{H}_{x, r, \phi}$-torsor given by the Cartesian diagram
    \begin{center}
        \begin{tikzcd}
            \msf{W}_{I, r} \ar[r] \ar[d] & H_{x, r, \phi, \ov{k}} \ar[d, "h \mapsto \sigma(h) h^{-1}"] \\
            \msf{u}_{I, r/2} \ar[r] & H_{x, r, \phi, \ov{k}}. 
        \end{tikzcd}
    \end{center}
    Then, as a representation of $\msf{HW}_{x, r, \phi}$, we have 
    \[
        H_c^i(\msf{W}_{I, r}, \Qla)[\psi] = 0 \quad (i \neq n) ,\quad 
        H_c^n(\msf{W}_{I, r}, \Qla)[\psi] \cong \kappa_{x, r, \phi} 
    \]  
    for some $n$. Moreover, $n = \dim \msf{W}_{I, r}$ when $I_0 = \phi$. In that case, the natural map
    \[
        H_c^n(\msf{W}_{I, r}, \Qla)[\psi] \to H^n(\msf{W}_{I, r}, \Qla)[\psi]
    \]  
    is an isomorphism. 
\end{thm2}

This construction can be viewed as the Deligne-Lusztig construction for Heisenberg groups. Here, $H_c^i(\msf{W}_{I, r}, \Qla)[\psi]$ denotes the $\psi$-isotypic part of the \'{e}tale cohomology for the central action of $k \subset \msf{HW}_{x, r, \phi}$. In \Cref{thm:cohAS}, we first compute the cohomology of certain Artin-Schreier sheaves and deduce $\dim H_c^i(\msf{W}_{I, r}, \Qla)[\psi]$ as an application. It is enough to determine the action of $\msf{H}_{x, r, \phi}$, but one needs to be careful about the $\msf{M}_{x, r, \phi}$-action since a standard method only computes the trace of tame elements. It turns out that exceptional cases where it is not sufficient (e.g. when $q = 2$) can be reduced to the computation in \cite{IT23}. 

In fact, $\msf{W}_{I, r}$ is found out in the course of the generalization of \cite{BW16}: in \cite{Tak26_special_affinoids}, we show that $\msf{W}_{I, r}$ is realized as (the perfection of) reductions of special affinoids in local Shimura varieties. The property $H_c^n(\msf{W}_{I, r}, \Qla)[\psi] \cong H^n(\msf{W}_{I, r}, \Qla)[\psi]$ plays an important role in constructing explicit supercuspidal representations in the middle \'{e}tale cohomology of local Shimura varieties. 

We note that the use of a polarization $\Phi_e^M = \Phi_e^{M+} \sqcup \Phi_e^{M-}$ looks similar to the positive-depth Deligne-Lusztig theory of \cite{IN26} in the tamely ramified setting. We expect that even without \Cref{ass2:without_symmetric_ramification}, $\kappa_{x, r, \phi}$ should be geometrically realized by utilizing $p \neq 2$. 

\subsubsection{Application to positive-depth Deligne-Lusztig theory}

For this application, we assume that $M$ splits over $\breve{F}$ and $Z_M / Z_G$ is anisotropic. To obtain the middle concentration property, we work with 
\[
    \msf{Y}_{P, r} = \{ h \in \msf{U}_r \cap \sigma^{-1}(\msf{U}_r) \backslash \msf{G}_{r, \ov{k}} \mid \sigma(h) h^{-1} \in \msf{U}_r \}
\] 
as a positive-depth parahoric Deligne-Lusztig variety (see \Cref{ssec:preliminaries} for our notation). We follow the approach taken in \cite{IN25} to apply \Cref{thm2:coh_W_Ir} to the computation of positive-depth parahoric Deligne-Lusztig induction. In \Cref{defi:MP_DL_context}, we introduce a depth filtration $\{\msf{Y}_{P, [s, r]}\}_{s > 0}$ inside $\msf{Y}_{P, r}$ so that $\msf{W}_{I, r}$ is realized as a quotient of $\msf{Y}_{[r/2, r]}$. By showing that the \textit{generic} contribution is concentrated on $\msf{Y}_{P, [r/2, r]}$, we get the following computation. Let $\IC_{\msf{Y}, r} = \Qla[\dim \msf{Y}_{P, r}]$ be the intersection cohomology complex on $\msf{Y}_{P, r}$. 

\begin{thm2}\textup{(\Cref{cor:induction_step_positive} and \Cref{thm:positive_depth_DL_induction})}
    Let $\phi \colon M(F) \to \Qlax$ be a $(G, M)$-generic character of depth $r$ and let $\rho$ be a representation of $M(F)_x / M(F)_{x, r}$. For every $0 < s \leq \tfrac{r}{2}$, 
    \[
        R\Gamma_c(\msf{Y}_{P, [s, r]}, \IC_{\msf{Y}, r}\vert_{\msf{Y}_{P, [s, r]}}) \mathop{\otimes}_{\msf{M}_r(k)} (\rho \otimes \phi) = \Ind_{r, \phi}(\rho). 
    \]
    Moreover, when $P$ is convex (see \Cref{defi:convex_parabolic}), we have 
    \[
        R\Gamma_c(\msf{Y}_{P, r}, \IC_{\msf{Y}, r}) \mathop{\otimes}_{\msf{M}_r(k)} (\rho \otimes \phi) = \Ind_{r, \phi}(\rho). 
    \]
\end{thm2}

Here, the main obstacle to establishing the last claim in full generality is \Cref{conj:obstruction_for_depth_0_induction}, concerning the generic vanishing of the cohomology of certain explicit subvarieties of Lie algebras. The convexity assumption on $P$ introduced in \cite{IN25} is quite a clean way to simplify the associated geometry. It would be interesting to see if the natural map
\[
    R\Gamma_c(\msf{Y}_{P, r}, \IC_{\msf{Y}, r}) \mathop{\otimes}\limits_{\msf{M}_r(k)} (\rho \otimes \phi) \to R\Gamma(\msf{Y}_{P, r}, \IC_{\msf{Y}, r}) \mathop{\otimes}\limits_{\msf{M}_r(k)} (\rho \otimes \phi) 
\]
is an isomorphism as in the Deligne-Lusztig theory (see \cite[Theorem 11.7]{BR03}). As in loc. cit., Yu's genericity condition {\bf GE2} is irrelevant to middle concentration. 

In particular, we verify \cite[Conjecture 5.12]{CO25} for convex parabolic subgroups, which conjectures the Euler characteristic formula of $R\Gamma_c(\msf{Y}_{P, r}, \IC_{\msf{Y}, r}) \otimes (\rho \otimes \phi)$. As commented in loc. cit., \cite[Theorem 5.6]{CO25} is recovered for suitably chosen Borel subgroups without technical conditions on residual characteristic. 


\addtocontents{toc}{\protect\setcounter{tocdepth}{1}}
\subsection*{The structure of the paper}

In \Cref{sec:symplectic_geometry}, we prove \Cref{prop2:decomposition_m_perp} by introducing an axiomatic framework for decomposition of symplectic representations. In \Cref{sec:twisted_Heisenberg_Weil}, we construct the twisted Heisenberg-Weil representation $\kappa_{x, r, \phi}$ as a tensor product of polarized or unitary Heisenberg-Weil representations. In \Cref{sec:Yu_construction}, we show that twisted Yu's construction works in residual characteristic $2$. In \Cref{sec:cohAS}, we compute the cohomology of certain Artin-Schreier sheaves, and we use this result to show that $\kappa_{x, r, \phi}$ is realized in the cohomology of $\msf{W}_{I, r}$ in \Cref{sec:geometric_realization}. In \Cref{sec:positive_depth_induction}, we apply the computation in \Cref{sec:geometric_realization} to the explicit description of positive-depth Deligne-Lusztig induction in the generic case.  

\subsection*{Acknowledgements}
I would like to thank my advisor Yoichi Mieda for his constant support and encouragement. This work was supported by the WINGS-FMSP program at the Graduate School of Mathematical Sciences, the University of Tokyo and JSPS KAKENHI Grant number JP24KJ0865.

\subsection*{Notation}
All rings are assumed to be commutative. We fix two primes $p \neq \ell$. The center of an algebraic group $G$ is denoted by $Z_G$. For an endomorphism $\tau$ of a vector space or a scheme, the $\tau$-fixed locus is denoted by the superscript $(-)^\tau$. Unless otherwise stated (i.e. except for \Cref{sec:symplectic_geometry}), the coefficient field of representations of finite groups is taken as $\Qla$. 

Let $k$ be a finite field in characteristic $p$ of cardinality $q$ and let $\ov{k}$ be its algebraic closure. We fix a nontrivial additive character $\psi \colon k \to \Qlax$. The base change from $k$ to $\ov{k}$ is denoted by $\ov{(-)}$ or $(-)_{\ov{k}}$. The Frobenius action relative to $k$ is denoted by $\sigma$. Note that $\sigma$ also denotes the Frobenius element in $\Gal(\ov{k} / k)$. Let $\PerfAlg$ be the category of perfect $\ov{k}$-algebras. For a perfect $\ov{k}$-scheme $\msf{X}$ perfectly of finite type, $x \in \msf{X}$ denotes a $\ov{k}$-valued point of $\msf{X}$ unless otherwise stated. 

\addtocontents{toc}{\protect\setcounter{tocdepth}{2}}

\section{Decomposition of symplectic representations} \label{sec:symplectic_geometry}

\subsection{Genuine weight decomposition} \label{ssec:weight_decomposition}

In this section, we study symplectic representations over finite fields equipped with a certain direct sum decomposition and show that they can be decomposed into polarized or unitary symplectic representations. 



\begin{defi}
    We say that a smooth action of $\Gal(\ov{k} / k) \times \{\pm 1\}$ on a set $S$ is \textit{genuine} if $s \neq -s$ for every $s \in S$. Here, $-s = (-1) \cdot s$. 
\end{defi}

\begin{defi} \label{defi:weight_decomposition}
    Let $\msf{M}$ be a finite group and let $(V, \langle - , -\rangle)$ be a symplectic $\msf{M}$-representation over $k$. We say that an $\msf{M}$-stable direct sum decomposition
    \[
        \ov{V} = \bigoplus_{s \in S} \ov{V}^s
    \]
    over $\ov{k}$ is a \textit{genuine weight decomposition} if $S$ is equipped with a genuine smooth action of $\Gal(\ov{k} / k) \times \{\pm 1\}$ so that the following hold for every $s \in S$. 
    \begin{enumerate}
        \item $\sigma(\ov{V}^s) = \ov{V}^{\sigma(s)}$. 
        \item $\ov{V}^t$ is orthogonal to $\ov{V}^s$ for every $t \neq -s$. 
    \end{enumerate}
\end{defi}

Since $\langle -, - \rangle$ is non-degenerate, (2) implies that $\langle -, - \rangle \vert_{\ov{V}^s \times \ov{V}^{-s}}$ is non-degenerate. 

Now, fix a symplectic $\msf{M}$-representation $(V, \langle - , -\rangle)$ over $k$ equipped with a genuine weight decomposition indexed by $S$. Then, we show that $V$ can be decomposed into polarized or unitary symplectic $\msf{M}$-representations. 

\begin{defi} \label{defi:rational_form}
    For $s\in S$, let $d_s$ be the minimum positive integer with $\sigma^{d_s}(s) = s$. Let $k_s$ be the finite extension of $k$ of degree $d_s$ and let $V^s  = (\ov{V}^s)^{\sigma^{d_s}}$ be the $k_s$-rational form of $\ov{V}^s$. 
\end{defi}
 
\begin{defi} \label{defi:C_component}
    For each subset $C \subset S$ stable under $\Gal(\ov{k} / k)$, let 
    \[
        \ov{V}^C = \bigoplus_{s \in C} \ov{V}^s \subset \ov{V} ,\quad V^C = \left(\ov{V}^C \right)^{\Gal(\ov{k} / k)} \subset V. 
    \]
\end{defi}

\begin{prop} \label{prop:orbit_description}
    Let $C \subset S$ be an orbit under $\Gal(\ov{k} / k) \times \{ \pm 1\}$ and let $s \in C$ be a representative. Then, the following hold.
    \begin{enumerate}
        \item When $C \neq \Gal(\ov{k} / k) \cdot s$, we have $V^C \cong V^s \oplus V^{-s}$. Moreover, 
        \[
            B \colon V^s \times V^{-s} \to k_s, \quad 
            (v, w) \mapsto \langle v, w \rangle
        \]
        is an $\msf{M}$-stable non-degenerate pairing such that under $V^C \cong V^s \oplus V^{-s}$, 
        \[
            \langle (v,v^*), (w, w^*) \rangle = \Tr_{k_s / k} (B(v, w^*) - B(w, v^*)) \quad 
            (v, w \in V^s, v^*, w^* \in V^{-s})
        \]
        In particular, $V^C$ is a polarized symplectic $\msf{M}$-representation over $k_s$. 
        \item When $C = \Gal(\ov{k} / k) \cdot s$, we have $V^C \cong V^s$ and $d_s$ is even. Moreover, 
        \[
            B \colon V^s \times V^s \to k_s, \quad (v, w) \mapsto \langle v, \sigma^{d_s/2}(w) \rangle
        \] 
        is an $\msf{M}$-stable non-degenerate skew-Hermitian form such that under $V^C \cong V^s$, 
        \[  
            \langle v , w \rangle = \Tr_{k_s / k}(B(v, w)) \quad (v, w \in V^s). 
        \]
        In particular, $(V^C, B)$ is a skew-Hermitian unitary $\msf{M}$-representation over $k_s$. 
    \end{enumerate}
\end{prop}
\begin{proof}
    First, we show (1). In this case, $C = \Gal(\ov{k} / k) \cdot s \sqcup  \Gal(\ov{k} / k) \cdot (-s)$. We have
    \[
        V^s \oplus V^{-s} \cong V^C, \quad 
        (v, v^*) \mapsto \sum_{0 \leq i < d_s} \sigma^i(v) + \sum_{0 \leq i < d_s} \sigma^i(v^*)
    \]  
    since $d_s = d_{-s}$. By condition (2) in \Cref{defi:weight_decomposition}, $B$ is a $k_s$-linear non-degenerate pairing, and $\langle (v,v^*), (w, w^*) \rangle = \Tr_{k_s / k} (B(v, w^*) - B(w, v^*))$ follows from the above description. 

    Next, we show (2). In this case, $C = \Gal(\ov{k} / k) \cdot s$, so $C = -C$. It follows that $d_s = \lvert C \rvert$ is even since the action of $\Gal(\ov{k} / k) \times \{ \pm 1 \}$ on $S$ is genuine. Then, $\sigma^{d_s / 2}$ induces $V^s \cong V^{-s}$ and condition (2) in \Cref{defi:weight_decomposition} implies that $B$ is a non-degenerate skew-Hermitian form. We have an isomorphism
    \[
        V^s \cong V^C ,\quad v \mapsto \sum_{0\leq k < d_s} \sigma^k(v)
    \]
    and it is easy to check $\langle v , w \rangle = \Tr_{k_s / k}(B(v, w))$ under this isomorphism.   
\end{proof}     

In particular, the direct sum decomposition
\[
    V = \bigoplus_{C \subset S} V^C
\]
over orbits $C \subset S$ is a decomposition into polarized or unitary symplectic $\msf{M}$-representations. 

\subsection{Tame toral action without symmetric ramification} \label{ssec:tame_toral_action}

In this section, we explain how weight decompositions of symplectic representations appear in a practical situation. We fix a finite group $\msf{M}$. 

\begin{defi} \label{defi:tame_extension_group}
    We say that a profinite group $\Gamma$ equipped with an extension 
    \[
        0 \to \bb{Z} / e \bb{Z} \to \Gamma \to \Gal(\ov{k} / k) \to 0
    \]
    for some $e \geq 1$ coprime to $p$ is a tame extension of $\Gal(\ov{k} / k)$. 
\end{defi}

\begin{defi} \label{defi:tame_toral_action}
    Let $\msf{\ov{T}}$ be a torus over $\ov{k}$ and let $(\ov{V}, \langle - , -\rangle)$ be a symplectic $\msf{M}$-representation over $\ov{k}$. Let $\Gamma$ be a tame extension of $\Gal(\ov{k} / k)$. We say that a datum of actions $(\msf{\ov{T}} \circlearrowright \ov{V}, \Gamma \circlearrowright \msf{\ov{T}}, \ov{V})$ is a \textit{tame toral action} if the following conditions hold. 
    \begin{enumerate}
        \item The toral action $\msf{\ov{T}} \circlearrowright \ov{V}$ preserves $\langle -, - \rangle$ and commutes with $\msf{M}$. 
        \item The tame actions $\Gamma \circlearrowright \msf{\ov{T}}, \ov{V}$ are semilinear over $\ov{k}$ and commute with $\msf{M}$. Moreover, 
        \[
            \langle \gamma(v), \gamma(w) \rangle = \gamma\langle v, w \rangle ,\quad 
            \gamma(t \cdot v) = \gamma(t) \cdot \gamma(v)
        \]
        for every $v, w \in \ov{V}$, $t \in \msf{\ov{T}}$ and $\gamma \in \Gamma$. 
        \item The profinite group actions $\Gamma \circlearrowright X^*(\msf{\ov{T}}), \ov{V}$ are smooth. Here, $X^*(\msf{\ov{T}})$ is the set of characters of $\msf{\ov{T}}$. 
    \end{enumerate}
    Let $\Phi \subset X^*(\msf{\ov{T}})$ be the set of weights in $\ov{V}$. We say that a tame toral action $(\msf{\ov{T}}, \ov{V}, \Gamma)$ is \textit{without symmetric ramification} if $- \alpha \notin \bb{Z} / e \bb{Z} \cdot \alpha$ for every $\alpha \in \Phi$. 
\end{defi}

Here, $\Phi = - \Phi$ since $\msf{\ov{T}} \circlearrowright \ov{V}$ preserves the symplectic pairing $\langle -, - \rangle$. Let $\ov{V}^\alpha \subset \ov{V}$ be the weight space of weight $\alpha$ for each $\alpha \in \Phi$. By condition (1), each $\ov{V}^\alpha$ is stable under $\msf{M}$. 

\begin{lem} \label{lem:criterion_without_symmetric_ramification}
    A tame toral action is without symmetric ramification when $e$ is odd and $0 \not \in \Phi$. 
\end{lem}
\begin{proof}
    When $e$ is odd, any $\bb{Z} / e \bb{Z}$-orbit $C$ in $X^*(\msf{\ov{T}})$ has an odd number of elements. However, if $0 \not \in C$ and $C = - C$, $C$ has even number of elements, so we get a contradiction. 
\end{proof}

Now, fix a tame toral action $(\msf{\ov{T}}, \ov{V}, \Gamma)$ without symmetric ramification. Let $\tau$ be a generator of $\bb{Z}/ e \bb{Z}$. We will show that the set of $\tau$-invariants $\ov{V}^{\tau}$ admits a $k$-rational form $V^\tau$ and a genuine weight decomposition. 

\begin{defi}
    Let $\Phi_e = \Phi / \langle \tau \rangle$ be the set of $\bb{Z} / e \bb{Z}$-orbits in $\Phi$. For each $C \in \Phi_e$, let $-C = \{ -c \mid c\in C\} \in \Phi_e$ and let
    \[
        \ov{V}^C = \bigoplus_{\alpha \in C} \ov{V}^\alpha, \quad 
        \ov{V}^{\tau, C} = \left(\ov{V}^C\right)^{\tau} \subset \ov{V}^\tau. 
    \]
    Here, $(-)^\tau$ denotes the set of $\tau$-invariants. 
\end{defi}

\begin{prop} \label{prop:Vtau_decomposition}
    Let $V^\tau = \ov{V}^\Gamma$. Then, $V^\tau$ is a $k$-rational form of $\ov{V}^\tau$ and a symplectic $\msf{M}$-representation over $k$. Moreover, the decomposition
    \[
        \ov{V}^\tau = \bigoplus_{C \in \Phi_e} \ov{V}^{\tau, C}
    \]
    is a genuine weight decomposition. 
\end{prop}
\begin{proof}
    First, we show that $\ov{V}^\tau$ is symplectic. For every nonzero $v \in \ov{V}^\tau$, $\langle v, w \rangle \neq 0$ for some $w \in \ov{V}$. Then, $\langle v, \sum_{0 \leq i < e} \tau^i w \rangle = e \langle v, w \rangle \neq 0$ since $e$ is coprime to $p$, so $\ov{V}^\tau$ is symplectic. 

    Since the $\Gal(\ov{k} / k)$-action on $\ov{V}^\tau$ is smooth, it admits a $k$-rational form $V^\tau$ as given in the statement. By condition (2) in \Cref{defi:tame_toral_action}, $V^\tau$ is stable under $\msf{M}$ and the symplectic pairing $\langle - , - \rangle\vert_{\ov{V}^\tau}$ descends to $V^\tau$. Thus, $V^\tau$ is a symplectic $\msf{M}$-representation over $k$. 
    
    It is enough to show that the decomposition indexed by $\Phi_e$ is a genuine weight decomposition. By condition (2) in \Cref{defi:tame_toral_action}, there is a natural $\Gamma$-action on $\Phi$ so that $\gamma(\ov{V}^\alpha) = \ov{V}^{\gamma(\alpha)}$ for every $\alpha \in \Phi$ and $\gamma \in \Gamma$. It induces a smooth $\Gal(\ov{k} / k)$-action on $\Phi_e$ since $\Gamma / (\bb{Z} / e \bb{Z}) \cong \Gal(\ov{k} / k)$. The involution sending $C$ to $-C$ is an action of $\{ \pm 1\}$ on $\Phi_e$ commuting with $\Gal(\ov{k} / k)$. Since $(\ov{\msf{T}}, \ov{V}, \Gamma)$ is without symmetric ramification, $C \neq -C$ for every $C \in \Phi_e$, so the action of $\Gal(\ov{k} / k) \times \{ \pm 1\}$ on $\Phi_e$ is genuine. 
    
    By construction, condition (1) in \Cref{defi:weight_decomposition} is satisfied. Moreover, condition (1) in \Cref{defi:tame_toral_action} implies that $\ov{V}^\alpha$ is orthogonal to $\ov{V}^\beta$ if $\beta \neq -\alpha$. It implies condition (2) in \Cref{defi:weight_decomposition}. 
\end{proof}

\begin{cor} \label{cor:decomposition_tame_toral}
    For every tame toral action $(\msf{\ov{T}}, \ov{V}, \Gamma)$ without symmetric ramification, the symplectic $\msf{M}$-representation $V^\tau$ can be decomposed into polarized or unitary symplectic $\msf{M}$-representations. 
\end{cor}
\begin{proof}
    It follows from \Cref{prop:orbit_description} and \Cref{prop:Vtau_decomposition}. 
\end{proof}

\section{Twisted Heisenberg-Weil representations} \label{sec:twisted_Heisenberg_Weil}

In this section, we give a direct interpretation of twisted Heisenberg-Weil representations in Yu's construction when there are no symmetric and ramified roots. In particular, our construction works uniformly even if $p = 2$. 

\subsection{Preliminaries} \label{ssec:prelim_HW}
 
In this section, we introduce basic notation in Yu's construction. We will follow \cite{KP23} for Bruhat-Tits buildings and Moy-Prasad filtrations. 

Let $F$ be a non-archimedean local field and let $O_F$ be the ring of integers of $F$ with a residue field $k$. Let $\pi$ be a uniformizer of $F$. 

\begin{sett}
    Let $G$ be a connected reductive group over $F$ and let $M \subset G$ be a tamely ramified twisted Levi subgroup. Let $E / F$ be a finite tamely ramified extension such that $M$ splits over $E$. 
\end{sett}

Let $\pi_E$ be a uniformizer of $E$ and let $k_E$ be the residue field of $E$. Let $\breve{F}$ (resp.\ $\breve{E}$) be the completed maximal unramified extension of $F$ (resp.\ $E$). Fix an embedding $\breve{F} \subset \breve{E}$ so that $\breve{E} = E \breve{F}$. Let $e = [\breve{E} \colon \breve{F}]$ and take the additive valuation $\nu$ on $\breve{E}$ so that $\nu(\pi_E) = \tfrac{1}{e}$. Then, $\Gamma = \Aut(\breve{E} / F)$ is a tame extension of $\Gal(\ov{k} / k)$ in the sense of \Cref{defi:tame_extension_group}. Let $\Sigma = \Gamma \times \{\pm 1\}$. 

For a perfect ring $R$, $W(R)$ denotes the ring of Witt vectors when $F$ is in characteristic $0$, and denotes the ring of power series $R\llbracket \pi \rrbracket$ when $F$ is in characteristic $p$. For a perfect $k$-algebra $R$, we set $W_{O_F}(R)=W(R)\otimes_{W(k)} O_F$. The Frobenius on $W_{O_F}(R)$ relative to $k$ is denoted by $\sigma$. 
We set $W_{O_{\breve{E}}}(R) = W(R) \otimes_{W(\ov{k})} O_{\breve{E}}$ for $R \in \PerfAlg$. It is equipped with a natural semilinear $\Gamma$-action over $O_F$.  

For an $F$-scheme $X$, its base change to $\breve{F}$ is denoted by $\breve{X}$ and the Frobenius action on $\breve{X}$ is simply denoted by $\sigma$.  By abuse of notation, its base change to $\breve{E}$ is denoted by $\breve{X}_E$.

\begin{sett}
    Let $x \in \cl{\wtd{B}}(M, F)$ be a point of the enlarged Bruhat-Tits building and let $T \subset M$ be a tamely ramified maximal torus such that $x \in \cl{A}(T, F)$. We enlarge $E$ so that $T$ splits over $E$. 
\end{sett}

Here, $\cl{\wtd{B}}(-, F)$ denotes the enlarged Bruhat-Tits building. As in \cite[Section 14.2]{KP23}, we fix an admissible embedding $\cl{\wtd{B}}(M, F) \subset \cl{\wtd{B}}(G, F)$. Let $M(F)_x \subset M(F)$ (resp.\ $G(F)_x \subset G(F)$) be the full stabilizer of the image of $x$ in the reduced Bruhat-Tits building $\cl{B}(M, F)$ (resp.\ $\cl{B}(G, F)$). 

Let $\Phi$ be the set of roots of $\breve{G}_{E}$ with respect to $\breve{T}_E$. Let $\Phi_M \subset \Phi$ be the set of roots inside $\breve{M}_E$ and let $\Phi^M = \Phi - \Phi^M$. For each $\alpha \in \Phi$, let $U_\alpha \subset \breve{G}_E$ be the root group associated to $\alpha$ and let $\nu_{\alpha, x}$ be the valuation on $U_{\alpha}(\breve{E})$ associated to $x$. For each $r \geq 0$, let $U_{\alpha, x, r}  = \nu_{\alpha, x}^{-1}([r, \infty])$ and let $\cl{U}_{\alpha, x, r}$ be the associated integral model of $U_\alpha$ over $O_{\breve{E}}$. Similarly, let $\breve{T}_{E, x, r} \subset T(\breve{E})$ be the Moy-Prasad filtration at $x$ and let $\cl{\breve{T}}_{E, x, r}$ be the associated integral model of $\breve{T}_E$.

Let $\mfr{u}_\alpha = \Lie(U_\alpha)$, $\mfr{\breve{t}}_{E} = \Lie(\breve{T}_E)$, $\mfr{\breve{m}}_{E} = \Lie(\breve{M}_E)$ and $\mfr{\breve{g}}_{E} = \Lie(\breve{G}_E)$. Let
\[
    \mfr{\breve{m}}_{E}^\perp = \bigoplus_{\alpha \in \Phi^M} \mfr{u}_\alpha
\]
be a complement of $\mfr{\breve{m}}_E$ in $\mfr{\breve{g}}_E$. In each case, $(-)_{x, r}$ denotes the Moy-Prasad filtration for Lie algebras (e.g. $\mfr{u}_{\alpha, x, r} \subset \mfr{u}_\alpha$ and $\mfr{\breve{t}}_{E, x, r} \subset \mfr{\breve{t}}_E$). Let $\mfr{t} = \Lie(T)$, $\mfr{m} = \Lie(M)$ and $\mfr{g} = \Lie(G)$. The semilinear $\Gamma$-action on $\breve{\mfr{g}}_E$ stabilizes $\mfr{\breve{m}}_{E}^\perp$ and let $\mfr{m}^\perp = (\mfr{\breve{m}}_{E}^\perp)^\Gamma$ be the associated complement of $\mfr{m}$ inside $\mfr{g}$. 


For $r \geq 0$, $r+$ denotes $r + \varepsilon$ for sufficiently small $\varepsilon > 0$. Since the valuation on $\breve{E}$ is discrete, $r+$ is valid as an index of Moy-Prasad filtrations. For each $r \geq 0$, we use the sans-serif type to denote the quotient $\ov{k}$-vector spaces such as 
\[
    \msf{u}_{x, \alpha, r} = \mfr{u}_{x, \alpha, r} / \mfr{u}_{x, \alpha, r+}, \quad \msf{t}_{E, x, r} = \mfr{\breve{t}}_{E, x, r} / \mfr{\breve{t}}_{E, x, r+}. 
\]
Here, we remove breve accents and use similar notation for other Lie algebras. When there is no confusion, such $\ov{k}$-vector spaces are also regarded as $\ov{k}$-varieties. Similarly, we use the sans-serif type to denote the quotient $k$-vector spaces such as 
\[
    \msf{m}_{x, r} = \mfr{m}_{x, r} / \mfr{m}_{x, r+}, \quad 
    \msf{m}^\perp_{x, r} = \mfr{m}^\perp_{x, r} / \mfr{m}^\perp_{x, r+}. 
\]

\subsection{Heisenberg group schemes}

In this section, we introduce Heisenberg group schemes as quotients of positive loop groups of certain Moy-Prasad group schemes. They later play an essential role in the geometrization of twisted Heisenberg-Weil representations. 

\begin{defi}
    For each $0 < \tfrac{r}{2} \leq s \leq r$, let $f_{r, s}\colon \Phi \cup \{0\} \to \bb{R}_{>0}$ be the concave function 
    \[
        f_{r, s}(\alpha) = \left\{
            \begin{alignedat}{4}
                & s & \; & (\alpha \in \Phi^M) \\
                & r & \; & (\alpha \in \Phi_M \cup \{0\}). 
            \end{alignedat}
        \right.
    \]
\end{defi}

\begin{defi} \label{defi:MP_subgroup}
    Let $\breve{\cl{G}}_{E, x, r, s}$ be the Moy-Prasad group scheme over $O_{\breve{E}}$ associated to $f_{r, s}$ at $x$ (see \cite[Theorem 8.5.2]{KP23}) and let $(L\breve{G}_E)_{x, r, s} = L^+\breve{\cl{G}}_{E, x, r, s}$ be the positive loop group associated to $\breve{\cl{G}}_{E, x, r, s}$. In other words, it is a group presheaf
    \[
        (L\breve{G}_E)_{x, r, s}(R) = \breve{\cl{G}}_{E, x, r, s}(W_{O_{\breve{E}}}(R)) ,\quad 
        R \in \PerfAlg. 
    \]
\end{defi}

\begin{lem} \label{lem:product_description_loop}
    Let $0 < \tfrac{r}{2} \leq s \leq r$. The natural open immersion \textup{(see \cite[Theorem 8.5.2]{KP23})}
    \[
        \cl{\breve{T}}_{E, x, r} \times \prod_{\alpha \in \Phi} \cl{U}_{\alpha, x, f_{r, s}} \hookrightarrow \cl{\breve{G}}_{E, x, r, s}
    \]
    induces an isomorphism 
    \[
        (L\breve{G}_E)_{x, r, s} \cong L^+\cl{\breve{T}}_{E, x, r} \times \prod_{\alpha \in \Phi} L^+\cl{U}_{\alpha, x, f_{r, s}}. 
    \]
    Here, the order of the product factors can be chosen arbitrarily. 
\end{lem}
\begin{proof}
    Since the positive loop functor sends open immersions to open immersions, the map 
    \[
        L^+\cl{\breve{T}}_{E, x, r} \times \prod_{\alpha \in \Phi} L^+\cl{U}_{\alpha, x, f_{r, s}} \to (L\breve{G}_E)_{x, r, s}
    \]
    is an open immersion. To see that it is an isomorphism, it is enough to see the bijectivity on geometric points. This follows from \cite[Proposition 7.3.12 (4),(5)]{KP23}. 
\end{proof}

\begin{lem} \label{lem:piecewise_quotient}
    For $\alpha \in \Phi$ and $r \geq 0$, there is a short exact sequence
    \[
        0 \to L^+\cl{U}_{\alpha, x, r+} \to L^+ \cl{U}_{\alpha, x, r} \to \msf{u}_{\alpha, x, r} \to 0
    \]
    of group presheaves on $\PerfAlg$. Similarly, for $r > 0$, there is a short exact sequence
    \[
        0 \to L^+\cl{\breve{T}}_{E, x, r +} \to L^+\cl{\breve{T}}_{E, x, r} \to \msf{t}_{E, x, r} \to 0. 
    \]
\end{lem}
\begin{proof}
    The claims are immediate by construction. 
\end{proof}


\begin{lem}
    For $r > 0$, $(L\breve{G}_E)_{x, r+, (r/2)+} \subset (L\breve{G}_E)_{x, r, r/2}$ is a closed normal subgroup scheme. 
\end{lem}
\begin{proof}
    By \Cref{lem:product_description_loop} and \Cref{lem:piecewise_quotient}, $(L\breve{G}_E)_{x, r+, (r/2)+} \subset (L\breve{G}_E)_{x, r, r/2}$ is closed. We show that $(L\breve{G}_E)_{x, r+, (r/2)+} \subset (L\breve{G}_E)_{x, r, r/2}$ is normal. Since the inclusion is closed, it is enough to check on geometric points. Then, the claim follows from \cite[Proposition 13.2.5 (4)]{KP23}. 
\end{proof}

\begin{prop} \label{prop:Heisenberg_over_E}
    Let $H_{E, x, r} = (L\breve{G}_E)_{x, r, r/2} / (L\breve{G}_E)_{x, r+, (r/2)+}$ be the quotient presheaf. Then, there is a canonical central extension sequence 
    \[
        0 \to \msf{m}_{E, x, r} \to H_{E, x, r} \to \msf{m}^\perp_{E, x, r/2} \to 0
    \]
    and there is a section $\msf{m}^\perp_{E, x, r/2} \to H_{E, x, r}$. In particular, $H_{E, x, r}$ is representable by an affine perfect group scheme perfectly of finite presentation over $\ov{k}$. 
\end{prop}
\begin{proof}
    We have a short exact sequence
    \[
        0 \to (L\breve{G}_E)_{x, r, (r/2)+} / (L\breve{G}_E)_{x, r+, (r/2)+} \to H_{E, x, r} \to (L\breve{G}_E)_{x, r, r/2} / (L\breve{G}_E)_{x, r, (r/2)+} \to 0. 
    \]
    Here, both left and right sides are abelian due to the Moy-Prasad isomorphism, and the description in the statement follows from \Cref{lem:product_description_loop}. For $x \in \breve{G}_{E, x, r/2}$ and $y \in \breve{G}_{E, x, (r/2)+}$, $[x, y] \in \breve{G}_{E, x, r+}$, so the extension is central. 
    
    For each $\alpha \in \Phi^M$, $L^+\cl{U}_{\alpha, x, r/2} \to (L\breve{G}_E)_{x, r, r/2}$ induces a lift $\msf{u}_{\alpha, x, r/2} \to H_{E, x, r/2}$ by \Cref{lem:piecewise_quotient}. It follows that $H_{E, x, r} \cong \msf{m}_{E, x, r} \times \msf{m}^\perp_{E, x, r/2}$ as a presheaf on $\PerfAlg$, so it is represented by an affine perfect group scheme perfectly of finite presentation. 
\end{proof}

There is a natural semilinear $\Gamma$-action on $(L\breve{G}_E)_{x, r, r/2}$ and it induces a semilinear $\Gamma$-action on the quotient $H_{E, x, r}$. First, we take the fixed locus under the linear inertial action. Recall that $\tau \in \Gal(\breve{E} / \breve{F})$ is a generator. 

\begin{prop}
    Let $H_{E, x, r}^\tau \subset H_{E, x, r}$ be the fixed locus of $\tau$. There is a unique affine perfect group scheme $H_{x, r}$ over $k$ such that $H_{x, r, \ov{k}} =  H_{E, x, r}^\tau$ equivariantly under $\Gal(\ov{k} / k)$. 
\end{prop}
\begin{proof}
    Since $H_{E, x, r}$ is affine, $H_{E, x, r}^\tau$ is affine and the semilinear $\Gamma$-action on $H_{E, x, r}$ induces a semilinear $\Gal(\ov{k} / k)$-action on $H_{E, x, r}^\tau$. Since positive loop groups can be defined over $k_E$, the Galois action is smooth. Moreover, $H_{E, x, r}^\tau$ is of finite presentation, so the Galois descent datum determines a unique affine perfect group scheme $H_{x, r}$ over $k$. 
\end{proof}

We adopt the notation in \cite{Fin21} and let $G_{x, r, s} = G(F) \cap \breve{\cl{G}}_{E, x, r, s}(O_{\breve{E}})$. In \cite{Yu01}, it is denoted by $(M, G)_{x, r, s}$ and it is shown that $G_{x, r, s}$ is independent of the choice of $E$ and $T$. 

\begin{lem} \label{lem:pointset_Hxr}
    There is a central extension sequence 
    \[
        0 \to \msf{m}_{x, r} \to H_{x, r} \to \msf{m}^\perp_{x, r/2} \to 0. 
    \]
    Moreover, we have $H_{x, r}(k) \cong G_{x, r, r/2} / G_{x, r+, (r/2)+}$. 
\end{lem}
\begin{proof}
    Since $e$ is coprime to $p$, \Cref{prop:Heisenberg_over_E} induces a short exact sequence 
    \[
        0 \to \msf{m}_{E, x, r}^\tau \to H_{E, x, r}^\tau \to (\msf{m}^\perp_{E, x, r/2})^\tau \to 0. 
    \]
    By the Galois equivariance, it provides a short exact sequence 
    \[
        0 \to \msf{m}_{x, r} \to H_{x, r} \to \msf{m}^\perp_{x, r/2} \to 0. 
    \]
    By construction, $H_{x, r}(k) \cong H_{E, x, r}(\ov{k})^\Gamma$, so there is a natural injection $G_{x, r, r/2} / G_{x, r+, (r/2)+} \to H_{x, r}(k)$. 
    Since we have a compatible short exact sequence coming from
    \[
        0 \to G_{x, r, (r/2)+} / G_{x, r+, (r/2)+} \to G_{x, r, r/2} / G_{x, r+, (r/2)+} \to G_{x, r, r/2} / G_{x, r, (r/2)+} \to 0, 
    \]
    we also get the surjectivity of $G_{x, r, r/2} / G_{x, r+, (r/2)+} \to H_{x, r}(k)$. 
\end{proof}

The commutator pairings of $H_{x, r}$ and $H_{E, x, r}$ provide pairings 
\[
    \langle -, - \rangle \colon \msf{m}^\perp_{x, r/2} \times \msf{m}^\perp_{x, r/2} \to \msf{m}_{x, r}, \quad 
    \langle -, - \rangle_E \colon \msf{m}^\perp_{E, x, r/2} \times \msf{m}^\perp_{E, x, r/2} \to \msf{m}_{E, x, r}. 
\]
Moreover, since $e$ is coprime to $p$, the inclusion $\msf{m}_{x, r, \ov{k}} = \msf{m}_{E, x, r}^\tau \subset \msf{m}_{E, x, r}$ admits a canonical retraction $\Av_\tau = \tfrac{1}{e} \sum_{0 \leq i < e} \tau^i$. For each $k$-linear map $\phi \colon \msf{m}_{x, r} \to k$, let $\ov{\phi}$ be its base change to $\ov{k}$ and let $\phi_E = \ov{\phi} \circ \Av_\tau$. Then, we set 
\[
    \langle -, - \rangle_\phi = \phi \circ \langle -, - \rangle, \quad 
    \langle -, - \rangle_{E, \phi} = \phi_E \circ \langle -, - \rangle. 
\]

\begin{defi} \label{defi:r_generic}
    We say that a $k$-linear map $\phi \colon \msf{m}_{x, r} \to k$ is $r$-\textit{generic} if it is nonzero and $\langle -, - \rangle_{E, \phi}$ is non-degenerate. Let
    \[
        H_{x, r, \phi} = H_{x, r} / \Ker(\phi) ,\quad 
        H_{E, x, r, \phi} = H_{E, x, r} / \Ker(\phi_E), \quad 
        \msf{H}_{x, r, \phi} = H_{x, r, \phi}(k). 
    \]
    We say that $H_{x, r, \phi}$ is the Heisenberg group scheme associated to $\phi$. 
\end{defi}

\subsection{Decomposition of Heisenberg-Weil groups}

In this section, we apply the result of \Cref{sec:symplectic_geometry} to obtain decompositions of Heisenberg-Weil groups into polarized or unitary Heisenberg-Weil groups. From now on, we fix an $r$-generic $k$-linear map $\phi \colon \msf{m}_{x, r} \to k$.  

Let $Z_M^\circ \subset M$ be the identity component of the center of $M$. Let $\cl{\breve{Z}}_{M, E}$ be the standard integral model of $Z_M^\circ$ over $O_{\breve{E}}$ (see \cite[Section B.2]{KP23}). 

\begin{lem}
    Let $\msf{Z}_{\msf{M}}$ be the special fiber of $\cl{\breve{Z}}_{M, E}$. Then, there is a natural action of $\msf{Z}_{\msf{M}}$ on $\msf{m}^\perp_{E, x, r/2}$ preserving $\langle -, - \rangle_{E}$ such that $(\msf{m}^\perp_{E, x, r/2})^{\msf{Z}_{\msf{M}}} = 0$. 
\end{lem}
\begin{proof}
    There is a natural adjoint action $\cl{\breve{Z}}_{M, E} \circlearrowright \cl{U}_{\alpha, x, r/2}$ for each $\alpha \in \Phi$. It induces a scalar action of $\msf{Z}_{\msf{M}}$ on $\msf{u}_{\alpha, x, r/2}$ given by $\alpha$. If $\alpha \in \Phi^M$, the restriction of $\alpha$ to $\msf{Z}_{\msf{M}}$ is nontrivial. Thus, we get an action of $\msf{Z}_{\msf{M}}$ on $\msf{m}^\perp_{E, x, r/2}$ such that $(\msf{m}^\perp_{x, r/2})^{\msf{Z}_{\msf{M}}} = 0$. Since $\langle - , - \rangle_E$ is the commutator pairing associated to $H_{E, x, r}$, it is stable under the adjoint action of $\msf{Z}_{\msf{M}}$. 
\end{proof}

\begin{prop}
    The triple $(\msf{Z}_{\msf{M}}, \msf{m}^\perp_{E, x, r/2}, \Gamma)$ is a tame toral action. 
\end{prop}
\begin{proof}
    The natural semilinear $\Gamma$-action on positive loop groups induces one on the quotient $\msf{m}^\perp_{E, x, r/2}$, and it is smooth since positive loop groups can be defined over $k_E$. Moreover, since $M$ is defined over $F$, there is a natural semilinear smooth $\Gamma$-action on $\cl{\breve{Z}}_{M, E}$, and it induces one on $\msf{Z}_{\msf{M}}$. By construction, the adjoint $\msf{Z}_{\msf{M}}$-action on $\msf{m}^\perp_{E, x, r/2}$ and the commutator pairing $\langle -, - \rangle_E$ are compatible with the $\Gamma$-action. Since $\phi_E = \phi \circ \Av_\tau$, $\phi_E$ is equivariant under $\Gamma$, so the $\Gamma$-action is compatible with $\langle - , -\rangle_{E, \phi}$. Thus, $(\msf{Z}_{\msf{M}}, \msf{m}^\perp_{E, x, r/2}, \Gamma)$ is a tame toral action. 
\end{proof}

From now on, we make the following assumption. 

\begin{ass} \label{ass:without_symmetric_ramification}
    The tame toral action $(\msf{Z}_{\msf{M}}, \msf{m}^\perp_{E, x, r/2}, \Gamma)$ is \textit{without symmetric ramification}. By \Cref{lem:criterion_without_symmetric_ramification}, this holds if $e$ is odd. 
\end{ass}


\begin{defi} \label{defi:Phi_e_and_u_I}
    Let $\Phi_e^M$ be the set of $\tau$-orbits in the set of $\msf{Z}_{\msf{M}}$-weights of $\msf{m}^\perp_{E, x, r/2}$. For each subset $I \subset \Phi_e^M$, let $I_E \subset \Phi^M$ be the inverse image of $I$ along $\Phi^M \to \Phi_e^M$ and let 
    \[
        \msf{u}_{E, I, r/2} = \bigoplus_{\alpha \in I_E} \msf{u}_{\alpha, x, r/2} \subset \msf{m}^\perp_{E, x, r/2}, \quad 
        \msf{u}_{I, r/2} = \msf{u}_{E, I, r/2}^\tau \subset \msf{m}^\perp_{x, r/2, \ov{k}}. 
    \]
    Let $\pi_I \colon \msf{m}^\perp_{x, r/2, \ov{k}} \to \msf{u}_{I, r/2}$ denote the natural projection. When $I = \{ \alpha \}$ is a singleton, we use the subscript $\alpha$ instead of $\{\alpha\}$. 
\end{defi}
\begin{lem} \label{prop:lifting_to_Heisenberg}
    Let $I \subset \Phi_e^M$ be a subset such that $I \cap (-I) = \phi$. Then, there is a canonical group homomorphism $\msf{u}_{I, r/2} \to H_{x, r, \ov{k}}$ lifting $\msf{u}_{I, r/2} \subset \msf{m}^\perp_{x, r/2, \ov{k}}$. 
\end{lem}
\begin{proof}
    It is enough to construct a $\tau$-equivariant group homomorphism $\msf{u}_{E, I, r/2} \to H_{E, x, r}$ lifting $\msf{u}_{E, I, r/2} \subset \msf{m}^\perp_{E, x, r/2}$. For each $\alpha \in I_E$, $L^+\cl{U}_{\alpha, x, r/2} \to (L\breve{G}_E)_{x, r, r/2}$ induces $\msf{u}_{\alpha, x, r/2} \to H_{E, x, r/2}$ by \Cref{lem:piecewise_quotient}. Since $\msf{u}_{E, I}$ is isotropic under the commutator pairing $\langle -, - \rangle_E$, any two lifts commute. The resulting homomorphism $\msf{u}_{E, I, r/2} \to H_{E, x, r}$ is $\tau$-stable since the collection of lifts $\msf{u}_{\alpha, x, r/2} \to H_{E, x, r/2}$ is $\tau$-stable by construction. 
\end{proof}

For such an $I$, we will fix an inclusion $\msf{u}_{I, r/2} \subset H_{x, r, \ov{k}}$ as above. It also induces an inclusion $\msf{u}_{I, r/2} \subset H_{x, r, \phi, \ov{k}}$. Under \Cref{ass:without_symmetric_ramification}, we get $\msf{u}_{\alpha, r/2} \subset H_{x, r, \phi, \ov{k}}$ for every $\alpha \in \Phi_e^M$. 

To introduce Heisenberg-Weil groups, we will make additional assumptions on $\phi$. The adjoint action of $M(F)_x$ on $G(F)$ stabilizes $G(F)_{x, r, s}$ for $0 < \tfrac{r}{2} \leq s \leq r$. Thus, $M(F)_x$ acts on $\msf{m}_{x, r}$, $\msf{m}^\perp_{x, r/2}$ and $H_{x, r}$. 

\begin{defi} \label{defi:M_stable}
    We say that a $k$-linear map $\phi \colon \msf{m}_{x, r} \to k$ is $M$-stable if the natural adjoint action of $M(F)_x$ on $\msf{m}_{x, r}$ stabilizes $\phi$. 
\end{defi}

\begin{lem} \label{lem:adjoint_M_xrphi}
    Suppose that an $r$-generic map $\phi$ is $M$-stable. Then, the adjoint action of $M(F)_x$ on $\msf{m}^\perp_{x, r/2}$ preserves the symplectic form $\langle -, - \rangle_\phi$. Let 
    \[
        \msf{M}_{x, r, \phi} \subset \Sp(\msf{m}^\perp_{x, r/2}, \langle - , -\rangle_{\phi})
    \]
    be the image of $M(F)_x$. Then, the adjoint action of $M(F)_x$ on $H_{x, r, \phi}$ factors through $\msf{M}_{x, r, \phi}$. 
\end{lem}
\begin{proof}
    By construction, the commutator pairing $\langle -, -\rangle$ is equivariant under $M(F)_x$. Thus, $\langle -, - \rangle_\phi$ is stable under $M(F)_x$ if $\phi$ is $M$-stable. 

    Since the adjoint action of $M(F)_x$ on $\msf{m}^\perp_{E, x, r/2}$ preserves $\msf{Z}_{\msf{M}}$-weights, $\msf{u}_{\alpha, r/2} \subset H_{x, r, \phi, \ov{k}}$ is stable under $M(F)_x$ for each $\alpha \in \Phi_e^M$. Thus, an isomorphism 
    \[
        \bb{A}^1 \times \prod_{\alpha \in \Phi_e^M} \msf{u}_{\alpha, r/2} \cong H_{x, r, \phi, \ov{k}}
    \]
    for an arbitrary order of the product factors is equivariant under $M(F)_x$. In particular, the adjoint action of $M(F)_x$ on $H_{x, r, \phi}$ factors through $\msf{M}_{x, r, \phi}$. 
\end{proof}

\begin{defi}
    For an $M$-stable $r$-generic map $\phi \colon \msf{m}_{x, r} \to k$, we say that 
    \[  
        \msf{HW}_{x, r, \phi} = \msf{H}_{x, r, \phi} \rtimes \msf{M}_{x, r, \phi}
    \]
    associated to the adjoint action is the Heisenberg-Weil group associated to $\phi$. 
\end{defi}

From now on, we will assume that $\phi$ is $M$-stable and $r$-generic. First, we decompose $\msf{H}_{x, r, \phi}$ along an orthogonal decomposition of $\msf{m}^\perp_{x, r/2}$. 

\begin{defi}
    Let $\Phi_\Sigma^M$ be the set of orbits in $\Phi_e^M$ under the $\Gal(\ov{k} / k) \times \{ \pm 1\}$-action. For each $C \in \Phi_\Sigma^M$, let $\msf{m}^{\perp, C}_{x, r/2} \subset \msf{m}^{\perp}_{x, r/2}$ be as in \Cref{defi:C_component}. Let $H_C \subset H_{x, r, \phi}$ be the preimage of $\msf{m}^{\perp, C}_{x, r/2} \subset \msf{m}^{\perp}_{x, r/2}$ and let $\msf{H}_C = H_C(k)$. 
\end{defi}

\begin{lem} \label{lem:Heisenberg_decomposition}
    Let $\msf{Z}^0 = \{ (z_C) \in k^{\oplus \Phi_\Sigma^M} \mid \sum_C z_C = 0\}$. Then, we have an isomorphism 
    \[
        \msf{H}_{x, r, \phi} \cong \bigl(\prod_{C \in \Phi_\Sigma^M} \msf{H}_C\bigr) / \msf{Z}^0. 
    \]
    Note that we set the right-hand side to $k$ when $\Phi_\Sigma^M$  is empty. 
\end{lem}
\begin{proof}
    This is a general fact since $\msf{m}^\perp_{x, r, \phi} = \bigoplus_C \msf{m}^{\perp, C}_{x, r/2}$ is an orthogonal decomposition. 
\end{proof}

Since the adjoint action of $M(F)_x$ preserves $\msf{Z}_{\msf{M}}$-weights, the orthogonal decomposition $\msf{m}^\perp_{x, r, \phi} = \bigoplus_C \msf{m}^{\perp, C}_{x, r/2}$ is stable under $\msf{M}_{x, r, \phi}$. To give an explicit description of $\msf{H}_C \rtimes \msf{M}_{x, r, \phi}$, we first recall Heisenberg-Weil groups associated to polarized or unitary symplectic forms. Our version is a central quotient of usual Heisenberg-Weil groups along the trace function.

\begin{defi}
    For a finite extension $l / k$ and an $l$-vector space $V$, let 
    $
        \msf{H}(V) = k \times (V \oplus V^*)
    $
    be the Heisenberg group equipped with a group law
    \[
        (z, (a, a^*)) \cdot (w, (b, b^*)) = (z + w + \Tr_{l/k} \langle a, b^* \rangle, (a + b, a^* + b^*)). 
    \]
    Let $\msf{HW}(V) = \msf{H}(V) \rtimes \GL(V)$ be the Heisenberg-Weil group associated to the adjoint action 
    \[
        g \cdot (z, (a, a^*)) = (z, (ga, ga^*)). 
    \]
\end{defi}

\begin{defi}
    Let $l^+ / k$ be a finite extension and let $l / l^+$ be a quadratic extension. Let $(V, \langle - , - \rangle_V)$ be an $l$-vector space equipped with a non-degenerate skew-Hermitian form relative to $l/ l^+$. Let 
    \[
        \msf{H}(V) = \{(z, v) \in \ov{k} \times V \mid z^q - z = - \langle v, v \rangle_V \}
    \] 
    be the Heisenberg group equipped with a group law
    \[
        (z, a) \cdot (w, b) = (z + w + \sum_{0 \leq i < [l^+ : k]} \langle a, b \rangle_V^{q^i}, a + b). 
    \]
    Let $\msf{HW}(V) = \msf{H}(V) \rtimes \msf{U}(V)$ be the Heisenberg-Weil group associated to the adjoint action 
    $
        g \cdot (z, a) = (z, ga). 
    $
    Here, $\msf{U}(V)$ is the unitary group. 
\end{defi}

By \Cref{prop:Vtau_decomposition}, we may apply \Cref{prop:orbit_description} to $\msf{m}^\perp_{x, r/2}$ to get a description for each $\msf{m}^{\perp, C}_{x ,r/2}$ and $\msf{H}_C$. Now, recall the notation in \Cref{defi:rational_form}. 

\begin{prop} \label{prop:HW_decomposition}
    Let $C \in \Phi_\Sigma^M$ and let $\alpha \in C \subset \Phi_e^M$ be a representative. 
    \begin{enumerate}
        \item When $C \neq \Gal(\ov{k} / k) \cdot \alpha$, there is a canonical homomorphism
        \begin{equation} \label{eq:HW_polarized}
            \msf{H}_C \rtimes \msf{M}_{x, r, \phi} \to \msf{HW}(\msf{m}^{\perp, \alpha}_{x, r/2})
        \end{equation}
        that induces an isomorphism $\msf{H}_C \cong \msf{H}(\msf{m}^{\perp, \alpha}_{x, r/2})$ and sends $\msf{M}_{x, r, \phi}$ to $\GL(\msf{m}^{\perp, \alpha}_{x, r/2})$ via the natural action $\msf{M}_{x, r, \phi} \circlearrowright \msf{m}^{\perp, \alpha}_{x, r/2}$. 
        \item When $C = \Gal(\ov{k} / k) \cdot \alpha$, there is a canonical homomorphism
        \[
            \msf{H}_C \rtimes \msf{M}_{x, r, \phi} \to \msf{HW}(\msf{m}^{\perp, \alpha}_{x, r/2})
        \]
        that induces an isomorphism $\msf{H}_C \cong \msf{H}(\msf{m}^{\perp, \alpha}_{x, r/2})$ and sends $\msf{M}_{x, r, \phi}$ to $\msf{U}(\msf{m}^{\perp, \alpha}_{x, r/2})$ via the natural action $\msf{M}_{x, r, \phi} \circlearrowright \msf{m}^{\perp, \alpha}_{x, r/2}$ preserving the skew-Hermitian form $B$. 
    \end{enumerate}
\end{prop}
\begin{proof}
    For (1), let $C^+ = \Gal(\ov{k} / k) \cdot \alpha$ and $C^- = - C^+$. Then, we have
    \begin{equation} \label{eq:polarization_HC}
        \bb{A}^1 \times \msf{u}_{C^-, r/2} \times \msf{u}_{C^+, r/2} \cong H_C, \quad (z, a^*, a) \mapsto z \cdot a^* \cdot a. 
    \end{equation}
    Since $\msf{m}^{\perp, \alpha}_{x, r/2} \cong \msf{u}_{C^+, r/2}(k)$ is given by $v \mapsto \sum_{0 \leq i < d_\alpha} \sigma^i(v)$ and there is a natural identification $\msf{m}^{\perp, -\alpha}_{x, r/2} \cong (\msf{m}^{\perp, \alpha}_{x, r/2})^*$, \eqref{eq:polarization_HC} induces $\msf{H}_C \cong \msf{H}(\msf{m}^{\perp, x, \alpha}_{x, r/2})$. Moreover, \eqref{eq:polarization_HC} is equivariant under $\msf{M}_{x, r, \phi}$ and the action factors through $\msf{M}_{x, r, \phi} \to \GL(\msf{m}^{\perp, \alpha})$, so it extends to \eqref{eq:HW_polarized}.
    
    For (2), let 
    \[
        C^+ = \{ \sigma^i(\alpha) \mid 0 \leq i < d_\alpha / 2 \}, \quad 
        C^- =\{ \sigma^i(\alpha) \mid d_\alpha / 2 \leq i < d_\alpha \}. 
    \]
    Then we similarly have an $\msf{M}_{x, r, \phi}$-equivariant isomorphism
    \[
        \bb{A}^1 \times \msf{u}_{C^-, r/2} \times \msf{u}_{C^+, r/2} \cong H_{C, \ov{k}}, \quad (z, a^*, a) \mapsto z \cdot a^* \cdot a. 
    \]
    Then, the map $\msf{H}(\msf{m}^{\perp, \alpha}_{x, r/2}) \to H_C(\ov{k})$ given by 
    \[
        (z, a) \mapsto \biggl(z, \sum_{d_\alpha / 2 \leq i < d_\alpha} \sigma^i(a), \sum_{0 \leq i < d_\alpha / 2} \sigma^i(a)\biggr)
    \]
    provides $\msf{H}(\msf{m}^{\perp, x, \alpha}_{x, r/2}) \cong \msf{H}_C$ since 
    \[
        \sigma\biggl(z, \sum_{d_\alpha / 2 \leq i < d_\alpha} \sigma^i(a), \sum_{0 \leq i < d_\alpha / 2} \sigma^i(a)\biggr) 
        = \biggl(z^q + \langle a, \sigma^{d_s / 2}(a) \rangle, \sum_{d_\alpha / 2 \leq i < d_\alpha} \sigma^i(a), \sum_{0 \leq i < d_\alpha / 2} \sigma^i(a)\biggr). 
    \] 
    By construction, the adjoint action of $\msf{M}_{x, r, \phi}$ factors through $\msf{U}(\msf{m}^{\perp, \alpha}_{x, r/2})$, so we get the claim. 
\end{proof}

\subsection{Construction of Heisenberg-Weil representations}

In this section, we construct the Heisenberg-Weil representation of $\msf{HW}_{x, r, \phi}$ using the description in \Cref{prop:HW_decomposition}. 

\begin{defi}
    Let $\Phi_\Sigma^M = \Phi_{\Sigma, \asym}^M \sqcup \Phi_{\Sigma, \sym}^M$ be the decomposition such that $C \in \Phi_{\Sigma, \sym}^M$ if and only if $C = \Gal(\ov{k} / k) \cdot \alpha$ for a representative $\alpha \in C$. 
\end{defi}

\begin{defi} \label{defi:d_gamma_sign_factor}
    For each $\gamma \in  \msf{M}_{x, r, \phi}$, let 
    \[
        d_\gamma = \sum_{C \in \Phi_{\Sigma, \sym}^M} (\dim_{k_\alpha} \msf{m}^{\perp, \alpha}_{x, r/2} - \dim_{k_\alpha} \msf{m}^{\perp, \alpha, \gamma}_{x, r/2}). 
    \]
    Here, the superscript $\gamma$ denotes the $\gamma$-fixed locus and $\alpha \in C$ is a representative. 
\end{defi}

\begin{thm} \label{thm:Heisenberg_Weil_representations}
    There is a unique irreducible representation $\kappa_{x, r, \phi}$ of $\msf{HW}_{x, r, \phi}$ such that
    \begin{enumerate}
        \item $\kappa_{x, r, \phi} \vert_{\msf{H}_{x, r, \phi}}$ is the Heisenberg representation with central character $\psi$, and
        \item $\tr(\gamma \mid \kappa_{x, r, \phi}) = (-1)^{d_\gamma} \sqrt{\lvert \msf{m}^{\perp, \gamma}_{x, r/2} \rvert}$ for every $\gamma \in \msf{M}_{x, r, \phi}$. 
    \end{enumerate}
    Moreover, the trace function of $\kappa_{x, r, \phi}$ is supported on the elements conjugate to $k \cdot \msf{M}_{x, r, \phi}$. 
\end{thm}
\begin{proof}
    Let $C \in \Phi^M_\Sigma$ and $\alpha \in C$. If $C \in \Phi^M_{\Sigma, \asym}$, the Heisenberg-Weil representation for $\msf{m}^{\perp, \alpha}_{x, r / 2}$ with central character $\psi \circ \Tr_{k_\alpha / k}$, as in \cite[Proposition 1.4]{Ger77}, can be regarded as an irreducible representation $\kappa_C$ of $\msf{HW}(\msf{m}^{\perp, \alpha}_{x, r/2})$. In this case, the character of $\kappa_C$ is given by 
    \begin{equation} \label{eq:trace_polarized}
        \tr( \gamma \mid \kappa_C) = \lvert \msf{m}^{\perp, \alpha, \gamma}_{x, r/2} \rvert = \sqrt{\lvert \msf{m}^{\perp, C, \gamma}_{x, r/2} \rvert}
    \end{equation}
    for $\gamma \in \GL(\msf{m}^{\perp, \alpha}_{x, r/2})$ by \cite[Corollary 1.4]{Ger77}. 
    
    If $C \in \Phi^M_{\Sigma, \sym}$, the Heisenberg-Weil representation for $\msf{m}^{\perp, \alpha}_{x, r / 2}$ with central character $\psi \circ \Tr_{k_{\alpha / 2} / k}$, as in \cite[Theorem 3.3]{Ger77}, can be regarded as an irreducible representation $\kappa_C$ of $\msf{HW}(\msf{m}^{\perp, \alpha}_{x, r/2})$. In this case, the character of $\kappa_C$ is given by 
    \begin{equation} \label{eq:trace_unitary}
        \tr( \gamma \mid \kappa_C) = (-1)^{(\dim_{k_\alpha} \msf{m}^{\perp, \alpha}_{x, r/2} - \dim_{k_\alpha} \msf{m}^{\perp, \alpha, \gamma}_{x, r/2})} \sqrt{\lvert \msf{m}^{\perp, \alpha, \gamma}_{x, r/2} \rvert}
    \end{equation}
    for $\gamma \in \msf{U}(\msf{m}^{\perp, \alpha}_{x, r/2})$ by \cite[Theorem 4.9.2]{Ger77}. 

    Let $\kappa_{x, r, \phi} = \bigotimes_C \kappa_C$ and regard it as an irreducible representation of $(\prod_C \msf{H}_C) \rtimes \msf{M}_{x, r, \phi}$ by pulling back along \Cref{prop:HW_decomposition}. Since the central character of $\kappa_{x, r, \phi}$ is trivial on $\msf{Z}^0$, it descends to an irreducible representation of $\msf{HW}_{x, r, \phi}$ by \Cref{lem:Heisenberg_decomposition}. By construction, $\kappa_{x, r, \phi} \vert_{\msf{H}_{x, r, \phi}}$ is the Heisenberg representation with central character $\psi$, and $\tr(\gamma \mid \kappa_{x, r, \phi}) = (-1)^{d_\gamma} \sqrt{\lvert \msf{m}^{\perp, \gamma}_{x, r/2} \rvert}$ for every $\gamma \in \msf{M}_{x, r, \phi}$ by combining \eqref{eq:trace_polarized} and \eqref{eq:trace_unitary}. 

    The last claim is a general fact for Heisenberg-Weil representations and the uniqueness of $\kappa_{x, r, \phi}$ follows from it. For completeness, we recall an argument here. For each $\gamma \in \msf{M}_{x, r, \phi}$, let
    \[
        \msf{H}_{\gamma} = \{ h \in \msf{H}_{x, r, \phi} \mid h \cdot \gamma \; \text{is conjugate to} \; z\gamma \; \text{for some} \; z \in k\}. 
    \]
    Since $\kappa_{x, r, \phi}$ is irreducible, we have 
    \[
        \lvert \msf{HW}_{x ,r, \phi} \rvert = \sum_{h \in \msf{H}_{x, r, \phi}} \sum_{\gamma \in \msf{M}_{x ,r, \phi}} \lvert \tr(h \cdot \gamma \mid \kappa_{x, r, \phi}) \rvert^2 \geq \sum_{\gamma \in \msf{M}_{x, r, \phi}} \lvert H_\gamma \rvert \cdot \lvert \tr(\gamma \mid \kappa_{x, r, \phi}) \rvert^2
    \]
    For $h \in \msf{H}_{x, r, \phi}$, let $\ov{h} \in \msf{m}^\perp_{x, r/2}$ be the image of the natural projection. Then, the projection of $h \cdot \gamma h^{-1} \gamma^{-1}$ equals $(1- \gamma) \ov{h}$, so $\lvert H_\gamma \rvert \geq q \lvert \Img(\gamma-1 \mid \msf{m}^\perp_{x, r/2}) \rvert$. In particular, we have 
    \[
       \sum_{\gamma \in \msf{M}_{x, r, \phi}} \lvert H_\gamma \rvert \cdot \lvert \tr(\gamma \mid \kappa_{x, r, \phi}) \rvert^2 \geq \sum_{\gamma \in \msf{M}_{x, r, \phi}} \tfrac{\lvert \msf{m}^\perp_{x, r/2} \rvert}{\lvert \msf{m}^{\perp, \gamma}_{x, r/2} \rvert} \cdot q \lvert \msf{m}^{\perp, \gamma}_{x, r/2} \rvert = q \lvert \msf{M}_{x, r, \phi} \rvert \cdot \lvert \msf{m}^\perp_{x, r/2} \rvert. 
    \]
    Since $\lvert \msf{HW}_{x ,r, \phi} \rvert = q \lvert \msf{M}_{x, r, \phi} \rvert \cdot \lvert \msf{m}^\perp_{x, r/2} \rvert$, the above inequalities are all equalities. In particular, $\lvert H_\gamma \rvert = q \lvert \Img(\gamma-1 \mid \msf{m}^\perp_{x, r/2}) \rvert$ and $\tr(h \cdot \gamma \mid \kappa_{x, r, \phi}) = 0$ unless $h \in H_\gamma$. 
\end{proof}

\subsection{Comparison with symplectic Heisenberg-Weil representations}

In this section, we compare $\kappa_{x, r, \phi}$ with the symplectic Heisenberg-Weil representation $\kappa_\Yu$ originally used in \cite{Yu01}. In fact, they differ exactly by the quadratic character $\epsilon_{x}^{G / M} \colon M(F)_{x} \to \{\pm 1\}$ introduced in \cite{FKS23}. In this section, we will assume $p \neq 2$ and additionally assume \Cref{ass2:without_symmetric_ramification}. 

First, we review the description of $\epsilon_{x}^{G / M}$. A root $\alpha \in \Phi^M$ is symmetric if $ - \alpha \in \Gamma \cdot \alpha$, and otherwise $\alpha$ is asymmetric. A symmetric root $\alpha \in \Phi^M$ is ramified if $- \alpha \in \bb{Z} / e \bb{Z} \cdot \alpha$, and otherwise $\alpha$ is unramified. Under \Cref{ass2:without_symmetric_ramification}, there are no symmetric and ramified roots in $\Phi^M$, so the description in \cite[Definition 3.1, Theorem 3.4]{FKS23} becomes simple. 

For each $\alpha \in \Phi^M$, let $\Gamma_\alpha = \Stab_\Gamma(\alpha)$ and let $F_\alpha = E^{\Gamma_\alpha}$ be the defining field of $\alpha$ with a residue field $k_\alpha$. If $\alpha$ is symmetric and unramified, $[k_\alpha \colon k]$ is even and there is a unique subextension $k_{\alpha / 2} \subset k_\alpha$ with $[k_\alpha \colon k_{\alpha / 2}] = 2$. Let $k_\alpha^1 \subset k_\alpha^\times$ be the set of norm $1$ elements relative to $k_{\alpha} / k_{\alpha / 2}$. Since $p \neq 2$, there are unique nontrivial quadratic characters on $k_\alpha^\times$ and $k_\alpha^1$, which are denoted by $\sgn_{k_\alpha}$ and $\sgn_{k_\alpha^1}$. 

Let $\Phi^M = \Phi^M_\asym \sqcup \Phi^M_{\sym}$ be the decomposition into asymmetric roots and symmetric and unramified roots. For every bounded element $\gamma \in T(F)_b$, we have
\begin{equation} \label{eq:epsilon_factor}
    \epsilon_{x}^{G / M}(\gamma) = \prod_{\substack{[\alpha] \in \Phi^M_\asym / \Sigma \\ \msf{u}_{[\alpha], r / 2}^\Gamma \neq 0}} \sgn_{k_\alpha} \alpha(\gamma) \cdot 
    \prod_{\substack{[\alpha] \in \Phi^M_\sym / \Sigma \\ \msf{u}_{[\alpha], r / 2}^\Gamma \neq 0}} \sgn_{k^1_\alpha} \alpha(\gamma)
\end{equation}
Here, $[\alpha] \subset \Phi^M$ runs through $\Sigma$-orbits and
$
    \msf{u}_{[\alpha], r/2} = \bigoplus_{\alpha \in [\alpha]} \msf{u}_{\alpha, r/2}. 
$
Note that $\sgn_{k_\alpha} \alpha(\gamma)$ is independent of a representative $\alpha \in [\alpha]$ and it can be seen that $\alpha(\gamma) \in k_\alpha^1$ when $\alpha$ is symmetric. First, we give a different characterization of $\epsilon_{x}^{G / M}$. 

\begin{defi} \label{defi:MC}
    Let $C \in \Phi^M_{\Sigma}$ and let $\alpha \in C$ be a representative. If $C \in \Phi^M_{\Sigma, \asym}$, let 
    \[
        \msf{M}_C = \GL(\msf{m}^{\perp, \alpha}_{x, r/2}), \quad 
        \dt_C \colon \msf{M}_C \to k_{\alpha}^\times, \quad 
        \sgn_C = \sgn_{k_\alpha} \circ \dt_C
    \] 
    and if $C \in \Phi^M_{\Sigma, \sym}$, let 
    \[
        \msf{M}_C = \msf{U}(\msf{m}^{\perp, \alpha}_{x, r/2}), \quad 
        \dt_C \colon \msf{M}_C \to k_{\alpha}^1, \quad 
        \sgn_C = \sgn_{k_\alpha^1} \circ \dt_C. 
    \] 
    Let $\wtd{\msf{M}}_{x, r, \phi} = \prod_C \msf{M}_C$ and let $\sgn_{x, r, \phi} = \prod_C \sgn_C$ be a quadratic character of $\wtd{\msf{M}}_{x, r, \phi}$. 
\end{defi}
\begin{prop} \label{prop:interpretation_epsilon_twist}
    There is a natural inclusion
    \[
        \msf{M}_{x, r, \phi} \subset \msf{\wtd{M}}_{x, r, \phi} \subset \Sp(\msf{m}^\perp_{x, r/2}). 
    \]
    Moreover, $\epsilon_{x}^{G / M}$ equals the inflation of $\sgn_{x, r, \phi}$ to $M(F)_x$. 
\end{prop}
\begin{proof}
    By \Cref{prop:orbit_description}, the natural action $\msf{M}_C \circlearrowright \msf{m}^{\perp, C}_{x, r/2}$ preserves $\langle -, - \rangle_\phi$, so we have $\msf{\wtd{M}}_{x, r, \phi} \subset \Sp(\msf{m}^{\perp}_{x, r/2})$. As in \Cref{prop:HW_decomposition}, $\msf{M}_{x, r, \phi} \circlearrowright \msf{m}^{\perp, C}_{x, r/2}$ factors through $\msf{M}_C$, so we get $\msf{M}_{x, r, \phi} \subset \wtd{\msf{M}}_{x, r, \phi}$. To prove the second claim, it is enough to check
    \[
        \epsilon_{x}^{G / M}(\gamma) = \sgn_{x, r, \phi}(\gamma)
    \]
    for every tame maximal torus $T \subset M$ and $\gamma \in T(F)_b$ by \cite[Lemma 3.5]{FKS23}. Let $C \in \Phi_\Sigma^M$ with a representative $\alpha \in C$. We decompose $\sgn_C(\gamma)$ into product factors of \eqref{eq:epsilon_factor}. 

    First, suppose $C \in \Phi^M_{\Sigma, \asym}$ and let $\Gamma_\alpha = \Stab_\Gamma(\alpha)$. Then, we have
    \[
        \msf{m}^{\perp, \alpha}_{x, r/2} \cong \bigoplus_{[\beta] \in \alpha_E / \Gamma_\alpha} \msf{u}_{[\beta], r/2}^{\Gamma_\alpha}.  
    \]
    Here, recall the notation in \Cref{defi:Phi_e_and_u_I}. Take a representative $\beta \in [\beta]$ for each $[\beta]$. Since $\msf{u}_{[\beta], r/2}^{\Gamma_\alpha} \cong \msf{u}_{\beta, r/2}^{\Gamma_\beta}$, it is at most one-dimensional over $k_\beta$ and $\gamma$ acts by scalar $\beta(\gamma)$ under this isomorphism. Thus, we have
    \[
        \sgn_C(\gamma) = \prod_{\substack{[\beta] \in \alpha_E / \Gamma_\alpha \\ \msf{u}_{[\beta], r/2}^{\Gamma_\alpha} \neq 0}} \sgn_{k_\alpha} (\Nm_{k_\beta / k_\alpha} \beta(\gamma)) =  \prod_{\substack{[\beta] \in C_E / \Sigma \\ \msf{u}_{[\beta], r/2}^\Gamma \neq 0}} \sgn_{k_\beta} \beta(\gamma). 
    \]

    Next, suppose $C \in \Phi^M_{\Sigma, \sym}$ and let $\Gamma_\alpha = \Stab_\Gamma(\alpha)$. Then, we have
    \[
        \msf{m}^{\perp, \alpha}_{x, r/2} \cong \bigoplus_{[\beta] \in \alpha_E / \Gamma_\alpha} \msf{u}_{[\beta], r/2}^{\Gamma_\alpha}.  
    \]
    Take a representative $\beta \in [\beta]$ for each $[\beta]$. As previously, $\gamma$ acts on $\msf{u}_{[\beta], r/2}^{\Gamma_\alpha} \cong \msf{u}_{\beta, r/2}^{\Gamma_\beta}$ by scalar $\beta(\gamma)$. If $\beta \in \Phi^M_\sym$, then $\beta(\gamma) \in k_\beta^1$. Otherwise, there is another class $[\beta'] \in \alpha_E / \Gamma_\alpha$ such that $-\beta \in \Gamma \cdot \beta'$ and $\gamma$ acts on $\msf{u}_{\beta', r/2}^{\Gamma_{\beta'}} \cong \msf{u}_{-\beta, r/2}^{\Gamma_\beta}$ by scalar $\beta(\gamma)^{-1}$. Since $\msf{u}_{\beta', r/2}^{\Gamma_{\beta'}} \cong \msf{u}_{-\beta, r/2}^{\Gamma_\beta}$ is semilinear with respect to $k_\alpha / k_{\alpha / 2}$, we have 
    \[
        \dt_C(\gamma) =  \prod_{\substack{[\beta] \in (C_E \cap \Phi^M_{\sym}) / \Sigma \\ \msf{u}_{[\beta], r/2}^\Gamma \neq 0}} \Nm_{k_\beta / k_\alpha} \beta(\gamma) \cdot \prod_{\substack{[\beta] \in (C_E \cap \Phi^M_{\asym}) / \Sigma \\ \msf{u}_{[\beta], r/2}^\Gamma \neq 0}} (\Nm_{k_\beta / k_\alpha} \beta(\gamma))^{1 - q^{d_\alpha / 2}}. 
    \]
    When $\beta \in \Phi^M_\sym$, $k_\alpha \not \subset k_{\beta / 2}$ since $- \beta \not \in \alpha_E$. In particular, $[k_\beta \colon k_\alpha]$ is odd, so
    \[
        \sgn_{k_\alpha^1}(\Nm_{k_\beta / k_\alpha} \beta(\gamma)) = \sgn_{k_\beta^1} \beta(\gamma). 
    \]
    When $\beta \in \Phi^M_\asym$, $\sgn_{k^1_\alpha}(\Nm_{k_\beta / k_\alpha} \beta(\gamma))^{1 - q^{d_\alpha / 2}} = \sgn_{k_\beta} \beta(\gamma)$. As a result, we have
    \[
        \sgn_C(\gamma) =  \prod_{\substack{[\beta] \in (C_E \cap \Phi^M_{\sym}) / \Sigma \\ \msf{u}_{[\beta], r/2}^\Gamma \neq 0}} \sgn_{k_\beta^1} \beta(\gamma) \cdot \prod_{\substack{[\beta] \in (C_E \cap \Phi^M_{\asym}) / \Sigma \\ \msf{u}_{[\beta], r/2}^\Gamma \neq 0}} \sgn_{k_\beta} \beta(\gamma). 
    \]
    By taking a product over all $C \in \Phi^M_{\Sigma}$, we get $\epsilon_{x}^{G / M}(\gamma) = \sgn_{x, r, \phi}(\gamma)$. 
\end{proof}

Let $\kappa_\Yu$ be the Heisenberg-Weil representation of $\msf{H}_{x, r, \phi} \rtimes M(F)_x$ with central character $\psi$ that is originally used in \cite{Yu01}. It is constructed as the pullback of the symplectic Weil representation (see \cite[Section 2]{Ger77}) along a special group homomorphism 
\[
    \msf{H}_{x, r, \phi} \rtimes M(F)_x \to (\msf{m}^\perp_{x, r/2})^\sharp \rtimes \Sp(\msf{m}^\perp_{x, r/2})
\]
constructed in \cite[Proposition 11.4]{Yu01}. It induces a special isomorphism $\msf{H}_{x, r, \phi} \cong (\msf{m}^\perp_{x, r/2})^\sharp$ and sends $M(F)_x$ to $\Sp(\msf{m}^\perp_{x, r/2})$ via the natural conjugation action. In particular, it factors through $\msf{HW}_{x, r, \phi}$. By \Cref{prop:interpretation_epsilon_twist}, $\epsilon_{x}^{G / M}$ also factors through $\msf{M}_{x, r, \phi}$. 

\begin{prop} \label{prop:comparison_epsilon_twist}
    In the above setting, we have $\kappa_{x, r, \phi} \cong \kappa_\Yu \otimes \epsilon_{x}^{G / M}$. 
\end{prop}
\begin{proof}
    By \cite[Corollary 2.5]{Ger77}, it is enough to show that for every $C \in \Phi_\Sigma^M$, the polarized or unitary Heisenberg-Weil representation $\kappa_C$ equals the twist of the symplectic Heisenberg-Weil representation $\kappa_{\Yu, C}$ by $\sgn_C$. Since there is an explicit relation between $B$ and $\langle -, - \rangle_\phi$ in \Cref{prop:orbit_description}, this is proved in \cite[Theorem 2.4 (c)]{Ger77} when $C \in \Phi^M_{\Sigma, \asym}$, and in \cite[Theorem 3.3 (a'')]{Ger77} when $C \in \Phi^M_{\Sigma, \sym}$. 
\end{proof}

\subsection{Restriction to parabolic subgroups} \label{ssec:parabolic_rest}

In this section, we describe the restriction of $\kappa_{x, r, \phi}$ to a parabolic-like subgroup of $\msf{HW}_{x, r, \phi}$ as a parabolic-like induction. This is a counterpart of \cite[Theorem 2.4 (b)]{Ger77}, which originally contains a typo, for $\kappa_{x, r, \phi}$. 

Here, we fix a cocharacter $\lambda \colon \bb{G}_m \to T$. We will consider an associated parabolic-like subgroup of $\msf{HW}_{x, r, \phi}$. Take a decomposition $\Phi^M = \Phi^M_- \sqcup \Phi^M_0 \sqcup \Phi^M_+$ with respect to the sign of $\langle \alpha, \lambda \rangle$ for $\alpha \in \Phi^M$. As it is stable under $\Gal(E / F)$, it induces a weight decomposition 
\begin{equation} \label{eq:decomposition_lambda}
    \msf{m}^\perp_{x, r/2} = \msf{u}^- \oplus \msf{u}^0 \oplus \msf{u}^+.
\end{equation}
As in \Cref{prop:lifting_to_Heisenberg}, there are canonical liftings $\msf{u}^-, \msf{u}^+ \subset \msf{H}_{x, r, \phi}$. 

Let $G^0 \subset G$ be the centralizer of $\lambda$ and let $M^0 = M \cap G^0$. Our construction can be applied to $M^0 \subset G^0$ and we use the superscript $(-)^0$ to denote the counterparts for $M^0 \subset G^0$. As a result, we have a Heisenberg-Weil representation $\kappa^0_{x, r, \phi}$ of $\msf{HW}^0_{x, r, \phi}$, where
\[
    (\msf{m}^0)^\perp_{x, r/2} = \msf{u}^0 ,\quad 
    \msf{HW}^0_{x, r, \phi} = \msf{H}^0_{x, r, \phi} \rtimes \msf{M}^0_{x, r, \phi}. 
\]
Next, we take the associated parabolic subgroup
\[
    P = \{ m \in M \mid \lim_{t \to 0} \lambda(t) m \lambda(t)^{-1} \; \text{exists}\} ,\quad 
    \msf{P}_{x, r, \phi} = \Img(P(F)_x \to \Sp(\msf{m}^\perp_{x, r/2})). 
\]
The action of $\msf{P}_{x, r, \phi}$ stabilizes $\msf{u}^+$ and $\msf{u}^0 \oplus \msf{u}^+$ and it induces a quotient map $\msf{P}_{x, r, \phi} \twoheadrightarrow \msf{M}^0_{x, r, \phi} \subset \Sp(\msf{u}^0)$. Moreover, $U(F)_x \subset P(F)_x$ maps into the kernel of $\msf{P}_{x, r, \phi} \twoheadrightarrow \msf{M}^0_{x, r, \phi}$. Let
\[
    \msf{HW}^{\geq 0}_{x, r, \phi} = (\msf{H}^0_{x, r, \phi} \times \msf{u}^+) \rtimes \msf{P}_{x, r, \phi} \subset \msf{HW}_{x, r, \phi}. 
\]

\begin{prop} \label{prop:restrction_as_induction}
    The restriction of $\kappa_{x, r, \phi}$ to $\msf{H}_{x, r, \phi} \rtimes \msf{P}_{x, r, \phi}$ is isomorphic to 
    \[  
        \Ind^{\msf{H}_{x, r, \phi} \rtimes \msf{P}_{x, r, \phi}}_{\msf{HW}^{\geq 0}} \kappa^0_{x, r, \phi}. 
    \]
    Here, $\kappa^0_{x, r, \phi}$ is inflated via the natural quotient $\msf{HW}^{\geq 0}_{x, r, \phi} \to \msf{HW}^0_{x, r, \phi}$. 
\end{prop}
\begin{proof}
    Recall the construction $\kappa_{x, r, \phi} = \bigotimes_{C} \kappa_C$ in the proof of \Cref{thm:Heisenberg_Weil_representations}. For each $C \in \Phi_\Sigma^M$, \eqref{eq:decomposition_lambda} induces
    \[
        \msf{m}^{\perp, C}_{x, r/2} = \msf{u}_C^- \oplus \msf{u}_C^0 \oplus \msf{u}_C^+. 
    \]
    Let $\msf{P}_C \subset \msf{M}_C$ be the stabilizer of $\msf{u}_C^+ \subset \msf{m}^{\perp, C}_{x, r/2}$ and let $\msf{M}_C^0$ be its Levi subgroup. Let $\kappa_C^0$ be the Heisenberg-Weil representation for $\msf{H}_C \rtimes \msf{M}_C^0$ with central character $\psi$. We show
    \[
        \kappa_C \vert_{\msf{H}_C \rtimes \msf{P}_C} \cong \Ind^{\msf{H}_C \rtimes \msf{P}_C}_{(\msf{H}^0_C \times \msf{u}_C^+) \rtimes \msf{P}_C} \kappa_C^0. 
    \]
    Here, $\kappa_C^0$ is inflated via the quotient $(\msf{H}^0_C \times \msf{u}_C^+) \rtimes \msf{P}_C \to \msf{H}_C \rtimes \msf{M}_C^0$. 

    First, when $\msf{M}_C$ is a unitary group, the claim is proved in \cite[Theorem 3.3 (b)]{Ger77}. When $\msf{M}_C$ is a general linear group, the characters of both representations take only nonnegative values on $\msf{P}_C$. It is enough for the claim since both sides only differ by a character of $\msf{P}_C$. 

    Now, by taking a product over $C \in \Phi_\Sigma^M$, we get 
    \[
        \kappa_{x, r, \phi}\vert_{\msf{H}_{x, r, \phi} \rtimes \prod \msf{P}_C} \cong \Ind_{(\msf{H}^0_{x, r, \phi} \times \msf{u}^+) \rtimes \prod \msf{P}_C}^{\msf{H}_{x, r, \phi} \rtimes \prod \msf{P}_C} \bigotimes_C \kappa_C^0. 
    \]
    We get the claim by restricting along $\msf{P}_{x, r, \phi} \subset \prod \msf{P}_C \subset \Sp(\msf{m}^{\perp}_{x, r/2})$. 
\end{proof}

\begin{cor} \label{cor:U_fixed_locus}
    We have $\kappa_{x, r, \phi}^{\msf{u}^+} \cong \kappa_{x, r, \phi}^0$ as a representation of $\msf{HW}^{\geq 0}_{x, r, \phi}$. 
\end{cor}
\begin{proof}
    By \Cref{prop:restrction_as_induction}, $\kappa_{x, r, \phi} = \bigoplus_{u \in \msf{u}^-} u \cdot \kappa_{x, r, \phi}^0 $. An element $v \in \msf{u}^+$ acts on $u \cdot \kappa_{x, r, \phi}^0$ by scalar $\psi(\langle v, u \rangle_\phi)$. Since $\langle -, - \rangle_\phi$ is non-degenerate on $\msf{u}^+ \times \msf{u}^-$, we get $\kappa_{x, r, \phi}^{\msf{u}^+} \cong \kappa_{x, r, \phi}^0$. 
\end{proof}

\begin{rmk}
    The above properties hold in general when $p \neq 2$. In this case, we just set
    \[
        \kappa_{x, r, \phi} = \kappa_\Yu \otimes \epsilon_x^{G / M}. 
    \]
    It is consistent with the previous notation by \Cref{prop:comparison_epsilon_twist}. By \cite[Lemma 4.1.8]{FKS23} and the following lemma, \Cref{prop:restrction_as_induction} and \Cref{cor:U_fixed_locus} hold without any assumptions. 
\end{rmk}

\begin{lem}
    When $p \neq 2$, $\epsilon_x^{G_0 / M_0} \chi^{\msf{u}^+} = \epsilon_x^{G / M}$ as a character on $\msf{P}_{x, r, \phi}$. Here, $\chi^{\msf{u}^+}$ is the quadratic character defined as $\gamma \mapsto \sgn \circ \det(\gamma \vert_{\msf{u}^+})$ (see \cite[Lemma 4.1.8]{FKS23}). 
\end{lem}
\begin{proof}
    Since $p \neq 2$, any quadratic character on $\msf{P}_{x, r, \phi}$ factors through $\msf{M}^0_{x, r, \phi}$. Then, it is enough to check evaluations at $\gamma \in T(F)_b$ for every tame maximal torus $T \subset M^0$ by \cite[Lemma 3.5]{FKS23}. By \cite[Definition 3.1]{FKS23}, the difference between $\epsilon_x^{G_0 / M_0}(\gamma)$ and $\epsilon_x^{G / M}(\gamma)$ stems only from the asymmetric term of $\epsilon^{G / M}_{\sharp, x}$ associated to roots in $\msf{u}^+$. Then, the claim follows from the analysis in the asymmetric part of \Cref{prop:interpretation_epsilon_twist}. 
\end{proof}

\section{Twisted Yu's construction as successive inductions} \label{sec:Yu_construction}

In this section, we show that twisted Yu's construction works in residual characteristic $2$ by replacing symplectic Heisenberg-Weil representations $\kappa_\Yu$ with $\kappa_{x, r, \phi}$. Since \Cref{cor:U_fixed_locus} is essentially equivalent to the intertwining property \cite[Proposition 4.1.6]{FKS23}, the existing proofs in \cite{Yu01}, \cite{Fin21} and \cite{FKS23} work. Instead of repeating the same arguments, we directly show that the positive-depth induction step of Yu's construction preserves irreducibility. 

In this section, we do not impose \Cref{ass2:without_symmetric_ramification} and set $\kappa_{x, r, \phi} = \kappa_\Yu \otimes \epsilon^{G / M}_x$ when $p \neq 2$. This notation is consistent with the previous one by \Cref{prop:comparison_epsilon_twist}. 

\subsection{Positive-depth induction}

Here, we keep the notation in \Cref{sec:twisted_Heisenberg_Weil}. In this section, we introduce an induction functor from $M(F)_x$ to $G(F)_x$ when $Z_M / Z_G$ is anisotropic. It is an induction step in Yu's construction. First, let us recall some genericity conditions. We adopt the following terminology reflecting the analogy with \cite[{}10.3]{BR03}. 


\begin{defi} \textup{(\cite[Section 9]{Yu01}, \cite[Definition 4.1.1]{FS25})} \label{defi:genericity}
    \begin{enumerate}
        \item We say that $X \in (\Lie^*(M)^M)(F)$ is $(G, M)$-generic of depth $r$ if 
        \begin{quote}
        \begin{itemize}
            \item[{\bf GE0}] $X \in \Lie^*(M)_{-r} \backslash  \Lie^*(M)_{(-r)+}$, and
            \item[{\bf GE1}] $\nu(X(H_\alpha)) = -r$ for every $\alpha \in \Phi^M$, where $H_\alpha = d\alpha^\vee(1) \in \mfr{g}_E$. 
        \end{itemize}
        \end{quote}
        If $X$ additionally satisfies Yu's genericity condition {\bf GE2} (see \cite[Section 8]{Yu01} and \cite[Section 4.1]{FS25}), we say that $X$ is $(G, M)$-supergeneric of depth $r$. 
        \item A character $\phi \colon M(F)_x \to \Qlax$ is $(G, M)$-generic (resp.\ supergeneric) of depth $r$ if $\phi\vert_{M(F)_{x, r+}} = 1$ and $\phi\vert_{\msf{m}_{x, r}} = \psi \circ X$ for some $(G, M)$-generic (resp.\ supergeneric) element $X \in (\Lie^*(M)^M)(F)$ of depth $r$. 
    \end{enumerate}
\end{defi}

The notation $\kappa_{x, r, \phi}$ for the Heisenberg-Weil representation is justified by the following. 

\begin{lem} \label{lem:lift_X}
    Let $X\in (\Lie^*(M)^M)(F)$ be a $(G, M)$-generic element of depth $r$. Then, $X \vert_{\msf{m}_{x, r}}$ is $r$-generic and $M$-stable, and $\kappa_{x, r, X}$ only depends on $\psi \circ X$ as a representation of
    \[   
        (G(F)_{x, r, r/2} / G(F)_{x, r+ , r/2+}) \rtimes (M(F)_x / M(F)_{x, 0+}). 
    \]
\end{lem}
\begin{proof}
    Since $X \in \Lie^*(M)^M$, $X$ is clearly $M$-stable. We show that {\bf GE1} implies that $X$ is $r$-generic. For each $\alpha \in \Phi^M$ with $\msf{u}_{\alpha, r/2} \neq 0$, fix a nonzero element $du_\alpha \in \msf{u}_{\alpha, r/2}$. Then, the commutator pairing $[du_\alpha, du_{-\alpha}]$ in $\msf{m}_{E, x, r}$ equals $\pi_E^{er} H_\alpha$ up to a unit. Since $X\vert_{\msf{m}_{E, x, r}}$ is $\tau$-equivariant, $X\vert_{\msf{m}_{E, x, r}} = X\vert_{\msf{m}_{x, r}} \circ \Av_\tau$ and {\bf GE1} implies $\langle du_\alpha, du_{-\alpha} \rangle_{E, X} \neq 0$. Thus, $\{du_\alpha\}$ is a symplectic basis for $\langle -, - \rangle_{E, X}$, so in particular, $X$ is $r$-generic. The last claim is immediate from the characterization in \Cref{thm:Heisenberg_Weil_representations}. 
\end{proof}

For any $X$ as in \Cref{defi:genericity} (2), $\psi\circ X = \phi\vert_{\msf{m}_{x, r}}$ is independent of the choice of $X$, so $\kappa_{x, r, \phi} = \kappa_{x, r, X}$ is well-defined as a representation of the above semidirect product. Note that we take $\kappa_{x, r, \phi}$ to be the restriction of $\phi$ when $G = M$. 

From now on, we will assume that $Z_M / Z_G$ is anisotropic. It is necessary to have $M(F)_x \subset G(F)_x$ for the full stabilizers of $x$. The following induction functor is a representation-theoretic analogue of positive-depth Deligne-Lusztig induction (see \Cref{thm:positive_depth_DL_induction}). 

\begin{defi} \label{defi:twisted_positive_depth_induction}
    Let $\phi \colon M(F) \to \Qlax$ be a $(G, M)$-generic character of depth $r$ and let $\rho$ be an $M(F)_x / M(F)_{x, r}$-representation. If $Z_M / Z_G$ is anisotropic, we say that a representation
    \[
        \Ind_{r,\phi}(\rho) = \Ind_{G(F)_{x, r/2} M(F)_x}^{G(F)_x} (\kappa_{x, r, \phi} \otimes \rho). 
    \]
    is the positive-depth induction of $\rho$ twisted by $\phi$. Here, $M(F)_x$ and $G(F)_{x, r, r/2}$ act by 
    \[
        m(x \otimes v) = mx \otimes \phi(m) mv, \quad 
        g(x \otimes v) = gx \otimes v \quad 
        (x \in \kappa_{x, r, \phi}, \; v \in \rho)
    \]
    for $m \in M(F)_x$ and $g \in G(F)_{x, r, r / 2}$. In particular, $\Ind_{r,\phi}(\rho)\vert_{G(F)_{x, r+}}$ is trivial. 
\end{defi}

An important consequence of {\bf GE2} is the following. 

\begin{lem} \label{lem:GE2_intertwiner}
    Let $\hat{\phi} \colon G(F)_{x, r, r/2+} \to \Qlax$ be the character such that $\hat{\phi} \vert_{G(F)_{x, r+, r/2+}} = 1$ and $\hat{\phi}\vert_{M(F)_{x, r}} = \phi\vert_{M(F)_{x, r}}$. If $\phi$ is $(G, M)$-supergeneric and $g \in G(F)$ intertwines $\hat{\phi}$, then $g \in G(F)_{x, r/2} M(F) G(F)_{x, r/2}$. 
\end{lem}
\begin{proof}
    This is (part of) \cite[Theorem 9.4]{Yu01} (see e.g. \cite[Lemma 3.4]{Fin21}). Note that the same proof works for $p = 2$. 
\end{proof}


\begin{thm} \label{thm:intertwining_induction}
    Let $\rho$ be an irreducible representation of $M(F)_x / M(F)_{x, r}$ and suppose that $\phi$ is $(G, M)$-supergeneric. 
    \begin{enumerate}
        \item $\Ind_{r, \phi}(\rho)$ is an irreducible representation of $G(F)_x / G(F)_{x, r+}$. 
        \item If every intertwiner of $\rho$ in $M(F)$ lies in $M(F)_x$, then every intertwiner of $\Ind_{r, \phi}(\rho)$ in $G(F)$ lies in $G(F)_x$. 
    \end{enumerate}
    In particular, if $\cInd_{M(F)_x}^{M(F)} \rho$ is irreducible and supercuspidal, then so is $\cInd_{G(F)_x}^{G(F)} \Ind_{r, \phi}(\rho)$. 
\end{thm}
\begin{proof}
    For (1), it is enough to show that an intertwiner $g \in G(F)_x$ of $\kappa_{x, r, \phi} \otimes \rho$ lies in $G(F)_{x, r/2} M(F)_x$. Since $(\kappa_{x, r, \phi} \otimes \rho) \vert_{G(F)_{x, r, r/2+}} = \hat{\phi} \cdot \id$, $g$ intertwines $\hat{\phi}$. Thus, $g \in G(F)_{x, r/2} M(F) G(F)_{x, r/2}$ by \Cref{lem:GE2_intertwiner}. Since $g \in G(F)_x$, $g \in G(F)_{x, r/2} M(F)_x$.

    For (2), let $g \in G(F)$ be an intertwiner of $\kappa_{x, r, \phi} \otimes \rho$. It is enough to show that $g \in G(F)_{x, r/2} M(F)_x$. First, $g \in G(F)_{x, r/2} M(F) G(F)_{x, r/2}$ by \Cref{lem:GE2_intertwiner}, so we may assume $g \in M(F)$. We deduce a contradiction by assuming $gx \neq x$ in the reduced Bruhat-Tits building $\cl{B}(M, F)$. 

    Let $T \subset M$ be a tame maximal torus such that $x, gx \in \cl{A}(T, E)$ and take $\lambda \in X_*(T) \otimes \bb{R}$ so that $gx = x + \lambda$. Since $\lambda$ is $F$-rational, we may apply \Cref{ssec:parabolic_rest} to (a positive multiple of) $\lambda$. From now on, we will use the notation in loc. cit. 

    Let $K = G(F)_{x, r/2} M(F)_x \cap G(F)_{gx, r/2} M(F)_{gx}$ and take a nonzero intertwiner 
    \[
        f \colon {}^g(\kappa_{x, r, \phi} \otimes \rho) \vert_K \to (\kappa_{x, r, \phi} \otimes \rho)\vert_K. 
    \]
    Let $K_{r/2} = G(F)_{x, r, r/2} \cap G(F)_{gx, r, r/2}$ and $K_M = M(F)_x \cap M(F)_{gx}$. By taking the untwisted diagonal action of $K_M$, $\kappa_{x, r, \phi} \otimes \rho$ can be regarded as a representation of $K_{r/2} \rtimes K_M$. Since $g \in M(F)$, ${}^g \phi = \phi$, so $f$ is equivariant under $K_{r/2} \rtimes K_M$. Fix $X$ as in \Cref{defi:genericity}. When we use the subscript $(x, r, \phi)$ in the following, we use the choice of $X$ as in \Cref{lem:lift_X}. 
    
    Now, the image of $K_{r/2}$ under $G(F)_{x, r, r/2} \to \msf{H}_{x, r, \phi}$ is $\msf{H}^0_{x, r, \phi} \times \msf{u}^+$. However, $\msf{u}^+$ lies in the image of $G(F)_{x, r, r/2} \cap G(F)_{gx, r+, r/2+}$, so $f$ factors through
    \[
        (\kappa_{x, r, \phi} \otimes \rho)^{\msf{u}^+} \cong \kappa^0_{x, r, \phi} \otimes \rho
    \]
    by \Cref{cor:U_fixed_locus}. Here, $\kappa^0_{x, r, \phi} \otimes \rho$ is a representation of $\msf{H}^0_{x, r, \phi} \rtimes K_M$. Here, such a semidirect product can be taken since the image of $K_M$ under $M(F)_x \to \msf{M}_{x, r, \phi}$ lies in $\msf{P}_{x, r, \phi}$. 
    
    On the other hand, consider the image of ${}^{g^{-1}}K_{r/2} = G(F)_{x, r, r/2} \cap G(F)_{g^{-1}x, r, r/2}$ in $\msf{H}_{x, r, \phi}$. Since $x, g^{-1}x \in \cl{A}({}^{g^{-1}}T, E)$ and $g^{-1}x = x - g^{-1}\lambda$, we may apply \Cref{ssec:parabolic_rest} to ${}^{g^{-1}} T$ with respect to $- g^{-1}\lambda$. We denote the associated decomposition in loc. cit. by the superscript $0$ and the dash. Then, the image of ${}^{g^{-1}} K_{r/2}$ in $\msf{H}_{x, r, \phi}$ is written as $\msf{H}'^{0}_{x, r, \phi} \times \msf{u}'^{+}$. Since $\msf{u}'^+$ lies in the image of $G(F)_{x, r, r/2} \cap G(F)_{g^{-1}x, r+, r/2+}$, $f$ factors through the coinvariant 
    \[
        (\kappa_{x, r, \phi} \otimes \rho)_{\msf{u}'^+} \cong \kappa'^0_{x, r, \phi} \otimes \rho 
    \]
    by \Cref{cor:U_fixed_locus}. Here, $\kappa'^0_{x, r, \phi} \otimes \rho$ is a representation of $\msf{H}'^0_{x, r, \phi} \rtimes {}^{g^{-1}}K_M$. 
    
    Now, it is easy to see from the construction that
    ${}^g \msf{HW}'^0_{x, r, \phi} = \msf{HW}^0_{x, r, \phi}$
    as quotients of $K_{r/2} \rtimes K_M$ and ${}^g \kappa'^0_{x, r, \phi} \cong \kappa^0_{x, r, \phi}$ as irreducible representations of $\msf{HW}^0_{x, r, \phi}$ by the trace characterization in \Cref{thm:Heisenberg_Weil_representations} (and by \cite[Corollary 4.1.11]{FKS23} when $p \neq 2$). Since $f$ induces a $(K_{r/2} \rtimes K_M)$-equivariant nonzero map
    \[
        {}^g \kappa'^0_{x, r, \phi} \otimes {}^g \rho \to \kappa^0_{x, r, \phi} \otimes \rho,
    \]
    we get a nonzero $K_M$-equivariant map ${}^g \rho \to \rho$ by taking the $\kappa^0_{x, r, \phi}$-isotypic part. That is, the adjoint map is a nonzero $(K_{r/2} \rtimes K_M)$-equivariant map
    \[
        {}^g \rho \to (({}^g \kappa'^0_{x, r, \phi})^* \otimes \kappa^0_{x, r, \phi})^{K_{r/2}} \otimes \rho \cong \rho
    \]
    by Schur's lemma since $\kappa^0_{x, r, \phi}\vert_{K_{r/2}}$ is already irreducible. 
\end{proof}

The property ${}^g \kappa'^0_{x, r, \phi} \cong \kappa^0_{x, r, \phi}$ used in the above proof is exactly the intertwining property of Heisenberg-Weil representations in arbitrary residual characteristic. 

\begin{rmk} \label{rmk:behavior_without_GE2}
    When $\phi$ is $(G, M)$-generic but does not satisfy {\bf GE2}, $\Ind_{r, \phi}(\rho)$ is not necessarily irreducible due to the appearance of disconnected groups in the intertwiner of $\phi$ (see \cite[Lemma 4.6.3]{FS25}). Nevertheless, considering the geometric analogy with the Deligne-Lusztig theory \cite[Theorem 8.3]{DL76} (see \Cref{thm:positive_depth_DL_induction}), it would be interesting to find when $\cInd_{G(F)_x}^{G(F)} \Ind_{r, \phi}(\rho)$ is supercuspidal, or when $\cInd_{G(F)_x}^{G(F)} \sigma$ is irreducible and supercuspidal for each irreducible subrepresentation $\sigma$ of $\Ind_{r, \phi}(\rho)$. For example, \cite[Theorem 4.6.8]{FS25} seems to suggest that it is true if $\rho$ arises from relaxed Yu's construction (working without {\bf GE2}). 
\end{rmk}

\subsection{Twisted Yu's construction} \label{ssec:Yu_construction}

In this section, we show that Yu's construction works in residual characteristic $2$ by replacing symplectic Heisenberg-Weil representations $\kappa_\Yu$ with $\kappa_{x, r, \phi}$. In fact, it is obtained just as an iteration of \Cref{thm:intertwining_induction}. 

We follow \cite[Section 2]{Fin21} to review Yu's construction. Consider a tuple 
\[
    ((G_i)_{1 \leq i \leq n + 1}, x, (r_i)_{1\leq i \leq n}, \rho, (\phi)_{1 \leq i \leq n}),
\] 
called a generic datum, where
\begin{enumerate}
    \item $G_{n + 1} \subsetneq G_n\subsetneq \cdots \subsetneq G_2 \subset G_1 = G$ is a sequence of twisted Levi subgroups of $G$ that split over a tamely ramified extension of $F$. 
    \item $x \in \cl{B}(G_{n + 1}, F)$, 
    \item $r_1 > r_2 > \cdots > r_n > 0$ are real numbers, 
    \item $\rho$ is an irreducible representation of $G_{n + 1}(F)_x$ that is trivial on $G_{n + 1}(F)_{x, 0+}$, and
    \item $\phi_i$ is a character of $G_{i + 1}(F)$ of depth $r_i$ that is trivial on $G_{i + 1}(F)_{x, r+}$, 
\end{enumerate}
satisfying the following conditions:
\begin{enumerate}
    \item $Z_{G_{n + 1}} / Z_G$ is anisotropic, 
    \item The image of $x$ in $\cl{B}(G_{n + 1}^\ad, F)$ is a vertex, 
    \item $\rho\vert_{G_{n + 1}(F)_{x, 0}}$ is a cuspidal irreducible representation of $G_{n + 1}(F)_{x, 0} / G_{n + 1}(F)_{x, 0+}$, and 
    \item $\phi_i$ is $(G_i, G_{i+1})$-supergeneric of depth $r_i$ (relative to $x$). 
\end{enumerate}

For each $1 \leq i \leq n$, $G_i(F)_{x, r, s} \subset G_i(F)$ denotes the Moy-Prasad subgroup in the notation of \Cref{defi:MP_subgroup} applied to $G_{i + 1} \subset G_i$. Then, we set
\begin{equation} \label{eq:K_tilde}
    \wtd{K} = G_1(F)_{x, r_1, r_1 / 2} \cdots G_n(F)_{x, r_n, r_n / 2} G_{n+1}(F)_x. 
\end{equation}
We construct an irreducible representation of $\wtd{K}$ as $\wtd{\rho} = \rho \otimes \kappa$. Here, $\rho$ is inflated to $\wtd{K}$, so that it is trivial on the positive depth part $G_1(F)_{x, r_1, r_1 / 2} \cdots G_n(F)_{x, r_n, r_n / 2}$. We define the underlying space of $\kappa$ as 
\[
    V_\kappa = \bigotimes_{1 \leq i \leq n} \kappa_{x, r_i, \phi_i}. 
\]
To describe the $\wtd{K}$-action on $V_\kappa$, take the character $\hat{\phi}_i$ of $G(F)_{x, r_i/2+} G_{i + 1}(F)_{x}$ as in \cite[Lemma 4.2.1]{FS25} so that 
$
    \hat{\phi}_i\vert_{G_{i + 1}(F)_{x}} = \phi_i\vert_{G_{i + 1}(F)_{x}}
$
and $\hat{\phi}_i\vert_{G(F)_{x, r_i / 2 +}}$ is an inflation of $\phi_i\vert_{G_{i + 1}(F)_{x, r_i / 2 +}}$ via the Moy-Prasad isomorphism. 

\begin{lem} \label{lem:kappa_product}
    There is a unique smooth representation $(\kappa, V_\kappa)$ of $\wtd{K}$ such that 
    \begin{enumerate}
        \item $G_i(F)_{x, r_i, r_i / 2}$ acts on $\kappa_{x, r_i, \phi_i}$ by the natural action, and on $\kappa_{x, r_j, \phi_j}$ by $\hat{\phi}_j$ for $j \neq i$, and 
        \item $G_{n + 1}(F)_x$ acts on $\kappa_{x, r_i, \phi_i}$ by $\phi_i$ times the natural action via $G_{n + 1}(F)_x \subset G_{i + 1}(F)_x$. 
    \end{enumerate}
\end{lem}
\begin{proof}
    By \eqref{eq:K_tilde}, conditions (1) and (2) pin down the action of $\wtd{K}$ at most uniquely. It is easy to see that these actions are compatible on intersections and with conjugation. 
\end{proof}

\begin{thm} \label{thm:twisted_Yu_construction}
    For every generic datum $((G_i)_{1 \leq i \leq n + 1}, x, (r_i)_{1\leq i \leq n}, \rho, (\phi)_{1 \leq i \leq n})$, 
    $\cInd_{\wtd{K}}^{G(F)} \wtd{\rho}$ is irreducible and supercuspidal. 
\end{thm}
\begin{proof}
    By transitivity of inductions, it can be directly seen from the construction that 
    \[
        \cInd_{\wtd{K}}^{G(F)_x} \wtd{\rho} = \Ind_{r_1, \phi_1} \circ \cdots \circ \Ind_{r_n, \phi_n}(\rho). 
    \]
    By condition (3), $\cInd^{G_{n+1}(F)}_{G_{n+1}(F)_x} \rho$ is irreducible and supercuspidal. Thus, we may apply \Cref{thm:intertwining_induction} to $\rho$ and inductively see that 
    \[
        \cInd_{G_i(F)_x}^{G_i(F)} \Ind_{r_i, \phi_i} \circ \cdots \circ \Ind_{r_n, \phi_n}(\rho)
    \]
    is irreducible and supercuspidal. We recover the desired claim by setting $i = 1$.  
\end{proof}

\begin{rmk}
    As seen from the proof, one may start from any irreducible representation $\rho$ of $G_{n + 1}(F)_x$ such that $\cInd^{G_{n+1}(F)}_{G_{n+1}(F)_x} \rho$ is irreducible and supercuspidal, and $\rho\vert_{G_{n + 1}(F)_{x, r_{n+1}}}$ is trivial for some $0 \leq r_{n + 1} < r_n$. 
\end{rmk}

\section{Cohomology of certain Artin-Schreier sheaves} \label{sec:cohAS}

In this section, we compute the \'{e}tale cohomology of certain explicit Artin-Schreier sheaves on affine spaces. Then, we apply it to generic parts of the \'{e}tale cohomology of certain Lang torsors of affine spaces. 

\subsection{General formulation}

For each $d \geq 1$, let $k_d \subset \ov{k}$ be the subextension of degree $d$ over $k$. Let
\[
    \Tr_{k_d / k}\colon k_d \to k ,\quad t \mapsto \sum_{0 \leq i < d} t^{q^i} 
\]
be the trace map and let $\psi_{d} = \psi \circ \Tr_{k_d / k}$. Let $\bb{A}^n = \Spec(k[x_1,\ldots,x_n])$. For $d \geq 1$, let
\[
    L_{k_d} \colon \bb{A}^1 \to \bb{A}^1 ,\quad t \mapsto t^{q^d} - t
\]
be the Lang $k_d$-torsor. Here, $k_d$ acts from right by addition. The rank $1$ \'{e}tale local system 
\[
    \cl{L}_\psi = L_{k} \times^{\psi} \ov{\bb{Q}}_\ell
\]
over $\bb{A}^1$ is called the Artin-Schreier sheaf associated to $\psi$. Here, 
$
    \cl{L}_\psi \cong L_{k_d} \times^{\psi_d} \ov{\bb{Q}}_\ell
$
and the trace of the geometric Frobenius of $\cl{L}_\psi$ at $x \in \bb{A}^1(k_d)$ is $\psi_d(x)$. The main theorem of this section is the following. 

\begin{thm} \label{thm:cohAS}
    Let
    \[
        f_n = \sum_{1\leq i \leq n} a_i x_i x_{i+1}^{q^{d_i}} \in \ov{k}[x_1,\ldots,x_n] ,\quad a_i \in \ov{k}^\times ,\quad d_i \geq 1
    \]
    and regard it as a morphism $\bb{A}^n_{\ov{k}} \to \bb{A}^1_{\ov{k}}$. Here, $x_{n+1} = x_1$. Then, the following hold. 
    \begin{enumerate}
        \item $H_c^i(\bb{A}^n_{\ov{k}}, f_n^*\cl{L}_{\psi}) = 0$ if $i \neq n$, and $\dim H_c^n(\bb{A}^n_{\ov{k}}, f_n^*\cl{L}_{\psi}) = q^{\sum_{1\leq i \leq n} d_i}$.
        \item $ H_c^n(\bb{A}^n_{\ov{k}}, f_n^*\cl{L}_{\psi}) \to  H^n(\bb{A}^n_{\ov{k}}, f_n^*\cl{L}_{\psi})$ is an isomorphism. 
    \end{enumerate}
\end{thm}

\subsection{Computation of Gauss sums} \label{ssec:Gauss}

Take $d \geq 1$ so that $f_n \in k_d[x_1,\ldots,x_n]$. Then, $f_n$ defines $\bb{A}^n_{k_d} \to \bb{A}^1_{k_d}$ and $f_n^*\cl{L}_\psi$ is defined over $\bb{A}^n_{k_d}$. In this section, we compute the following Gauss sum. 

\begin{defi}
    Let $S_d(a_\bullet; d_\bullet) = \sum_{x_\bullet \in \bb{A}^n(k_d)} \psi_d(f_n(x_\bullet))$ for $d \geq 1$ such that $a_i \in k_d^\times$. 
\end{defi}

\begin{lem} \label{lem:redalg}
    Let $n \geq 2$ and $x_\bullet \in \bb{A}^{n-1}(k_d)$. Then, 
    \[
        \sum_{x_n \in k_d} \psi_d(a_{n-1}x_{n-1}x_n^{q^{d_{n-1}}} + a_nx_nx_1^{q^{d_n}}) = 0
    \]
    if $x_{n-1} \neq - a_{n-1}^{-1}a_n^{q^{d_{n-1}}}x_1^{q^{d_{n-1}+d_n}}$. Otherwise, the left-hand side is equal to $q^d$. 
\end{lem}
\begin{proof}
    Since the map 
    \[
        k_d \to \Qlax, \quad x_n \mapsto \psi_d(a_{n-1}x_{n-1}x_n^{q^{d_{n-1}}} + a_nx_nx_1^{q^{d_n}})
    \]
    is a character, it is enough to show that it is trivial if and only if $x_{n-1} = - a_{n-1}^{-1}a_n^{q^{d_{n-1}}}x_1^{q^{d_{n-1}+d_n}}$. This character is trivial if and only if
    \begin{equation} \label{eq:Trace_x_n-1}
        \Tr_{k_d / k}(a_{n-1}x_{n-1}x_n^{q^{d_{n-1}}} + a_nx_nx_1^{q^{d_n}}) = \sum_{0\leq i < d} (a_n^{q^i}x_1^{q^{i+d_n}} + (a_{n-1}x_{n-1})^{q^{i-d_{n-1}}})x_n^{q^i}
    \end{equation}
    lies in $\Ker(\psi)$ for every $x_n \in k_d$. Since the image of \eqref{eq:Trace_x_n-1} is $0$ or $k$ and $\psi$ is nontrivial, the right-hand side of \eqref{eq:Trace_x_n-1} should be $0$ for all $x_n \in k_d$. Since it is a polynomial of degree at most $q^{d-1}$ and $\lvert k_d \rvert > q^{d-1}$, the condition is equivalent to 
    \[
        a_n^{q^i}x_1^{q^{i+d_n}} + (a_{n-1}x_{n-1})^{q^{i-d_{n-1}}} = 0
    \]
    for every $i$. It is equivalent to $x_{n-1} = - a_{n-1}^{-1} a_n^{q^{d_{n-1}}} x_1^{q^{d_{n-1}+d_n}}$. 
\end{proof}

\begin{cor} \label{cor:redformula}
    If $n\geq 3$, we have 
    \[
        S_d(a_\bullet;d_\bullet) = q^d S_d(a_1,\ldots,a_{n-3}, -a_{n-2}a_{n-1}^{-q^{d_{n-2}}}a_n^{q^{d_{n-2}+d_{n-1}}}; d_1,\ldots,d_{n-3}, d_{n-2} + d_{n-1} + d_n).
    \] 
    If $n=2$, $S_d(a_\bullet; d_\bullet)$ is $q^d$ times the number of solutions $x_{1} = - a_{1}^{-1}a_2^{q^{d_{1}}}x_1^{q^{d_{1}+d_2}}$ in $x_1 \in k_d$. 
\end{cor}
\begin{proof}
    The claim follows from \Cref{lem:redalg} by substituting $x_{n-1} = - a_{n-1}^{-1}a_n^{q^{d_{n-1}}}x_1^{q^{d_{n-1}+d_n}}$ into $S_d(a_\bullet;d_\bullet)$. 
\end{proof}

\begin{cor} \label{cor:comparison_Hc_vs_H}
    If the natural map $ H_c^\bullet(\bb{A}^n_{\ov{k}}, f_n^*\cl{L}_{\psi}) \to  H^\bullet(\bb{A}^n_{\ov{k}}, f_n^*\cl{L}_{\psi})$ is an isomorphism, then \Cref{thm:cohAS} holds. 
\end{cor}
\begin{proof}
    By the Artin vanishing, $H_c^i(\bb{A}^n_{\ov{k}}, f_n^*\cl{L}_\psi)=0$ for $i<n$ and $H^i(\bb{A}^n_{\ov{k}}, f_n^*\cl{L}_\psi)=0$ for $i > n$. Thus, the assumption implies $H_c^i(\bb{A}^n_{\ov{k}}, f_n^*\cl{L}_\psi)=0$ for $i \neq n$. Now, the Grothendieck-Lefschetz trace formula implies that for every $t\geq 1$, we may write
    \[
        S_{dt}(a_\bullet; d_\bullet) = (-1)^n \sum_{\lambda_i\in I} \lambda_i^{t}
    \]
    where $I$ is the multiset of eigenvalues of $\Frob_{q^d}$ in $H_c^n(\bb{A}^n_{\ov{k}}, f_n^*\cl{L}_\psi)$. Since $H_c^n(\bb{A}^n_{\ov{k}}, f_n^*\cl{L}_{\psi}) \cong  H^n(\bb{A}^n_{\ov{k}}, f_n^*\cl{L}_{\psi})$, $\lvert \lambda_i \rvert = q^{nd/2}$ by the purity of the \'{e}tale cohomology. By the pigeonhole principle, for every $\varepsilon>0$, there are infinitely many $t\geq 1$ such that $\lvert 1-(q^{-nd/2}\lambda_i)^t \rvert < \varepsilon$ for every $\lambda_i \in I$. Thus, the claim follows from \Cref{cor:redformula} since 
    \[
        \dim H_c^n(\bb{A}^n_{\ov{k}}, f_n^*\cl{L}_\psi) = \lvert I \rvert = \limsup_{t \to \infty} q^{-ndt/2}\lvert S_{dt}(a_\bullet; d_\bullet) \rvert = q^{\sum_{1\leq i \leq n} d_i}. 
    \]
    Note that the last equality follows by induction on $n$. The case $n = 1$ follows from \cite[(8.11)]{Del74} (see also \cite[Lemma B.2]{Mie16}). 
\end{proof}

\subsection{Proof of \Cref{thm:cohAS}}
In this section, we complete the proof of \Cref{thm:cohAS} by verifying the assumption  $ H_c^\bullet(\bb{A}^n_{\ov{k}}, f_n^*\cl{L}_{\psi}) \cong  H^\bullet(\bb{A}^n_{\ov{k}}, f_n^*\cl{L}_{\psi})$ in \Cref{cor:comparison_Hc_vs_H}. When $n = 1$, the claim follows from \cite[(8.11)]{Del74} (see also \cite[Lemma B.2]{Mie16}). For $n \geq 2$, we prove by a geometric analogue of the induction in \Cref{lem:redalg}. 
Take $d \geq 1$ as in \Cref{ssec:Gauss}. 

\begin{lem} \label{lem:supppr!}
    Let $\pr_{n-1}\colon \bb{A}^n \to \bb{A}^{n-1}$ be the projection to the first $n-1$ coordinates. Then, $\pr_{n-1!} f_n^*\cl{L}_\psi$ is supported on the locus $x_{n-1} = - a_{n-1}^{-1}a_n^{q^{d_{n-1}}}x_1^{q^{d_{n-1}+d_n}}$. 
\end{lem}
\begin{proof}
    Since $\pr_{n-1!} f_n^*\cl{L}_\psi$ is defined over $\bb{A}^{n-1}_{k_d}$, it is enough to show that for every $t\geq 1$, the stalk $(\pr_{n-1!} f_n^*\cl{L}_\psi)_x$ vanishes for every $x\in \bb{A}^{n-1}(k_{dt})$ with $x_{n-1} \neq - a_{n-1}^{-1}a_n^{q^{d_{n-1}}}x_1^{q^{d_{n-1}+d_n}}$. By the proper base change, $(\pr_{n-1!} f_n^*\cl{L}_\psi)_x$ computes the cohomology of an Artin-Schreier sheaf associated to a one-variable polynomial
    \[
        f(t) = A t^{q^{d}} - B t + C \quad (A, B, C \in \ov{k})
    \]
    with $A \neq B^{q^{d_{n-1}}}$. Since the Artin-Schreier sheaf $\cl{L}_\psi$ is additive and the pullback of $\cl{L}_\psi$ along $t \mapsto t^{q^d} - t$ is trivial, we get 
    \[
        H_c^\bullet(\bb{A}^1, f^* \cl{L}_\psi) \cong H_c^\bullet(\bb{A}^1, ((A^{q^{-d}} - B)t)^* \cl{L}_\psi) = 0. 
    \]
    Thus, we get $(\pr_{n-1!} f_n^*\cl{L}_\psi)_x = 0$ when $x_{n-1} \neq - a_{n-1}^{-1}a_n^{q^{d_{n-1}}}x_1^{q^{d_{n-1}+d_n}}$. 
\end{proof}

\subsubsection{Induction steps for $n\geq 3$} \label{sssec:indge3}

In this section, we assume $n\geq 3$. We will show that the claim for $n$ follows from that for $n-2$. In this section, 
by abuse of notation, let
\[
    \bb{A}^n = \Spec(\ov{k}[x_1, \ldots, x_n])
\]
and every scheme is considered over $\ov{k}$. The following is essential in our induction. 

\begin{prop} \label{prop:computege3}
    Let $g = a_2x_2x_3^{q^{d_{2}}} + a_3x_3x_1^{q^{d_3}}$ with $a_2,a_3\in \ov{k}^\times$ and regard it as a morphism $\bb{A}^3 \to \bb{A}^1$. The map $\pr_{2!} g^{*}\cl{L}_\psi \to \pr_{2*} g^{*}\cl{L}_\psi$ on $\bb{A}^2$ is an isomorphism. 
\end{prop}
\begin{proof}
    Let $\pi_X \colon X \to \bb{A}^3$ be the pullback of the Lang torsor $L_{k}$ along $g$. Then, $X \subset \bb{A}^4$ is a closed subscheme defined by the equation $t^q - t = a_2x_2x_3^{q^{d_{2}}} + a_3x_3x_1^{q^{d_3}}$. Now, consider the following diagram. 
    \begin{center}
        \begin{tikzcd}
            & X \ar[ld, "\pi_X"'] \ar[d]  & \\
            \bb{A}^3 \ar[r, "j", hook] \ar[rd, "\pr_2"']& \bb{A}^2 \times \bb{P}^1 \ar[d, "\pr_1"] & \bb{A}^2 \times \{\infty\} \ar[l, "i"', hook'] \ar[dl, "\pr_1"] \\
            & \bb{A}^2 &
        \end{tikzcd}
    \end{center}
    The cone of $\pr_{2!} g^{*}\cl{L}_\psi \to \pr_{2*} g^{*}\cl{L}_\psi$ is $i^*j_*g^*\cl{L}_\psi$. Thus, it is enough to show 
    \[  
        i^*j_*g^*\cl{L}_\psi = 0. 
    \]  
    Let $y = x_3^{-1}$ be a coordinate of $\bb{P}^1 - \{0\}$ and let $X^\circ = X \cap \{x_3 \neq 0\}$. Since the defining equation of $X^\circ$ can be written as
    \[
        (y^{q^{d_2-1}}t)^q - y^{q^{d_2}-q^{d_2-1}}(y^{q^{d_2-1}}t) = a_2x_2 + a_3 y^{q^{d_2}-1}x_1^{q^{d_3}}
    \]
    over $\bb{A}^2 \times (\bb{P}^1 - \{0, \infty\})$, the closed subscheme $Y \subset \bb{A}^4$ defined by
    \[
        t^q -  y^{q^{d_2}-q^{d_2-1}} t = a_2x_2 + a_3 y^{q^{d_2}-1}x_1^{q^{d_3}}
    \]
    is finite over $\bb{A}^2 \times (\bb{P}^1 - \{0\})$ and its fiber over $\bb{A}^2 \times (\bb{P}^1 - \{0, \infty\})$ equals $X^\circ$. Let $\pi_Y \colon Y \to \bb{A}^2 \times (\bb{P}^1 - \{0\})$ be the finite map and let $Y_\infty = Y \times_{\bb{A}^2 \times (\bb{P}^1 - \{0\})} \{\infty\}$. Then, we have the following diagram. 
    \begin{center}
        \begin{tikzcd}
            X^\circ \ar[d, "\pi_{X^\circ}"] \ar[r, "j_{X^\circ}", hook] & Y \ar[d, "\pi_Y"] & Y_\infty \ar[l, "i_Y"', hook'] \ar[d, "\pi_{Y_\infty}"]\\
            \bb{A}^2 \times (\bb{P}^1 - \{0, \infty\}) \ar[r, "j^\circ" ,hook] & \bb{A}^2 \times (\bb{P}^1 - \{0\}) & \bb{A}^2 \times \{\infty\} \ar[l, "i"', hook']
        \end{tikzcd}
    \end{center}
    By the projection formula, we have
    \[
        i^*j^\circ_{*} (\pi_{X^\circ})_* \ov{\bb{Q}}_\ell \cong i^* (\pi_Y)_* (j_{X^\circ})_* \ov{\bb{Q}}_\ell \cong (\pi_{Y_\infty})_* i_Y^* (j_{X^\circ})_* \ov{\bb{Q}}_\ell. 
    \]
    Now, $Y$ is isomorphic to $\bb{A}^3$ with coordinates $(x_1,y,t)$, so $Y_\infty$ is a normal crossing prime divisor of $Y$. Moreover, $\pi_{Y_\infty}$ sends $(x_1,t)$ to $(x_1,a_2^{-1}t^q)$, so it is a universal homeomorphism. Thus, the split monomorphism
    \[
        i^*j^\circ_* \ov{\bb{Q}}_\ell \to i^*j^\circ_{*} (\pi_{X^\circ})_* \ov{\bb{Q}}_\ell \cong (\pi_{Y_\infty})_* i_Y^* (j_{X^\circ})_* \ov{\bb{Q}}_\ell 
    \]
    is an isomorphism by \cite[(1.3.3.2) p.255]{SGA41/2}. Since $(g^* \cl{L}_\psi)\vert_{X^\circ}$ is a direct summand of $(\pi_{X^\circ})_* \ov{\bb{Q}}_\ell$ orthogonal to $\ov{\bb{Q}}_\ell$, we have $i^*j^\circ_* (g^* \cl{L}_\psi)\vert_{X^\circ}=0$. In particular, $i^*j_*g^*\cl{L}_\psi = 0$. 
\end{proof}

Let $\pr^3\colon \bb{A}^n \to \bb{A}^3$ be the projection sending $(x_1,\ldots,x_n)$ to $(x_1,x_{n-1}, x_n)$. Let $\pr^2 \colon \bb{A}^{n-1} \to \bb{A}^2$ be the projection sending $(x_1,\ldots,x_{n-1})$ to $(x_1,x_{n-1})$. Then, we have the following Cartesian square. 
\begin{center}
    \begin{tikzcd}
        \bb{A}^n \ar[r, "\pr^3"] \ar[d, "\pr_{n-1}"'] & \bb{A}^3 \ar[d, "\pr_2"] \\
        \bb{A}^{n-1} \ar[r, "\pr^2"] & \bb{A}^2
    \end{tikzcd}
\end{center}
Let $g_n = a_{n-1}x_{n-1}x_n^{q^{d_{n-1}}} + a_nx_nx_1^{q^{d_n}}$ and $h_{n-1} = f_n - g_n$. We regard $g_n$ (resp.\ $h_{n-1}$) as a morphism $\bb{A}^3 \to \bb{A}^1$ (resp.\ $\bb{A}^{n-1} \to \bb{A}^1$). By the additivity of Artin-Schreier sheaves,
\[
    f_n^*\cl{L}_\psi \cong \pr^{3*}g_n^{*}\cl{L}_\psi \otimes \pr_{n-1}^*h_{n-1}^*\cl{L}_\psi. 
\]
Since $h_{n-1}^*\cl{L}_\psi$ is invertible and $\pr^2$ is smooth, we have
\[
    \pr_{n-1*} f_n^*\cl{L}_\psi \cong h_{n-1}^*\cl{L}_\psi \otimes \pr_{n-1*} \pr^{3*}g_n^{*}\cl{L}_\psi \cong h_{n-1}^*\cl{L}_\psi \otimes \pr^{2*}\pr_{2*} g_n^{*}\cl{L}_\psi
\]
by the smooth base change. Similarly, we have
\[
    \pr_{n-1!} f_n^*\cl{L}_\psi \cong h_{n-1}^*\cl{L}_\psi \otimes \pr_{n-1!} \pr^{3*}g_n^{*}\cl{L}_\psi \cong h_{n-1}^*\cl{L}_\psi \otimes \pr^{2*}\pr_{2!} g_n^{*}\cl{L}_\psi
\]
by the projection formula and the proper base change. These isomorphisms are compatible so that \Cref{prop:computege3} implies that $\pr_{n-1!} f_n^*\cl{L}_\psi \to \pr_{n-1*} f_n^*\cl{L}_\psi$ is an isomorphism. 

Let $s\colon \bb{A}^{n-2} \to \bb{A}^{n-1}$ be the section of $\pr_{n-2}$ with $x_{n-1} = - a_{n-1}^{-1}a_n^{q^{d_{n-1}}}x_1^{q^{d_{n-1}+d_n}}$. By \Cref{lem:supppr!}, $\pr_{n-1!} f_n^*\cl{L}_\psi \cong \pr_{n-1*} f_n^*\cl{L}_\psi$ is supported on the image of $s$. Then, 
\[
    (\pr_{n-2}\circ \pr_{n-1})_!f_n^*\cl{L}_\psi \cong s^*\pr_{n-1!} f_n^*\cl{L}_\psi \cong (\pr_{n-2}\circ \pr_{n-1})_*f_n^*\cl{L}_\psi. 
\]
Moreover, 
\[
     s^*\pr_{n-1!} f_n^*\cl{L}_\psi \cong (h_{n-1}\circ s)^*\cl{L}_\psi \otimes (\pr^{2}\circ s)^*\pr_{2!} g_n^{*}\cl{L}_\psi
\]
We will compute $(\pr^{2}\circ s)^*\pr_{2!} g_n^{*}\cl{L}_\psi$. Let $\pr^1 \colon \bb{A}^{n-2} \to \bb{A}^1$ be the projection to the coordinate $x_1$. Let $s_1 \colon \bb{A}^1 \to \bb{A}^2$ (resp.\ $s_2 \colon \bb{A}^2 \to \bb{A}^3$) be a section of $\pr_1$ (resp.\ $\pr_{1,3}$) with $x_2 = - a_{n-1}^{-1}a_n^{q^{d_{n-1}}}x_1^{q^{d_{n-1}+d_n}}$. We have commutative diagrams
\begin{center}
    \begin{tikzcd}
        \bb{A}^{n-2} \ar[r, "s"] \ar[d, "\pr^1"] & \bb{A}^{n-1} \ar[d, "\pr^2"] &  & \bb{A}^2 \ar[r, "s_2"] \ar[d, "\pr_1"] & \bb{A}^3 \ar[d, "\pr_2"] \\
        \bb{A}^{1} \ar[r, "s_1"] & \bb{A}^2, &  & \bb{A}^1 \ar[r, "s_1"] & \bb{A}^2. 
    \end{tikzcd}
\end{center}
Thus, 
\[
    (\pr^{2}\circ s)^*\pr_{2!} g_n^{*}\cl{L}_\psi \cong \pr^{1*} s_1^*\pr_{2!} g_n^{*}\cl{L}_\psi \cong \pr^{1*} \pr_{1!} (g_n\circ s_2)^{*}\cl{L}_\psi 
\]
Now, $g_n \circ s_2$ corresponds to a polynomial
\[
    a_{n-1} (- a_{n-1}^{-1}a_n^{q^{d_{n-1}}}x_1^{q^{d_{n-1}+d_n}}) x_2^{q^{d_{n-1}}} + a_nx_2x_1^{q^{d_n}} = a_nx_2x_1^{q^{d_n}} - (a_nx_2x_1^{q^{d_n}})^{q^{d_{n-1}}}, 
\]
so it factors through the Lang $\bb{F}_{q^{d_{n-1}}}$-torsor over $\bb{A}^1$. Thus, $(g_n\circ s_2)^{*}\cl{L}_\psi \cong \ov{\bb{Q}}_\ell$ and
\[
    (\pr^{2}\circ s)^*\pr_{2!} g_n^{*}\cl{L}_\psi \cong \pr^{1*} \pr_{1!}  \ov{\bb{Q}}_\ell \cong \ov{\bb{Q}}_\ell(-1)[-2]. 
\]
Summarizing, we get the following commutative diagram. 
\begin{center}
    \begin{tikzcd}
        H_c^\bullet(\bb{A}^n, f_n^*\cl{L}_\psi) \ar[r] \ar[d, "\cong"] & H^\bullet(\bb{A}^n, f_n^*\cl{L}_\psi) \ar[d, "\cong"] \\
        H_c^\bullet(\bb{A}^{n-2}, (\pr_{n-2}\circ \pr_{n-1})_!f_n^*\cl{L}_\psi) \ar[r] \ar[d, "\cong"] & H^\bullet(\bb{A}^{n-2}, (\pr_{n-2}\circ \pr_{n-1})_*f_n^*\cl{L}_\psi) \ar[d, "\cong"] \\
        H_c^{\bullet-2}(\bb{A}^{n-2}, (h_{n-1}\circ s)^*\cl{L}_\psi)(-1) \ar[r] & H^{\bullet-2}(\bb{A}^{n-2}, (h_{n-1}\circ s)^*\cl{L}_\psi)(-1) 
    \end{tikzcd}
\end{center}
Now, $h_{n-1}\circ s$ corresponds to a polynomial in $n-2$ variables associated to the parameters 
\[
    (a_1,\ldots,a_{n-3}, -a_{n-2}a_{n-1}^{-q^{d_{n-2}}}a_n^{q^{d_{n-2}+d_{n-1}}}; d_1,\ldots,d_{n-3}, d_{n-2} + d_{n-1} + d_n)
\]
(cf. \Cref{cor:redformula}). Thus, the claim follows from the induction hypothesis for $n-2$. 

\subsubsection{Base step for $n=2$}

In this section, we will show the case $n=2$. Our argument goes along \Cref{sssec:indge3} with a slight modification.

\begin{prop} \label{prop:compute=2}
    Let $f_2 = a_1x_1x_2^{q^{d_{1}}} + a_2x_2x_1^{q^{d_2}}$ with $a_1,a_2\in \ov{k}^\times$ and regard it as a morphism $\bb{A}^2 \to \bb{A}^1$. Then, $\pr_{1!} f_2^{*}\cl{L}_\psi \to \pr_{1*} f_2^{*}\cl{L}_\psi$ is an isomorphism. 
\end{prop}
\begin{proof}
    Let $\pi_X \colon X \to \bb{A}^2$ be the pullback of the Lang torsor $L_{k}$ along $f_2$. Then, $X \subset \bb{A}^3$ is a closed subscheme defined by the equation $t^q - t = a_1x_1x_2^{q^{d_{1}}} + a_2x_2x_1^{q^{d_2}}$. By replacing $t$ with $t + \sum_{0 \leq i < d_1} (a_2x_2x_1^{q^{d_2}})^{q^i}$, the equation can be rewritten as $t^q - t = x_2^{q^{d_{1}}}(a_1x_1 + a_2^{q^{d_1}}x_1^{q^{d_1+d_2}})$. Now, consider the following diagram. 
    \begin{center}
        \begin{tikzcd}
            & X \ar[ld, "\pi_X"'] \ar[d]  & \\
            \bb{A}^2 \ar[r, "j", hook] \ar[rd, "\pr_1"']& \bb{A}^1 \times \bb{P}^1 \ar[d, "\pr_1"] & \bb{A}^1 \times \{\infty\} \ar[l, "i"', hook'] \ar[dl, "\pr_1"] \\
            & \bb{A}^1 &
        \end{tikzcd}
    \end{center}
    The cone of $\pr_{1!} f_2^{*}\cl{L}_\psi \to \pr_{1*} f_2^{*}\cl{L}_\psi$ is $i^*j_*f_2^*\cl{L}_\psi$. Thus, it is enough to show 
    \[
        i^*j_*f_2^*\cl{L}_\psi = 0. 
    \]
    Let $y = x_2^{-1}$ be a coordinate of $\bb{P}^1 - \{0\}$ and let $X^\circ = X \cap \{x_2 \neq 0\}$. Since the defining equation of $X^\circ$ can be written as
    \[
        (y^{q^{d_1-1}}t)^q - y^{q^{d_1}-q^{d_1-1}}(y^{q^{d_1-1}}t) = a_1x_1 + a_2^{q^{d_1}}x_1^{q^{d_1+d_2}}
    \]
    over $\bb{A}^1 \times (\bb{P}^1 - \{0, \infty\})$, the closed subscheme $Y \subset \bb{A}^3$ defined by
    \[
        t^q -  y^{q^{d_1}-q^{d_1-1}} t = a_1x_1 + a_2^{q^{d_1}}x_1^{q^{d_1+d_2}}
    \]
    is finite over $\bb{A}^1 \times (\bb{P}^1 - \{0\})$ and its fiber over $\bb{A}^1 \times (\bb{P}^1 - \{0, \infty\})$ equals $X^\circ$. Let $\pi_Y \colon Y \to \bb{A}^1 \times (\bb{P}^1 - \{0\})$ be the finite map and let $Y_\infty = Y \times_{\bb{A}^1 \times (\bb{P}^1 - \{0\})} \{\infty\}$. Then, we have the following diagram. 
    \begin{center}
        \begin{tikzcd}
            X^\circ \ar[d, "\pi_{X^\circ}"] \ar[r, "j_{X^\circ}", hook] & Y \ar[d, "\pi_Y"] & Y_\infty \ar[l, "i_Y"', hook'] \ar[d, "\pi_{Y_\infty}"]\\
            \bb{A}^1 \times (\bb{P}^1 - \{0, \infty\}) \ar[r, "j^\circ" ,hook] & \bb{A}^1 \times (\bb{P}^1 - \{0\}) & \bb{A}^1 \times \{\infty\} \ar[l, "i"', hook']
        \end{tikzcd}
    \end{center}
    By the projection formula, we have
    \[
        i^*j^\circ_{*} (\pi_{X^\circ})_* \ov{\bb{Q}}_\ell \cong i^* (\pi_Y)_* (j_{X^\circ})_* \ov{\bb{Q}}_\ell \cong (\pi_{Y_\infty})_* i_Y^* (j_{X^\circ})_* \ov{\bb{Q}}_\ell. 
    \]
    Now, the defining equation of $Y$ is separable with respect to $x_1$, so $Y \to \bb{A}^2$ sending $(x_1,y,t)$ to $(y,t)$ is finite \'{e}tale. In particular, $Y$ is smooth and $Y_\infty$ is a normal crossing prime divisor of $Y$. Moreover, the defining equation of $Y_\infty$ is $t^q =  a_1x_1 + a_2^{q^{d_1}}x_1^{q^{d_1+d_2}}$, so $\pi_{Y_\infty}$ is a universal homeomorphism. Thus, the split monomorphism
    \[
        i^*j^\circ_* \ov{\bb{Q}}_\ell \to i^*j^\circ_{*} (\pi_{X^\circ})_* \ov{\bb{Q}}_\ell \cong (\pi_{Y_\infty})_* i_Y^* (j_{X^\circ})_* \ov{\bb{Q}}_\ell 
    \]
    is an isomorphism by \cite[(1.3.3.2) p.255]{SGA41/2}. Since $(f_2^* \cl{L}_\psi)\vert_{\bb{A}^1 \times (\bb{P}^1 - \{0, \infty\})}$ is a direct summand of $(\pi_{X^\circ})_* \ov{\bb{Q}}_\ell$ orthogonal to $\ov{\bb{Q}}_\ell$, we have $i^*j^\circ_* (f_2^* \cl{L}_\psi)\vert_{\bb{A}^1 \times (\bb{P}^1 - \{0, \infty\})}=0$. In particular, $i^*j_*f_2^*\cl{L}_\psi = 0$. 
\end{proof}

By \Cref{lem:supppr!}, the support of $\pr_{1!} f_2^{*}\cl{L}_\psi$ is zero-dimensional. Thus, the natural map 
\[
    H_c^\bullet(\bb{A}^1, \pr_{1!} f_2^{*}\cl{L}_\psi)  \to H^\bullet(\bb{A}^1, \pr_{1!} f_2^{*}\cl{L}_\psi) \to H^\bullet(\bb{A}^1, \pr_{1*} f_2^{*}\cl{L}_\psi) 
\]
is an isomorphism. Thus, we get the claim.

\subsection{Lang torsors attached to polarization} \label{ssec:Lang_torsor}

Let $\msf{M}$ be a finite group and let $(V, \langle - , -\rangle)$ be a symplectic $\msf{M}$-representation over $k$ equipped with a genuine weight decomposition indexed by $S$ (see \Cref{ssec:weight_decomposition}). In this section, we introduce a Lang torsor over an affine space attached to a polarization of $S$, which plays an important role in the geometric realization of Heisenberg-Weil representations. 

\begin{defi}
    Let $S$ be a set equipped with a genuine smooth action of $\Gal(\ov{k} / k) \times \{\pm 1\}$. We say that a decomposition $S = S^+ \sqcup S^-$ is a \textit{polarization} if $S^+ = - S^-$. 
\end{defi}

From now on, we fix a polarization $S = S^+ \sqcup S^-$. 

\begin{defi}
    Let $I_0 = \{ s \in S^+ \mid \Gal(\ov{k} / k) \cdot s \subset S^+ \}$ and $I_1 = S^+ \cap \sigma(S^-)$. We say that $I = I_0 \sqcup I_1$ is the \textit{characteristic index set} of the polarization $S = S^+ \sqcup S^-$.  
\end{defi}

For a vector space $U$ over $\ov{k}$, we denote the associated algebraic variety over $\ov{k}$ by the sans-serif type $\msf{U}$. This notation also applies to $\ov{V}$. 

\begin{defi} \label{defi:fin_et_torsor}
    Let $U_I = \bigoplus_{s\in I} \ov{V}^s$ and let $\msf{\wtd{U}}_I \to \msf{U}_I$ be the finite \'{e}tale $V$-torsor defined by the following Cartesian diagram
    \begin{center}
        \begin{tikzcd}
            \msf{\wtd{U}}_I \ar[r] \ar[d] & \ov{\msf{V}} \ar[d, "\sigma - \id"]\\
            \msf{U}_I \ar[r] & \ov{\msf{V}}. 
        \end{tikzcd}
    \end{center}
\end{defi}

We study the geometry of $\msf{\wtd{U}}_I$ and apply \Cref{thm:cohAS} to compute the cohomology of an Artin-Schreier local system on it. First, we take a suitable basis of $\ov{V}$. Recall the notation in \Cref{ssec:weight_decomposition}. 

\begin{lem} \label{lem:suitable_basis}
    We can take a basis $\{v_\lambda\}_{\lambda \in \Lambda}$ of $\ov{V}$ so that the following hold. 
    \begin{enumerate}
        \item For every $\lambda \in \Lambda$, there is a unique $s_\lambda \in S$ such that $v_\lambda \in V^{s_\lambda}$
        \item For every $\lambda \in \Lambda$, there is a unique $\sigma(\lambda) \in \Lambda$ such that $\sigma(v_\lambda) \in \ov{k}^\times \cdot v_{\sigma(\lambda)}$. Moreover, $\sigma(v_\lambda) = v_{\sigma(\lambda)}$ for every $\lambda \in \Lambda$ such that $s_\lambda \in I_0 \sqcup (-I_0)$. 
        \item For every $s \in S$, let $V_s = \{ v_\lambda \mid s_\lambda = s\}$. For each $s \in S^+$, $V_{-s}$ is the dual basis of $V_s$ with respect to the pairing $V_s \times V_{-s} \to \ov{k}$.  
    \end{enumerate}
\end{lem}
\begin{proof}
    Fix a $\sigma$-orbit $C \subset S$ and $s \in C$. We construct $V_t$ simultaneously for $t \in C \cup -C$. 
    
    First, suppose $-s \not \in C$, which holds if $s \in I_0 \sqcup (-I_0)$. We take a basis $V_s$ of $V^s$ and set $V_{\sigma^i(s)} = \sigma^i(V_s)$ for each $i$. This determines $V_t$ for every $t \in C$, and we may set $V_{-t}$ as the dual basis of $V_t$. 
    
    Next, suppose $-s \in C$. Then $C = -C$, so $d_s$ is even and $-s = \sigma^{d_s/2}(s)$. Then, $V^s$ admits an orthogonal basis $V_s$ with respect to a skew-Hermitian form $B$ as in \Cref{prop:orbit_description}. For $0 \leq i < d_s/2$, we set $V_{\sigma^i(s)} = \sigma^i(V_s)$, and for $d_s/2 \leq i < d_s$, we set $V_{\sigma^i(s)}$ as the dual basis of $I_{\sigma^{i-d_s/2}(s)}$. Since $V_s$ is orthogonal with respect to $B$, $V_{\sigma^{d_s/2}(s)}$ is equal to $\sigma^{d_s}(V_s)$ up to scalar, so we get a desired basis $V_t$ for every $t \in C$. 
\end{proof}

From now on, we will fix this basis $\{ v_\lambda \}$ of $\ov{V}$. 

\begin{lem}
    For each $ i \in \{0, 1\}$, let $\Lambda_i = \{ \lambda \in \Lambda \mid s_\lambda \in I_i \}$. For each $\lambda \in \Lambda_1$, there is a positive integer $n_\lambda$ such that 
    \[
        \sigma^{-n_\lambda}(s_\lambda) \in S^- \cap \sigma(S^+), \quad 
        \sigma^{-k}(s_\lambda) \in S^- \quad (1\leq k < n_\lambda). 
    \]
    Let $\beta_\lambda \in \Lambda_1$ be the element such that $v_{\beta_\lambda}$ is dual to $v_{\sigma^{-n_\lambda}(\lambda)}$.  
\end{lem}
\begin{proof}
    Since $s_\lambda \in S^+$ and $\sigma^{-1}(s_\lambda) \in S^-$ for $\lambda \in \Lambda_1$, the existence of $n_\lambda$ follows formally. By \Cref{lem:suitable_basis} (3), we can take $\beta_\lambda \in \Lambda$ so that $v_{\beta_\lambda}$ is dual to $v_{\sigma^{-n_\lambda}(\lambda)}$. Since $s_{\beta_\lambda} = - \sigma^{-n_\lambda}(s_\lambda)$, we have $\beta_\lambda \in \Lambda_1$. 
\end{proof}

\begin{lem} \label{lem:pi_U_isom}
    Let $V^{-I_0} \subset V$ be the $k$-rational form of $\bigoplus_{s \in I_0} V^{-s}$. Then, there is a natural isomorphism $\msf{\wtd{U}}_I \cong \msf{U}_I \times V^{-I_0}$. 
\end{lem}
\begin{proof}
    Fix $\bb{A}^\Lambda \cong \msf{\ov{V}}$ associated to $\{v_\lambda\}_{\lambda \in \Lambda}$ and let $u_\lambda$ denote the $\lambda$-th coordinate. Let $c_\lambda \in \ov{k}^\times$ be the element such that $v_\lambda = c_\lambda \sigma(v_{\sigma^{-1}(\lambda)})$. Then, $\msf{\wtd{U}}$ is the closed subspace of $\bb{A}^\Lambda$ given by 
    \[
        u_{\sigma^{-1}(\lambda)}^q = c_\lambda u_\lambda
    \]
    for every $\lambda$ such that $s_\lambda \not \in I$. If $s_\lambda \in - I_0$, then $c_\lambda = 1$ and the condition is equivalent to $u_\lambda \in k_{s_\lambda}$ and $u_{\sigma(\lambda)} = \sigma(u_\lambda)$. Otherwise, let $d_\lambda$ be the minimum positive integer such that $\sigma^{-d_\lambda}(\lambda) \in \Lambda_1$. Then, the defining equation of $\msf{\wtd{U}}$ for $u_\lambda$ is written as 
    \[
        u_\lambda = \prod_{0 \leq i < d_\lambda} c_{\sigma^{-i}(\lambda)}^{-q^i} \cdot u^{q^{d_\lambda}}_{\sigma^{-d_\lambda}(\lambda)}. 
    \]
    Thus, the natural projection
    \[
        \wtd{\msf{U}}_I \to \msf{U}_I \times V^{-I_0} ,\quad 
        (u_\lambda) \mapsto \Bigl((u_\lambda)_{\lambda \in \Lambda_0 \sqcup \Lambda_1}, \sum_{\substack{\lambda \in \Lambda \\ s_\lambda \in - I_0}} u_\lambda v_\lambda\Bigr) 
    \]
    is an isomorphism. 
\end{proof}

Finally, we can apply \Cref{thm:cohAS} to the following Artin-Schreier local system on $\wtd{\msf{U}}$. Let $U_{I_1} = \bigoplus_{s \in I_1} \ov{V}^s$ and let $\pi_{I_1} \colon \ov{V} \to U_{I_1}$ be the natural projection. 

\begin{prop} \label{prop:cohmology_application}
    Let $f \colon \msf{\wtd{U}}_I \to \bb{A}^1$ be the restriction of the map associated to the polynomial map $v \mapsto \langle \sigma(v), \pi_{I_1}(v) \rangle$ on $\ov{V}$. Let $n = 2\lvert \Lambda_0 \rvert + \lvert \Lambda_1 \rvert$. Then, we have 
    \[
        H_c^i(\msf{\wtd{U}}_I, f^*\cl{L}_{\psi}) = 0 \quad (i \neq n), \quad 
        H_c^n(\msf{\wtd{U}}_I, f^*\cl{L}_{\psi}) = \sqrt{\lvert V \rvert}. 
    \]
    Moreover, if $I_0 = \phi$, the natural map $H_c^n(\msf{\wtd{U}}_I, f^*\cl{L}_{\psi}) \to H^n(\msf{\wtd{U}}_I, f^*\cl{L}_{\psi})$ is an isomorphism. 
\end{prop}

\begin{proof}
    We keep the notation in the proof of \Cref{lem:pi_U_isom}. There is an isomorphism $\msf{U}_I \times V^{-I_0} \cong \msf{\wtd{U}}_I$ which induces an embedding $\msf{U}_I \subset \msf{\ov{V}}$ given by 
    \[
        u_\lambda = \prod_{0 \leq i < d_\lambda} c_{\sigma^{-i}(\lambda)}^{-q^i} \cdot u^{q^{d_\lambda}}_{\sigma^{-d_\lambda}(\lambda)} \quad (s_\lambda \notin I \sqcup (-I_0)), \quad 
        u_\lambda = 0 \quad (s_\lambda \in - I_0). 
    \]
    By construction, $f$ is stable under the translation by $V^{-I_0}$, so it is enough to compute on the above component $\msf{U}_I \subset \msf{\ov{V}}$. 
    
    For each $\lambda \in \Lambda$, let $-\lambda \in \Lambda$ be the element such that $v_{-\lambda}$ is dual to $v_\lambda$. For $\lambda \in \Lambda_1$, we have $d_{-\lambda} = n_\lambda$ and $\sigma^{-k_{-\lambda}}(-\lambda) = \beta_\lambda$. Thus, $f_I \colon \bb{A}^{\Lambda_0 \sqcup \Lambda_1} \cong \msf{U}_I \xrightarrow{f} \bb{A}^1$ is given by the polynomial
    \[
        \left\langle \sum_{\lambda \in \Lambda} u_i^q \sigma(v_i), \sum_{\lambda \in \Lambda_1} u_\lambda v_\lambda \right\rangle = \sum_{\lambda \in \Lambda_1} a_\lambda u_\lambda u_{\beta_\lambda}^{q^{n_\lambda}} 
    \]
    for some $a_\lambda \in \ov{k}^\times$. 
    
    Now, take a decomposition $\Lambda_1 = \bigsqcup C$ into orbits under the map $\lambda \mapsto \beta_\lambda$. For each $C$, let $f_C\colon \bb{A}^{C} \to \bb{A}^1$ be the map associated to the polynomial $\sum_{\lambda \in C} a_\lambda u_\lambda u_{\beta_\lambda}^{q^{n_\lambda}}$. Then, we have
    \[
        (\bb{A}^{\Lambda_0 \sqcup \Lambda_1}, f_I^*\cl{L}_\psi) \cong \prod_{C} (\bb{A}^{C}, f_C^* \cl{L}_\psi) \times (\bb{A}^{\Lambda_0}, \Qla). 
    \]
    Let $n_C = \sum_{\lambda \in C} n_\lambda$. By \Cref{thm:cohAS} (1), we have 
    \[
        R\Gamma_c(\bb{A}^{C}, f_C^*\cl{L}_\psi) \cong \Qla^{\oplus q^{n_C}}[- \lvert C \rvert] ,\quad 
        R\Gamma_c(\bb{A}^{\Lambda_0}, \Qla) \cong \Qla[- 2 \lvert \Lambda_0 \rvert]. 
    \]
    Since $\sum_C n_C = \tfrac{1}{2} \lvert \Lambda \backslash (\Lambda_0 \sqcup (-\Lambda_0)) \rvert$ and $\dim_k V^{-I_0} = \lvert \Lambda_0 \rvert$, we get 
    \[
        R\Gamma_c(\bb{A}^{\Lambda_0 \sqcup \Lambda_1}, f_I^*\cl{L}_\psi) \cong \Qla^{\oplus \sqrt{\lvert V \rvert} / \lvert V^{-I_0} \rvert}[-n]
    \]
    by the K\"{u}nneth formula. Since $\msf{U}_I \times V^{-I_0} \cong \msf{\wtd{U}}_I$, the first claim follows. The second claim follows from \Cref{thm:cohAS} (2). 
\end{proof}

\begin{rmk} \label{rmk:simple_polarization}
    One simple way to choose a polarization is that for each orbit $C \subset S$ under $\Gal(\ov{k} / k) \times \{ \pm 1\}$, we choose a representative $s \in C$ and if $C \neq \Gal(\ov{k} / k) \cdot s$, we take a polarization
    \[
        C = \Gal(\ov{k} / k) \cdot s \sqcup \Gal(\ov{k} / k) \cdot (-s), 
    \]
    and if $C = \Gal(\ov{k} / k) \cdot s$, we take a polarization
    \[
        C = \{ \sigma^i(s) \mid 0 \leq i < d_s / 2 \} \sqcup \{ \sigma^i(s) \mid d_s / 2 \leq i < s \}. 
    \]
    Then, we only need the case $n = 1$ for \Cref{thm:cohAS}, which essentially reduces to \cite[(8.11)]{Del74}. Nevertheless, the above generality is needed in applications to positive-depth Deligne-Lusztig induction and our forthcoming work on the generalization of \cite{BW16}. 
\end{rmk}


\section{Geometric realization of Heisenberg-Weil representations} \label{sec:geometric_realization}

In this section, we realize twisted Heisenberg-Weil representations $\kappa_{x, r, \phi}$ as $\psi$-isotypic parts of the \'{e}tale cohomology of certain Lang torsors of affine spaces. In the unramified setting, such Lang torsors arise from the Deligne-Lusztig construction for Heisenberg groups. 
We keep the notation in \Cref{sec:twisted_Heisenberg_Weil} and let $\phi \colon \msf{m}_{x, r} \to k$ be an $r$-generic $M$-stable map. 

\subsection{Cohomology of Heisenberg torsors}

In this section, we introduce a finite \'{e}tale $\msf{H}_{x, r, \phi}$-torsor of an affine space and compute its cohomology by applying \Cref{prop:cohmology_application}. 

Under \Cref{ass:without_symmetric_ramification}, we have a genuine weight decomposition of $\msf{m}^\perp_{x, r}$ indexed by $\Phi_e^M$. To apply the results of \Cref{ssec:Lang_torsor}, we fix a polarization $\Phi_e^M = \Phi_e^{M+} \sqcup \Phi_e^{M-}$. 

\begin{defi}
    Let $\msf{u}^+_{r/2} = \msf{u}_{\Phi_e^{M+}, r/2}$ and $\msf{u}^-_{r/2} = \msf{u}_{\Phi_e^{M-}, r/2}$. Let
    \[
        i^+ \colon \msf{u}^+_{r/2} \to H_{x, r, \phi, \ov{k}} ,\quad 
        i^- \colon \msf{u}^-_{r/2} \to H_{x, r, \phi, \ov{k}}
    \]
    be homomorphisms lifting the natural inclusions as in \Cref{prop:lifting_to_Heisenberg} and let 
    \[
        i = (i^- \circ \pi^-) \cdot (i^+\circ \pi^+) \colon \msf{m}^\perp_{x, r/2, \ov{k}} \to H_{x, r, \phi, \ov{k}}. 
    \]
    be a section of the quotient $H_{x, r,\phi} / \bb{A}^1 \cong \msf{m}^\perp_{x, r/2}$ over $\ov{k}$. Here, $\pi^+ = \pi_{\Phi_e^{M+}}$ and $\pi^- = \pi_{\Phi_e^{M-}}$. 
\end{defi}

\begin{defi} \label{defi:diag_W}
    Let $L_{x, r, \phi} \colon H_{x, r, \phi} \to H_{x, r, \phi}$ be the Lang torsor of $H_{x, r, \phi}$ sending $h$ to $\sigma(h) h^{-1}$. 
    Let $I \subset \Phi_e^{M+}$ be the characteristic index set of the chosen polarization and let $\msf{W}_{I, r} \to \msf{u}_{I, r}$ be the finite \'{e}tale $\msf{H}_{x, r, \phi}$-torsor given by the Cartesian diagram
    \begin{center}
        \begin{tikzcd}
            \msf{W}_{I, r} \ar[r] \ar[d] & H_{x, r, \phi, \ov{k}} \ar[d, "L_{x,r, \phi}"] \\
            \msf{u}_{I, r/2} \ar[r] & H_{x, r, \phi, \ov{k}}. 
        \end{tikzcd}
    \end{center}
    Here, the bottom map is the map induced by \Cref{prop:lifting_to_Heisenberg}. 
\end{defi}

Note that $L_{x, r, \phi}$ is a finite \'{e}tale $\msf{H}_{x, r, \phi}$-torsor because $H_{x, r, \phi}$ is the perfection of a smooth linear algebraic group over $k$ by \cite[Lemma A.26]{Zhu17}. 

Let $L_{\msf{m}^\perp_{x, r/2}} \colon \msf{m}^\perp_{x, r/2} \to \msf{m}^\perp_{x, r/2}$ be the Lang torsor associated to $\msf{m}^\perp_{x, r/2}$. There is a natural commutative diagram
\begin{center}
    \begin{tikzcd}
        H_{x, r, \phi} \ar[r] \ar[d, "L_{x, r}"']& \msf{m}^\perp_{x, r/2} \ar[d, "L_{\msf{m}^\perp_{x, r/2}}"] \\
        H_{x, r, \phi} \ar[r] & \msf{m}^\perp_{x, r/2}. 
    \end{tikzcd}
\end{center}
There is a natural inclusion $k \subset \msf{H}_{x, r, \phi}$ and we set $\wtd{\msf{u}}_{I, r/2} = \msf{W}_{I, r} / k$. Then, $\wtd{\msf{u}}_{I, r/2} \to \msf{u}_{I, r/2}$ is the pullback of $L_{\msf{m}^\perp_{x, r/2}}$ along $\msf{u}_{I, r/2} \subset \msf{m}^\perp_{x, r/2, \ov{k}}$, so it is compatible with the notation in \Cref{defi:fin_et_torsor}. Let $I_1 = \Phi_e^{M+} \cap \sigma(\Phi_e^{M-})$. 

\begin{prop} \label{prop:identification_local_system}
    Let $f \colon \msf{\wtd{u}}_{I, r/2} \to \bb{A}^1$ be the map associated to the restriction of the polynomial map $v \mapsto \langle \sigma(v), \pi_{I_1}(v) \rangle_{\phi}$ on $\msf{m}^\perp_{x, r/2, \ov{k}}$. Then, we have an isomorphism 
    \[
        f^* \cl{L}_\psi \cong \msf{W}_{I, r} \times^{k, \psi} \Qla
    \]
    of rank $1$ \'{e}tale $\Qla$-local systems on $\msf{\wtd{u}}_{I, r/2}$. 
\end{prop}
\begin{proof}
    Let $A \in \PerfAlg$. The fiber of $\msf{W}_{I, r} \to \wtd{\msf{u}}_{I, r/2}$ at a point $v \in \wtd{\msf{u}}_{I, r/2}(A)$ classifies the set of elements $a \in A$ such that 
    \[
        \sigma(a i(v)) (a i(v))^{-1}= i^+(\sigma(v) - v).  
    \]
    Since $\bb{A}^1 \subset H_{x, r, \phi}$ is central, this equation is equivalent to 
    \begin{equation} \label{eq:rewrite_moduli}
        a^q - a = \sigma(i(v))^{-1} i^+(\sigma(v) - v) i(v). 
    \end{equation}
    It is enough to show that the morphism $\wtd{\msf{u}}_{I, r/2} \to \bb{A}^1$ associated to the right-hand side is equal to the map $v \mapsto \langle \sigma(v),  \pi_{I_1}(v) \rangle_{\phi}$ up to a factor $F^q - F$ for some function $F \colon \msf{\wtd{u}}_{I, r/2} \to \bb{A}^1$. 

    For each $\alpha \in \Phi_e^M$, let $v_\alpha \in \msf{u}_{\alpha, r/2}(A)$ be the projection of $v$. Then, $\sigma(v_{\sigma^{-1}(\alpha)}) = v_\alpha$ if $\alpha \not \in I$. Let $i_\alpha \colon \msf{u}_{\alpha, r/2} \to H_{x, r, \phi, \ov{k}}$ be the canonical homomorphism. Since $i_\alpha$ and $i_\beta$ commute unless $\beta = -\alpha$, the right-hand side of \eqref{eq:rewrite_moduli} can be decomposed into a sum over $\alpha \in \Phi_e^{M+}$, where each factor is a product of the terms in $i_\alpha$ and $i_{-\alpha}$. 

    Let $\alpha \in \Phi_e^{M+}$. First, suppose $\sigma^{-1}(\alpha) \in \Phi_e^{M+}$. Then, the involved terms in \eqref{eq:rewrite_moduli} are 
    \[
        \sigma( i_{-\sigma^{-1}(\alpha)}(v_{-\sigma^{-1}(\alpha)}) i_{\sigma^{-1}(\alpha)}(v_{\sigma^{-1}(\alpha)}) )^{-1} \cdot i_\alpha(\sigma(v_{\sigma^{-1}(\alpha)}) - v_\alpha) \cdot i_{-\alpha}(v_{-\alpha}) i_\alpha(v_\alpha). 
    \]
    Since $- \alpha \not \in I$ and $i^-$, $i^+$ are homomorphisms, the above product equals the commutator
    \[
        [i_\alpha(\sigma(v_{\sigma^{-1}(\alpha)}) - v_\alpha), i_{-\alpha}(v_{-\alpha})]. 
    \] 
    It vanishes unless $\alpha \in I_0$. In that case, the commutator equals
    \[
        \sigma([i_{\sigma^{-1}(\alpha)}(v_{\sigma^{-1}(\alpha)}), i_{-\sigma^{-1}(\alpha)}(v_{-\sigma^{-1}(\alpha)})]) - [i_\alpha(v_\alpha), i_{-\alpha}(v_{-\alpha})]. 
    \] 
    Then, the sum over all $\alpha \in I_0$ equals $F^q - F$ for $F = \sum_{\alpha \in I_0} [i_\alpha(v_\alpha), i_{-\alpha}(v_{-\alpha})]$. 

    Next, suppose $\sigma^{-1}(\alpha) \in \Phi_e^{M-}$. Then, the involved terms in \eqref{eq:rewrite_moduli} are 
    \[
        \sigma( i_{\sigma^{-1}(\alpha)}(v_{\sigma^{-1}(\alpha)}) i_{-\sigma^{-1}(\alpha)}(v_{-\sigma^{-1}(\alpha)}) )^{-1} \cdot i_\alpha(\sigma(v_{\sigma^{-1}(\alpha)}) - v_\alpha) \cdot i_{-\alpha}(v_{-\alpha}) i_\alpha(v_\alpha). 
    \]
    Since $-\alpha \notin I$, this equals the commutator $[i_{-\alpha}(-v_{-\alpha}), i_\alpha(-v_\alpha)]$. Then, \eqref{eq:rewrite_moduli} is rewritten as 
    \[
        a^q - a = \sum_{\alpha \in I_1} \langle \sigma(v_{-\sigma^{-1}(\alpha)}), v_\alpha \rangle_\phi + (F^q - F) = \langle \sigma(v), \pi_{I_1}(v) \rangle_\phi + (F^q - F). 
    \]
    Thus, we get the claim. 
\end{proof}

\begin{cor} \label{cor:Heisenberg_realization}
    Let
    $
        H_c^\bullet(\msf{W}_{I, r}, \Qla)[\psi] 
    $
    be the $\psi$-isotypic part of $H_c^\bullet(\msf{W}_{I, r}, \Qla)$ under the $k$-action. For some explicit $n$, we have 
    \[
        H_c^i(\msf{W}_{I, r}, \Qla)[\psi] = 0 \quad (i \neq n) ,\quad 
        \dim H_c^n(\msf{W}_{I, r}, \Qla)[\psi] = \lvert \msf{m}^\perp_{x, r/2} \rvert^{1/2}. 
    \]  
    In particular, $H_c^n(\msf{W}_{I, r}, \Qla)[\psi]$ is the Heisenberg representation of $\msf{H}_{x, r, \phi}$ with central character $\psi$. Moreover, $n = \dim \msf{W}_{I, r}$ when $I_0 = \phi$, and in that case, we have 
    \[
        H_c^n(\msf{W}_{I, r}, \Qla)[\psi] \cong H^n(\msf{W}_{I, r}, \Qla)[\psi]. 
    \]  
\end{cor}
\begin{proof}
    By \Cref{prop:identification_local_system}, we have
    \[
        H_c^\bullet(\msf{W}_{I, r}, \Qla)[\psi] \cong H_c^\bullet(\msf{\wtd{u}}_I, \msf{W}_{I, r} \times^{k, \psi^{-1}} \Qla) \cong H_c^\bullet(\msf{\wtd{u}}_I, f^* \cl{L}_{\psi^{-1}}). 
    \]
    Then, we get the dimension of $H_c^\bullet(\msf{W}_{I, r}, \Qla)[\psi]$ by \Cref{prop:cohmology_application}. As in \cite[Lemma 1.2]{Ger77}, it is a general consequence that $\msf{H}_{x, r, \phi} \circlearrowright H_c^n(\msf{W}_{I, r}, \Qla)[\psi]$ is a Heisenberg representation. The second claim is a consequence of the second claim of \Cref{prop:cohmology_application}. 
\end{proof}

By the geometric construction,  it is easy to extend the Heisenberg representation $\msf{H}_{x, r, \phi} \circlearrowright H_c^n(\msf{W}_{I, r}, \Qla)[\psi]$ to a Heisenberg-Weil representation. 

\begin{lem}
    The adjoint action of $\msf{M}_{x, r, \phi}$ on $H_{x, r, \phi}$ stabilizes $\msf{W}_{I, r}$. In particular, there is a natural action $\msf{HW}_{x, r, \phi} \circlearrowright \msf{W}_{I, r}$. 
\end{lem}
\begin{proof}
    The adjoint action of $\msf{M}_{x, r, \phi}$ stabilizes $\msf{u}_{I, r} \subset H_{x, r, \phi, \ov{k}}$ and commutes with $\sigma$. Since the Lang torsor $L_{x, r, \phi}$ is equivariant under $\msf{M}_{x, r, \phi}$, $\msf{W}_{I, r} \subset H_{x, r, \phi, \ov{k}}$ is stable under $\msf{M}_{x, r, \phi}$. 
\end{proof}

\begin{prop} \label{prop:crude_geometric_realization}
    Suppose that $(\msf{Z}_{\msf{M}}, \msf{m}^\perp_{E, x, r/2}, \Gamma)$ is a tame toral action without symmetric ramification. For a polarization $\Phi_e^M = \Phi_e^{M+} \sqcup \Phi_e^{M-}$ with characteristic index set $I \subset \Phi_e^{M+}$, there is a unique irreducible representation $\kappa_{I}$ of $\msf{HW}_{x, r, \phi}$ with the following conditions. 
    \begin{enumerate}
        \item $\kappa_I \vert_{\msf{H}_{x, r, \phi}}$ is the Heisenberg representation with central character $\psi$. 
        \item $\kappa_I \cong H_c^n(\msf{W}_{I, r}, \Qla)[\psi]$ for a unique $n \geq 0$. 
    \end{enumerate}
\end{prop}
\begin{proof}
    The action $\msf{HW}_{x, r, \phi} \circlearrowright \msf{W}_{I, r}$ induces $\msf{HW}_{x, r, \phi} \circlearrowright H_c^\bullet(\msf{W}_{I, r}, \Qla)$. Since $\phi$ is $M$-stable, the center $k \subset \msf{H}_{x, r, \phi}$ is also central in $\msf{HW}_{x, r, \phi}$. Thus, the above action induces $\msf{HW}_{x, r, \phi} \circlearrowright H_c^\bullet(\msf{W}_{I, r}, \Qla)[\psi]$. Let $n$ be as in \Cref{cor:Heisenberg_realization}. Then, $\kappa_I = H_c^n(\msf{W}_{I, r}, \Qla)[\psi]$ is the desired irreducible representation of $\msf{HW}_{x, r, \phi}$. 
\end{proof}

\subsection{Identification of Heisenberg-Weil representations}

In this section, we compute the trace of $\kappa_I$ to show $\kappa_I \cong \kappa_{x, r, \phi}$. For this, we first extend the group action on $\kappa_I$. 

For each $C \in \Phi_\Sigma^M$, let $\msf{M}_C \subset \Aut(\msf{m}^{\perp, C}_{x, r/2})$ be as in \Cref{defi:MC} and let $\wtd{\msf{M}}_C = \prod_C \msf{M}_C$. As in \Cref{prop:interpretation_epsilon_twist}, we have $\msf{M}_{x, r, \phi} \subset \msf{\wtd{M}}_{x, r, \phi}$ inside $\Sp(\msf{m}^\perp_{x, r/2})$. 

\begin{prop}
    The adjoint action of $\msf{M}_{x, r, \phi}$ on $H_{x, r, \phi}$ extends to $\msf{\wtd{M}}_{x, r, \phi}$ so that 
    \[
        \bb{A}^1 \times \msf{m}^\perp_{x, r/2, \ov{k}} \xrightarrow{\sim} H_{x, r, \phi, \ov{k}}, \quad (a, v) \mapsto a  i(v)
    \]
    is equivariant under $\msf{\wtd{M}}_{x, r, \phi}$. 
\end{prop}
\begin{proof}
    The isomorphism in the statement equips a unique $\msf{\wtd{M}}_{x, r, \phi}$-action on $H_{x, r, \phi, \ov{k}}$ as a $\ov{k}$-variety. We show that it is an action as a group scheme. 

    First, the $\wtd{\msf{M}}_{x, r, \phi}$-action preserves each $\msf{\ov{m}}^{\perp, \alpha}_{x, r/2}$ and is additive on $\msf{u}^+_{r / 2}, \msf{u}^-_{r / 2} \subset H_{x, r, \phi, \ov{k}}$. Thus, it is enough to verify that the action commutes with the multiplication between $\msf{u}^+_{r / 2}$ and $\msf{u}^-_{r / 2}$. Let $\alpha \in \Phi_e^{M-}$, $\beta \in \Phi_e^{M+}$ and let $u_\alpha \in \msf{\ov{m}}_{x, r/2}^{\perp, \alpha}$, $u_\beta \in \msf{\ov{m}}_{x, r/2}^{\perp, \beta}$. For every $\gamma \in \wtd{\msf{M}}_{x, r, \phi}$, 
    \[
        \gamma \cdot (i(u_\alpha) \cdot  i(u_\beta)) = \gamma \cdot i(u_\alpha + u_\beta) = i(\gamma u_\alpha + \gamma u_\beta) = i(\gamma u_\alpha) \cdot i(\gamma u_\beta), 
    \]
    and since $[u_\beta, u_\alpha] = [\gamma u_\beta, \gamma u_\alpha]$, we have 
    \[
        \gamma \cdot (i(u_\beta) \cdot  i(u_\alpha)) = \gamma \cdot ([u_\beta, u_\alpha] \cdot i(u_\alpha + u_\beta)) = [u_\beta, u_\alpha] \cdot i(\gamma u_\alpha + \gamma u_\beta) = i(\gamma u_\beta) \cdot i(\gamma u_\alpha). 
    \]
    Thus, the $\wtd{\msf{M}}_{x, r, \phi}$-action on $H_{x, r, \phi, \ov{k}}$ is an action as a group scheme. 

    It is enough to show that this action descends to $H_{x, r, \phi}$. In other words, we show that it commutes with $\sigma$. First, $\sigma(i(u_\alpha)) = i(\sigma u_\alpha)$ for each $\alpha \in \Phi_e^M$ and $u_\alpha \in \msf{\ov{m}}_{x, r/2}^{\perp, \alpha}$. Thus, 
    \[
        \gamma \cdot \sigma(i(u_\alpha)) = i(\gamma \cdot \sigma(u_\alpha)) = i(\sigma(\gamma u_\alpha)) = \sigma(\gamma \cdot i(u_\alpha))
    \] 
    since the $\wtd{\msf{M}}_{x, r, \phi}$-action on $\msf{\ov{m}}^\perp_{x, r/2}$ commutes with $\sigma$. Since the actions of $\sigma$ and $\gamma$ are group actions on $H_{x, r, \phi, \ov{k}}$ and $H_{x, r, \phi, \ov{k}}$ is generated by $\bb{A}^1$ and $i(u_\alpha)$ as above, it follows that $\gamma$ commutes with $\sigma$. 
\end{proof}

\begin{defi}
    We say that the semidirect product 
    \[  
        \msf{\wtd{HW}}_{x, r, \phi} = \msf{H}_{x, r, \phi} \rtimes \msf{\wtd{M}}_{x, r, \phi}
    \]
    is the extended Heisenberg-Weil group. Since $\msf{u}_{I, r/2} \subset H_{x, r, \phi, \ov{k}}$ is stable under $\msf{\wtd{M}}_{x, r, \phi}$, there is a natural action $\msf{\wtd{HW}}_{x, r, \phi} \circlearrowright \msf{W}_{I, r}$. Let $\wtd{\kappa}_I$ be the irreducible representation $\wtd{\msf{HW}}_{x, r, \phi} \circlearrowright H_c^n(\msf{W}_{I, r}, \Qla)[\psi]$ that is an extension of $\kappa_I$. 
\end{defi}

We will show that $\wtd{\kappa}_I$ is independent of a polarization. Since $\wtd{\kappa}_I$ is irreducible and its restriction to $\msf{H}_{x, r, \phi}$ is independent of $I$, any two $\wtd{\kappa}_I$ differ only by a character of $\wtd{\msf{M}}_{x, r, \phi}$. Thus, it is enough to determine $\wtd{\kappa}_I\vert_{\wtd{\msf{M}}_{x, r, \phi}}$. First, we have the following general computation. 


\begin{defi}(cf.\ \Cref{defi:d_gamma_sign_factor})
    For each $\gamma \in  \msf{\wtd{M}}_{x, r, \phi}$, let 
    \[
        d_\gamma = \sum_{C \in \Phi_{\Sigma, \sym}^M} (\dim_{k_\alpha} \msf{m}^{\perp, \alpha}_{x, r/2} - \dim_{k_\alpha} \msf{m}^{\perp, \alpha, \gamma}_{x, r/2}). 
    \]
    Here, the superscript $\gamma$ denotes the $\gamma$-fixed locus and $\alpha \in C$ is a representative.  
\end{defi}

\begin{prop} \label{prop:tame_trace}
    If $\gamma \in  \msf{\wtd{M}}_{x, r, \phi}$ is of order coprime to $p$, we have
    \[
        \tr(\gamma \mid \wtd{\kappa}_I) = (-1)^{d_\gamma} \sqrt{\lvert \msf{m}^{\perp, \gamma}_{x, r/2} \rvert}. 
    \]
\end{prop}
\begin{proof}
    Consider the diagram of $\gamma$-fixed loci
    \begin{center}
        \begin{tikzcd}
            \msf{W}_{I, r}^{\gamma} \ar[d] \ar[r] & H_{x, r, \phi, \ov{k}}^{\gamma} \ar[d] \\
            \msf{u}_{I, r / 2}^{\gamma} \ar[r] & H_{x, r, \phi, \ov{k}}^{\gamma}. 
        \end{tikzcd}
    \end{center}
    By construction, this diagram is Cartesian and the vertical arrow $H_{x, r, \phi, \ov{k}}^{\gamma} \to H_{x, r, \phi, \ov{k}}^{\gamma}$ is a Lang torsor associated to $H_{x, r, \phi}^\gamma$. Since $\gamma$ is of order coprime to $p$, we get 
    \[
        \tr(\gamma \mid R\Gamma_c(\msf{W}_{I, r}, \Qla)[\psi]) = \chi(R\Gamma_c(\msf{W}^{\gamma}_{I, r}, \Qla)[\psi]). 
    \]
    Here, $\chi$ denotes the Euler characteristic of a perfect complex. We will compute the right-hand side. Since $\gamma$ is of order coprime to $p$, $\langle -, - \rangle_\phi$ is non-degenerate over $\msf{m}^{\perp, \gamma}_{x, r/2}$. Moreover, 
    \[
        \msf{m}^{\perp, \gamma}_{x, r/2, \ov{k}} = \bigoplus_{\alpha \in \Phi_e^M} \msf{u}^\gamma_{\alpha, r/2}
    \]
    is a genuine weight decomposition. As $\wtd{\msf{u}}^\gamma_{I, r/2} \cong \msf{W}^\gamma_{I, r} / k$, we have 
    \[
        \chi(R\Gamma_c(\msf{W}^{\gamma}_{I, r}, \Qla)[\psi]) = \chi(R\Gamma_c(\wtd{\msf{u}}^\gamma_{I, r/2}, \msf{W}^{\gamma}_{I, r} \times^{k, \psi^{-1}} \Qla)). 
    \]
    By the same argument as in \Cref{cor:Heisenberg_realization}, we get 
    \[
        \chi(R\Gamma_c(\wtd{\msf{u}}^\gamma_{I, r/2}, \msf{W}^{\gamma}_{I, r} \times^{k, \psi^{-1}} \Qla)) = (-1)^m \sqrt{\lvert \msf{m}^{\perp, \gamma}_{x, r/2}\rvert}
    \] 
    for some specific $m \geq 0$. Thus, we have 
    \[
        \tr(\gamma \mid \wtd{\kappa}_I) = (-1)^{n - m} \sqrt{\lvert \msf{m}^{\perp, \gamma}_{x, r/2} \rvert}. 
    \]
    We compute the parity of $n - m$. By \Cref{prop:cohmology_application}, $n \equiv \dim \msf{u}_{I_1, r/2}$ and $m \equiv \dim \msf{u}^\gamma_{I_1, r/2}$ modulo $2$. We will show that 
    \[
        \dim \msf{u}_{I_1, r/2} - \dim \msf{u}^\gamma_{I_1, r/2} \equiv d_\gamma \pmod{2}. 
    \]
    We decompose into a sum over each $C \in \Phi_\Sigma^M$. If $C \in \Phi_{\Sigma, \asym}^M$, $\lvert I_1 \cap C \rvert$ is even, so 
    \[
        \dim \msf{u}_{I_1 \cap C, r/2} \equiv \dim \msf{u}_{I_1 \cap C, r/2}^\gamma \equiv 0
    \]
    since $\dim \msf{u}_{\alpha, r/2}$ and $\dim \msf{u}_{\alpha, r/2}^\gamma$ are independent of $\alpha \in C$. By the same reason, if $C \in \Phi^M_{\Sigma, \sym}$, $\lvert I_1 \cap C \rvert$ is odd and we have 
    \[
        \dim \msf{u}_{I_1 \cap C, r/2} - \dim \msf{u}^\gamma_{I_1 \cap C, r/2} \equiv \dim \msf{u}_{\alpha, r/2} - \dim \msf{u}^\gamma_{\alpha, r/2} \pmod{2}.  
    \]
    We get the claim since $\dim_{k_\alpha} \msf{m}^{\perp, \alpha}_{x, r/2} = \dim \msf{u}_{\alpha, r/2}$ and $\dim_{k_\alpha} \msf{m}^{\perp, \alpha, \gamma}_{x, r/2} = \dim \msf{u}^\gamma_{\alpha, r/2}$. 
\end{proof}

For exact computations, we need to decompose $\wtd{\kappa}_I$ as a tensor product over $C \in \Phi_\Sigma^M$.

\begin{lem} \label{lem:Kunneth_decomposition}
    For each $C \in \Phi_\Sigma^M$, let 
    \[
        \msf{\wtd{u}}_{I \cap C, r/2} = \msf{\wtd{u}}_{I, r/2} \cap \msf{u}_{C, r/2}, \quad f_C = f\vert_{\msf{\wtd{u}}_{I \cap C, r/2}}. 
    \]
    Then, we have $(\msf{\wtd{u}}_{I, r/2}, f^* \cl{L}_{\psi^{-1}}) \cong \prod_{C} (\msf{\wtd{u}}_{I \cap C, r/2}, f_C^* \cl{L}_{\psi^{-1}})$ and 
    \[
        \wtd{\kappa}_I\vert_{\msf{\wtd{M}}_{x, r, \phi}} \cong \bigotimes_C H_c^{d_C}(\msf{\wtd{u}}_{I \cap C, r/2}, f_C^* \cl{L}_{\psi^{-1}})
    \]
    as a representation of $\msf{\wtd{M}}_{x, r, \phi} \cong \prod_C \msf{M}_C$ for a unique $d_C \geq 0$ for each $C$. 
\end{lem}
\begin{proof}
    First, we have $\msf{\wtd{u}}_{I, r/2} \cong \prod_C \msf{\wtd{u}}_{I \cap C, r/2}$ by the explicit defining equation of $\msf{\wtd{u}}_{I, r/2} \subset \msf{m}^\perp_{x, r/2, \ov{k}}$ in the proof of \Cref{lem:pi_U_isom}. Moreover, the polynomial associated to $f$ is decomposed into a sum of $f_C$ as in the proof of \Cref{prop:cohmology_application}. Then, the last claim follows from the K\"{u}nneth formula.  
\end{proof}

\begin{lem} \label{lem:compute_polarized}
    Let $C \in \Phi_\Sigma^M$ and $\alpha \in C$. If $I \cap C = \Gal(\ov{k} / k) \cdot \alpha$, then we have 
    \[
        \tr(\gamma \mid H_c^{d_C}(\msf{\wtd{u}}_{I \cap C, r/2}, f_C^* \cl{L}_\psi)) = \sqrt{\lvert \msf{m}^{\perp, C, \gamma}_{x, r/2} \rvert} \quad (\gamma \in \msf{M}_C). 
    \]
\end{lem}
\begin{proof}
    Let $C^+ = C \cap \Phi_e^{M+}$ and $C^- = C \cap \Phi_e^{M-}$. Since $I \cap C = \Gal(\ov{k} / k) \cdot \alpha$, $C^+$ and $C^-$ are $\Gal(\ov{k} / k)$-orbits and we have 
    \[
        \msf{\wtd{u}}_{I \cap C, r/2} \cong \msf{u}_{C^+, r/2} \times \msf{m}^{\perp, C^-}_{x, r/2}, \quad 
        f_C^* \cl{L}_\psi \cong \Qla
    \]
    by \Cref{lem:pi_U_isom} and \Cref{prop:cohmology_application}. Then, $d_C = 2 \dim \msf{u}_{C^+, r/2}$ and the action of $\gamma$ on $H_c^{d_C}(\msf{u}_{C^+, r/2}, \Qla) \cong \Qla$ is trivial. In particular, $\tr(\gamma \mid H_c^{d_C}(\msf{\wtd{u}}_{I \cap C, r/2}, f_C^* \cl{L}_\psi)) = \lvert \msf{m}^{\perp, C^-, \gamma}_{x, r/2} \rvert$. 
\end{proof}

\begin{lem} \label{lem:compute_unitary}
    Let $C \in \Phi_{\Sigma, \sym}^M$ and $\alpha \in C$. If $I \cap C = \{ \alpha \}$, then we have 
    \[
        \tr(\gamma \mid H_c^{d_C}(\msf{\wtd{u}}_{I \cap C, r/2}, f_C^* \cl{L}_{\psi^{-1}})) = (-1)^{(\dim_{k_\alpha} \msf{m}^{\perp, \alpha}_{x, r/2} - \dim_{k_\alpha} \msf{m}^{\perp, \alpha, \gamma}_{x, r/2})} \sqrt{\lvert \msf{m}^{\perp, C, \gamma}_{x, r/2} \rvert} \quad (\gamma \in \msf{M}_C).  
    \]
\end{lem}
\begin{proof}
    Under the assumption, we have 
    \[
        C \cap \Phi_e^{M+} = \{ \sigma^i(\alpha) \mid 0 \leq i < d_\alpha / 2 \} ,\quad 
        C \cap \Phi_e^{M-} = \{ \sigma^i(\alpha) \mid d_\alpha / 2 \leq i < d_\alpha \}. 
    \]
    Let $k = \dim_{k_\alpha} \msf{m}^{\perp,\alpha}_{x, r/2}$ and let $v_1, \ldots, v_k$ be an orthogonal basis of $\msf{m}^{\perp,\alpha}_{x, r/2}$ with respect to $B$ as in \Cref{prop:orbit_description}. Let $k_{\alpha / 2} \subset k_\alpha$ be the extension of degree $d_\alpha / 2$ over $k$. Since $B(v_i, v_i) \in \Ker(\Tr \colon k_\alpha \to k_{\alpha / 2})$ is nonzero and $\Ker(\Tr \colon k_\alpha \to k_{\alpha / 2})$ is one-dimensional over $k_{\alpha / 2}$, we may assume $B(v_i, v_i) = \xi$ for some nonzero $\xi \in \Ker(\Tr \colon k_\alpha \to k_{\alpha / 2})$ for every $i$. 

    By the description in \Cref{prop:cohmology_application}, $\wtd{\msf{u}}_{I \cap C, r/2} \cong \bb{A}^k$ and 
    \[
        f_C = \bigl\langle \sum_{1 \leq i \leq k} \sum_{1 \leq j \leq d_\alpha} \sigma^j(u_i v_i), \sum_{0 \leq i < k} u_i v_i \bigr\rangle = 
        \sum_{1 \leq i \leq k} - B(v_i, v_i)\cdot u_i^{q^{d_\alpha / 2} + 1} = - \xi \cdot \sum_{1\leq i \leq k} u_i^{q^{d_\alpha / 2} + 1}, 
    \]
    where $u_1, \ldots, u_k$ denote the standard coordinate on $\bb{A}^k$. Take the fiber product
    \begin{center}
        \begin{tikzcd}
            X_k \ar[r] \ar[d] & \bb{A}^1 \ar[d, "L_{k_{\alpha/2}}"] \\
            \bb{A}^k \ar[r, "f_C"] & \bb{A}^1. 
        \end{tikzcd}
    \end{center}
    Since $\cl{L}_\psi \cong L_{k_{\alpha/2}} \times^{\psi_{d_{\alpha}/2}} \Qla$, $H_c^{d_C}(\wtd{\msf{u}}_{I \cap C, r/2}, f_C^* \cl{L}_{\psi^{-1}}) \cong H_c^{d_C}(X_k, \Qla)[\psi_{d_{\alpha}/2}]$. Note that the sign difference from \cite[(2.7)]{IT23} is due to the difference of the sign convention for $\cl{L}_\psi$. 
    
    Here, let $z$ denote the standard coordinate of $\bb{A}^1$. As in \cite[Remark 2.1]{IT23}, the coordinate change $z \mapsto \xi^{-1} z$ identifies $X_k$ as the one introduced in \cite[Section 2.3]{IT23}. Moreover, this identification is equivariant under $\mrm{U}(\msf{m}^{\perp, \alpha}_{x, r/2})$ by the explicit description in loc. cit. Then, it follows from \cite[Lemma 2.2, Theorem 2.5]{IT23} that for every $\gamma \in \mrm{U}(\msf{m}^{\perp, \alpha}_{x, r/2})$, we have 
    \[
        \tr(\gamma \mid H_c^{d_C}(X_k, \Qla)[\psi_{d_{\alpha}/2}]) = (-1)^{(\dim_{k_\alpha} \msf{m}^{\perp, \alpha}_{x, r/2} - \dim_{k_\alpha} \msf{m}^{\perp, \alpha, \gamma}_{x, r/2})} \sqrt{\lvert \msf{m}^{\perp, C, \gamma}_{x, r/2} \rvert}. 
    \]
\end{proof}

Now, we can show the desired geometric realization of $\kappa_{x, r, \phi}$ as $\kappa_I$. 

\begin{thm} \label{thm:extended_HW_rep}
    For every polarization with characteristic index set $I$, $\kappa_{x, r, \phi}$ is isomorphic to the Heisenberg-Weil representation $\msf{HW}_{x, r, \phi} \circlearrowright H_c^n(\msf{W}_{I, r}, \Qla)[\psi]$ for a unique $n \geq 0$. 
\end{thm}
\begin{proof}
    First, we show $\wtd{\kappa}_I \cong \wtd{\kappa}_{I'}$ for any two polarizations of $\Phi_e^M$. Since their restriction to $\msf{H}_{x, r, \phi}$ are irreducible and isomorphic, $\wtd{\kappa}_{I'} \cong \wtd{\kappa}_I \otimes \chi$ for some character $\chi \colon \msf{\wtd{M}}_{x, r, \phi} \to \Qlax$. To show $\chi = 1$, it is enough to show that its restriction $\chi_C \colon \msf{M}_C \to \Qlax$ is trivial. 

    By \Cref{prop:tame_trace}, $\chi_C$ is trivial on elements of order coprime to $p$. Let 
    \[
        \msf{M}_C^\der = \mrm{SL}(\msf{m}^{\perp, \alpha}_{x, r/2}) \quad \text{or} \quad \mrm{SU}(\msf{m}^{\perp, \alpha}_{x, r/2})
    \]
    depending on whether $C \in \Phi_{\Sigma, \asym}^M$ or not. Suppose $\lvert k_\alpha \rvert \geq 3$ in the former case and $\lvert k_{\alpha / 2} \rvert \geq 4$ in the latter case. Then, it is well-known that $\msf{M}_C^\der = [\msf{M}_C, \msf{M}_C]$. For the unitary case, see the proof of \cite[Theorem 3.3]{Ger77}. In particular, $\chi_C$ is trivial on $\msf{M}_C^\der$. Since $\lvert \msf{M}_C / \msf{M}_C^\der \rvert$ is coprime to $p$, it follows from the Sylow theorems that $\msf{M}_C$ is generated by $\msf{M}_C^\der$ and elements of order coprime to $p$. Thus, $\chi_C = \id$. 

    In the remaining cases, $\lvert k_\alpha \rvert = 2$ and $d_\alpha = 1$ in the former case and $\lvert k_{\alpha / 2} \rvert = 2, 3$ and $d_\alpha = 2$ in the latter case. Then, we may apply the trace computation in \Cref{lem:compute_polarized} and \Cref{lem:compute_unitary} to get the independence of the polarization of $C$. 

    Now, we show $\kappa_{x, r, \phi} \cong \wtd{\kappa}_I\vert_{\msf{HW}_{x, r, \phi}}$. For this, we choose a polarization as in \Cref{rmk:simple_polarization}. Then, we may apply \Cref{lem:compute_polarized} or \Cref{lem:compute_unitary} for every $C \in \Phi_\Sigma^M$. Then, we get condition (2) in \Cref{thm:Heisenberg_Weil_representations} from \Cref{lem:Kunneth_decomposition}.
\end{proof}

\section{Positive-depth parahoric Deligne-Lusztig induction} \label{sec:positive_depth_induction}

In this section, we apply the cohomology of $\msf{W}_{I, r}$ to an explicit comparison of positive-depth parahoric Deligne-Lusztig induction with the representation-theoretic counterpart $\Ind_{r, \phi}$. Here, we will assume that $M$ splits over $\breve{F}$ and $Z_M / Z_G$ is anisotropic.

\subsection{Preliminaries} \label{ssec:preliminaries}

In this section, we introduce positive-depth parahoric Deligne-Lusztig varieties. 
Let $P \subset \breve{G}$ be a parabolic subgroup over $\breve{F}$ with a Levi decomposition $P = \breve{M} U$. We choose a polarization of $\Phi^{M}_e$ so that $\Phi^{M+}_e$ consists of roots inside $U$. Let $I = I_0 \sqcup I_1 \subset \Phi^{M+}_e$ be the characteristic index set. 

\begin{lem} \label{lem:ass_meaning}
    In this setting, $I_0 = \phi$ and $\msf{u}_{I, r/2} = \msf{u}^+_{r / 2} \cap \sigma(\msf{u}^-_{r/2})$. In particular, 
    \[
        H_c^n(\msf{W}_{I, r/2}, \Qla)[\psi] \to H^n(\msf{W}_{I, r/2}, \Qla)[\psi]
    \]
    is nonzero and an isomorphism for $n  = \dim \msf{W}_{I, r/2}$. 
\end{lem}
\begin{proof}
    Suppose $I_0 \neq \phi$ and let $\alpha \in I_0$. Then, the sum of $\Gal(\ov{k} / k) \cdot \alpha \subset \Phi^{M+}_e$ is a nonzero character of $Z_M / Z_G$. It is a contradiction since $Z_M / Z_G$ is anisotropic. The remaining claims follow from \Cref{cor:Heisenberg_realization}. 
\end{proof}
\begin{rmk}
    The latter property is important to lift the cohomology of varieties over finite fields to the cohomology of rigid analytic varieties via the nearby cycle functor (see \cite[Theorem 2.8]{Mie16}). In our forthcoming work on the generalization of \cite{BW16}, we will utilize this property to construct Jacquet-Langlands pairs of regular supercuspidal representations in the cohomology of local Shimura varieties. 
\end{rmk}

For $r \geq 0$, let $\msf{G}_r$ be the perfect group scheme over $k$ such that $\msf{G}_r(k) = G(F)_{x} / G(F)_{x, r+}$. Here, the identity component $\msf{G}_r^\circ \subset \msf{G}_r$ is the quotient of Moy-Prasad positive loop groups $L^+\cl{G}_{x, 0} / L^+ \cl{G}_{x, r+}$. Similarly, we define $\msf{M}_r \subset \msf{G}_r$ and let
\[
    \msf{U}_r = L^+\cl{U}_{x, 0} /  L^+\cl{U}_{x, r+} \subset \msf{G}_{r, \ov{k}}^\circ. 
\]  

\begin{defi}
    We define the positive-depth parahoric Deligne-Lusztig variety $\msf{Y}_{P, r}$ as
    \[
        \msf{Y}_{P, r} = \{ h \in \msf{U}_r \cap \sigma^{-1}(\msf{U}_r) \backslash \msf{G}_{r, \ov{k}} \mid \sigma(h) h^{-1} \in \msf{U}_r \}. 
    \] 
    It has an action of $\msf{G}_r(k) \times \msf{M}_r(k)$ such that
    \[
        (g, m)\cdot h = m h g^{-1} \quad 
        (g \in \msf{G}_r(k), \; m \in \msf{M}_r(k), \; h \in \msf{U}_r \cap \sigma^{-1}(\msf{U}_r) \backslash \msf{G}_{r, \ov{k}}). 
    \]
\end{defi}
\begin{rmk}
    In the literature (e.g. \cite{CO25_JEMS}, \cite{Cha25}, \cite{CO25}), the affine fibration
    \[
        \msf{X}_{P, r} = \{ h \in \msf{G}_{r, \ov{k}} \mid \sigma(h) h^{-1} \in \msf{U}_r \}
    \]
    over $\msf{Y}_{P, r}$ is called a positive-depth parahoric Deligne-Lusztig variety. Here, we adopt $\msf{Y}_{P, r}$ because the associated Deligne-Lusztig induction should satisfy middle concentration in the generic case. In particular, the natural map $H_c^\bullet \to H^\bullet$ should behave better for $\msf{Y}_{P, r}$. 
\end{rmk}

Let $U^-$ be the opposite unipotent radical to $U$. We define $\msf{U}^-_r \subset \msf{G}_{r, \ov{k}}$ as previously. 

\begin{lem} \label{lem:dim_Y_Pr}
    As a perfect $\ov{k}$-scheme, $\msf{Y}_{P, r}$ is perfectly smooth of dimension $\dim \msf{U}_r \cap \sigma(\msf{U}^-_r)$. 
\end{lem}
\begin{proof}
    By construction, $\msf{X}_{P, r} \to \msf{U}_r$ is locally finite \'{e}tale, so $\msf{X}_{P, r}$ is perfectly smooth of dimension $\dim \msf{U}_r$. Thus, $\msf{Y}_{P, r}$ is perfectly smooth of dimension 
    $
        \dim \msf{U}_r - \dim \msf{U}_r \cap \sigma^{-1}(\msf{U}_r) = \dim \msf{U}_r \cap \sigma(\msf{U}^-_r). 
    $
\end{proof}

\begin{defi}
    For each $\msf{M}_r(k)$-representation $\rho$, we say that the complex
    \[
        R\Gamma_c(\msf{Y}_{P, r}, \IC_{\msf{Y}, r}) \mathop{\otimes}_{\msf{M}_r(k)} \rho. 
    \]
    of $\msf{G}_r(k)$-representations is the positive-depth parahoric Deligne-Lusztig induction of $\rho$. Here, $\IC_{\msf{Y}, r}$ denotes the intersection cohomology complex of $\msf{Y}_{P, r}$. 
\end{defi}

We will not keep track of Frobenius eigenvalues in this paper, so we will ignore Tate twists in $\IC_{\msf{Y}, r}$. In particular, $\IC_{\msf{Y}, r} = \Qla[\dim \msf{Y}_{P, r}]$. 

From now on, we fix a $(G, M)$-generic character $\phi \colon M(F) \to \Qlax$ of depth $r$ and take $X$ as in \Cref{defi:genericity}. By abuse of notation, the $k$-linear map $\msf{m}_{x, r} \to k$ given by $X$ is also denoted by $\phi$. The aim of the following sections is to compare 
\[
    H_c^\bullet(\msf{Y}_{P, r}, \Qla) \mathop{\otimes}_{\msf{M}_r(k)} (\rho \otimes \phi)
\]
with $\Ind_{r, \phi}(\rho)$ for an $\msf{M}_{r-}(k)$-representation $\rho$ (see \Cref{thm:positive_depth_DL_induction}). 

\subsection{Depth filtration}

In this section, we introduce a filtration $\msf{Y}_{P, [s, r]}$ of $\msf{Y}_{P, r}$ and study the contribution of generic positive-depth parahoric Deligne-Lusztig induction in each filtration. 

For $0 < s \leq r$, let 
\[
    \msf{G}_{[s, r]} = L^+\cl{G}_{x, s} /  L^+\cl{G}_{x, r+} \subset \msf{G}_r^\circ ,\quad 
    \msf{U}_{[s, r]} = L^+\cl{U}_{x, s} /  L^+\cl{U}_{x, r+} \subset \msf{U}_r. 
\]
We will apply the same notation to other groups and also to Lie algebras. 
\begin{defi} \label{defi:MP_DL_context}
    Let
    \[
        \msf{Y}_{P, [s,r]} = \{ h \in \msf{U}_{[s, r]} \cap \sigma^{-1}(\msf{U}_{[s, r]}) \backslash \msf{G}_{r, \ov{k}} \mid \sigma(h) h^{-1} \in \msf{U}_{[s, r]} \}. 
    \] 
    be a perfect $\ov{k}$-scheme with a natural $\msf{G}_r(k) \times \msf{M}_r(k)$-action. 
\end{defi}

\begin{prop} \label{prop:depth_filtration}
    The natural map $\msf{Y}_{P, [s, r]} \to \msf{Y}_{P, r}$ is a closed immersion equivariant under $\msf{G}_r(k) \times \msf{M}_r(k)$. We say that $\{ \msf{Y}_{P, [s, r]} \}$ is the depth filtration of $\msf{Y}_{P, r}$. 
\end{prop}
\begin{proof}
    The $\msf{G}_r(k) \times \msf{M}_r(k)$-equivariance is immediate. Now, the map
    \[
        \iota \colon \msf{U}_{[0, s-]} \cap \sigma^{-1}(\msf{U}_{[0, s-]}) \to \msf{U}_{[0, s-]} ,\quad u \mapsto \sigma(u) u^{-1}
    \]
    is injective since $I_0 = \phi$. Moreover, the Lang torsor under $\msf{G}^\circ_{s-}$ is finite \'{e}tale, so $\iota$ is also finite. Thus, $\iota$ is a closed immersion. Then, the claim follows since
    \[
        \msf{Y}_{P, [s, r]} \cong \{ h \in \msf{U}_r \cap \sigma^{-1}(\msf{U}_r) \backslash \msf{G}_{r, \ov{k}} \mid \sigma(h) h^{-1} \bmod \msf{U}_{s} \in \Img(\iota) \}. 
    \]    
\end{proof}

\begin{lem} \label{lem:induction_to_Gr}
    For $0 < s \leq \tfrac{r}{2}$, let 
    $
        \msf{G}_{2s, s} = L^+\cl{G}_{x, 2s, s} / L^+\cl{G}_{x, r+} \subset \msf{G}^\circ_r
    $ and 
    \[
        \msf{Y}_{P, 2s, s} = \{ h \in \msf{U}_{[s, r]} \cap \sigma^{-1}(\msf{U}_{[s, r]}) \backslash \msf{G}_{2s, s, \ov{k}} \mid \sigma(h) h^{-1} \in \msf{U}_{[s, r]} \}.  
    \]
    Then, we have 
    \[
        \msf{Y}_{P, [s, r]} \cong \msf{Y}_{P, 2s, s} \times^{\msf{G}_{2s, s}(k)} \msf{G}_r(k). 
    \]
\end{lem}
\begin{proof}
    There is a natural inclusion $\msf{Y}_{P, 2s, s} \subset \msf{Y}_{P, [s, r]}$. If one passes to the canonical $\msf{U}_{[s, r]} \cap \sigma^{-1}(\msf{U}_{[s, r]})$-fibrations, both sides become locally finite \'{e}tale over $\msf{U}_{[s, r]}$, so the inclusion is a $\msf{G}_{2s, s}(k)$-equivariant closed and open immersion. Then, the claim follows formally. 
\end{proof}

First, we study the cohomology of $\msf{Y}_{P, [r/2, r]}$. It is directly related to $\msf{W}_{I, r}$ and we will show that the generic contribution of positive-depth parahoric Deligne-Lusztig induction is concentrated at $\msf{Y}_{P, [r/2, r]}$.

\begin{lem} \label{lem:twisted_isomorphism}
    For every $s \leq r$, the map
    \[  
        (\msf{u}^+_{[s, r]} \cap \sigma(\msf{u}^+_{[s, r]})) \times (\msf{u}^+_{[s, r]} \cap \sigma(\msf{u}^-_{[s, r]})) \to \msf{u}^+_{[s, r]}, \quad (u, a) \mapsto u - \sigma^{-1}(u) + a
    \]      
    is an isomorphism. 
\end{lem}
\begin{proof}
    Since $\msf{u}^+_s / \msf{u}^+_{[s, r]} \cap \sigma(\msf{u}^-_{[s, r]}) \cong \msf{u}^+_s \cap \sigma(\msf{u}^+_{[s, r]})$, it is enough to show that 
    \begin{equation} \label{eq:sigma_twist}
        \msf{u}^+ \cap \sigma(\msf{u}^+_{[s, r]}) \to \msf{u}^+ \cap \sigma(\msf{u}^+_{[s, r]}) ,\quad u \mapsto u - \sigma^{-1}(u) \pmod{\msf{u}^+_{[s, r]} \cap \sigma(\msf{u}^-_{[s, r]})}
    \end{equation}
    is an isomorphism. We construct the inverse map. For each $u \in \msf{u}^+_{[s, r]} \cap \sigma(\msf{u}^+_{[s, r]})$, let 
    \[
        g(u) = \sigma^{-1}(u) \pmod{\msf{u}^+_{[s, r]} \cap \sigma(\msf{u}^-_{[s, r]})}
    \]  
    be the element in $\msf{u}^+_{[s, r]} \cap \sigma(\msf{u}^+_{[s, r]})$. Then, for every $n \geq 1$, we have
    \[
        g^n(u) \in \sigma(\msf{u}^+_{[s, r]}) \cap \msf{u}^+_{[s, r]} \cap \cdots \cap \sigma^{-n} (\msf{u}^+_{[s, r]}). 
    \]
    By \Cref{lem:ass_meaning}, $I_0 = \phi$, so $g^n(u) = 0$ for sufficiently large $n$. Then, it is easy to check that the map $u \mapsto \sum_{n \geq 0} g^n(u)$ is the inverse to \eqref{eq:sigma_twist}. 
\end{proof}

\begin{lem} \label{prop:W_inside_parabolic_DL_var}
    Let $\msf{\wtd{W}}_{I, r} = \{ g \in H_{x, r, \ov{k}} \mid \sigma(g) g^{-1} \in \msf{u}_{I, r/2} \}$. Then, there is a natural $\msf{G}_{r, r/2}(k)$-equivariant affine fibration 
    \[
        \msf{Y}_{P, r, r/2} \to \msf{\wtd{W}}_{I, r}
    \]
    of relative dimension $\dim \msf{U}_{[r/2+, r]} \cap \sigma(\msf{U}^-_{[r/2+, r]})$. 
\end{lem}
\begin{proof}
    First, \Cref{lem:twisted_isomorphism} implies that the natural map
    \[
        \msf{\wtd{W}}_{I, r} \to \{ g \in \msf{u}^+_{r/2} \cap \sigma^{-1}(\msf{u}^+_{r/2}) \backslash H_{x, r, \ov{k}} \mid \sigma(g) g^{-1} \in \msf{u}^+_{r/2} \}
    \]
    is an isomorphism. Moreover, there is a natural map
    \[  
        \msf{Y}_{P, r, r/2} \to \{ g \in \msf{u}^+_{r/2} \cap \sigma^{-1}(\msf{u}^+_{r/2}) \backslash H_{x, r, \ov{k}} \mid \sigma(g) g^{-1} \in \msf{u}^+_{r/2} \}
    \]  
    that is a quotient under the group action of
    \[
        \wtd{\msf{u}}^+_{[r/2+, r]} = \{ g \in \msf{m}^\perp_{x, [r/2 +, r]} / (\msf{u}^+_{[r/2+, r]} \cap \sigma^{-1}(\msf{u}^+_{[r/2+, r]})) \mid \sigma(g) - g \in \msf{u}^+_{[r/2 +, r]} \}. 
    \]  
    Since $I_0 = \phi$, the same proof as in \Cref{lem:pi_U_isom} implies that the projection $\msf{m}^\perp_{x,  [r/2+, r]} \to \msf{u}^+_{[r/2+, r]}$ induces $\wtd{\msf{u}}^+_{[r/2+, r]} \cong \msf{u}^+_{[r/2+, r]} \cap \sigma(\msf{u}^-_{[r/2+, r]})$ by \Cref{lem:twisted_isomorphism}. Thus, $\wtd{\msf{u}}^+_{[r/2+, r]}$ is an affine space of the desired dimension. 
\end{proof}

\begin{prop} \label{prop:geometric_Induction}
    For every $\msf{M}_{r-}(k)$-representation $\rho$, we have
    \[
        H_c^i(\msf{Y}_{P, [r/2, r]}, \Qla) \mathop{\otimes}\limits_{\msf{M}_r(k)} (\rho \otimes \phi) = \left\{
            \begin{alignedat}{4}
                & 0 & \quad & (i \neq \dim \msf{Y}_{P, r}) \\
                & \Ind_{r, \phi}(\rho) & \quad & (i = \dim \msf{Y}_{P, r}). 
            \end{alignedat}
        \right. 
    \]
\end{prop}
\begin{proof}
    By \Cref{lem:induction_to_Gr}, the left-hand side equals
    \[
        \cInd^{\msf{G}_r(k)}_{\msf{G}_{r, r/2}(k) \msf{M}_r(k)} \biggl(\cInd^{\msf{G}_{r, r/2}(k) \msf{M}_r(k)}_{\msf{G}_{r, r/2}(k)} H_c^i(\msf{Y}_{P, r, r/2}, \Qla)[\phi\vert_{\msf{m}_{x, r}}] \mathop{\otimes}\limits_{\msf{M}_r(k)} (\rho \otimes \phi)\biggr). 
    \]
    Let $d = \dim \msf{U}_{[r/2+, r]} \cap \sigma(\msf{U}^-_{[r/2+, r]})$. By \Cref{prop:W_inside_parabolic_DL_var}, we have 
    \[
        H_c^i(\msf{Y}_{P, r, r/2}, \Qla)[\phi\vert_{\msf{m}_{x, r}}] \cong H_c^{i - 2d}(\msf{W}_{I, r}, \Qla)[\psi] = \left\{
            \begin{alignedat}{4}
                & 0 & \quad & (i \neq \dim \msf{u}_{I, r/2} + 2d) \\
                & \kappa_{x, r, \phi} & \quad & (i = \dim \msf{u}_{I, r/2} + 2d). 
            \end{alignedat}
        \right. 
    \]
    Since $\phi$ is $(G, M)$-generic of depth $r$, we have $r \in \bb{Z}$, so 
    \[
        \dim \msf{u}_{I, r/2} + 2d = \dim \msf{U}_{[0, r]} \cap \sigma(\msf{U}^-_{[0, r]}) = \dim \msf{Y}_{P, r}
    \]
    by \Cref{lem:dim_Y_Pr}. The $\msf{G}_{r, r/2}(k)$-action on $H_c^\bullet(\msf{Y}_{P, r, r/2}, \Qla)$ factors through $\msf{H}_{x, r}$ by \Cref{prop:W_inside_parabolic_DL_var}, so it is enough to identify 
    \[
        \cInd_{\msf{H}_{x, r}}^{\msf{H}_{x, r}\msf{M}_r(k)} \kappa_{x, r, \phi} \mathop{\otimes}\limits_{\msf{M}_r(k)} (\rho \otimes \phi) \cong \kappa_{x, r, \phi} \otimes \rho, 
    \]
    where the action on the right-hand side is given as in \Cref{defi:twisted_positive_depth_induction}. 
    
    First, the twist by $\phi$ provides $\cInd^{\msf{M}_r(k)}_{\msf{m}_{x, r}} \phi\vert_{\msf{m}_{x, r}} \cong \Qla[\msf{M}_{r-}(k)]$. Then, 
    \[
        \cInd_{\msf{H}_{x, r}}^{\msf{H}_{x, r}\msf{M}_r(k)} \kappa_{x, r, \phi} \cong \kappa_{x, r, \phi} \otimes \Qla[\msf{M}_{r-}(k)]. 
    \]  
    For $a \in \kappa_{x, r, \phi}$ and $g \in \msf{M}_{r-}(k)$, the $\msf{H}_{x, r}\msf{M}_r(k)$ and the $\msf{M}_r(k)$-actions are given by 
    \[
        (m, h) \cdot (a \otimes \delta_g) = {}^gh  a \otimes \phi(m) \delta_{gm^{-1}}, \quad 
        m \cdot (a\otimes \delta_g) = ma \otimes \phi(m)^{-1} \delta_{mg}
    \]
    for $h \in \msf{H}_{x, r}$ and $m \in \msf{M}_r(k)$. Now, there is a natural map
    \[
        \kappa_{x, r, \phi} \otimes \Qla[\msf{M}_{r-}(k)] \to \Hom(\rho, \kappa_{x, r, \phi} \otimes \rho), \quad 
        a \otimes \delta_g \mapsto (v \mapsto g^{-1}a \otimes g^{-1} v).  
    \]
    It follows from the explicit actions given above that it induces an isomorphism
    \[
        \kappa_{x, r, \phi} \otimes \Qla[\msf{M}_{r-}(k)]  \mathop{\otimes}\limits_{\msf{M}_r(k)} (\rho \otimes \phi) \to \kappa_{x, r, \phi} \otimes \rho
    \]
    equivariant under $\msf{H}_{x, r}\msf{M}_r(k)$. 
\end{proof}

\subsection{Induction step in the positive depth}

We get a middle concentration
\[
    R\Gamma_c(\msf{Y}_{P, [r/2, r]}, \IC_{Y, r}\vert_{\msf{Y}_{P, [r/2, r]}}) \otimes (\rho \otimes \phi) \cong \Ind_{r, \phi}(\rho). 
\]
In fact, we expect that $R\Gamma_c(\msf{Y}_{P, r}, \IC_{Y, r}) \otimes (\rho \otimes \phi)$ should be concentrated on $\msf{Y}_{P, [r/2, r]}$ and the restriction map allows one to lift the above isomorphism to $\msf{Y}_{P, r}$. In this section, we prove this claim on $\msf{Y}_{P, [0+, r]}$ by induction along the depth filtration $\msf{Y}_{P, [s, r]}$ for $0 < s < \tfrac{r}{2}$. 

Recall $\msf{u}_s \cong \msf{U}_{[s, r]} / \msf{U}_{[s+, r]}$ with our notation. Let $\iota_s \colon \msf{u}^+_s \cap \sigma^{-1}(\msf{u}^+_s) \to \msf{u}^+_s$ be the map sending $u$ to $\sigma(u) - u$. As in the proof of \Cref{prop:depth_filtration}, $\iota_s$ is a closed immersion. 

\begin{lem} \label{lem:description_inv_image}
    Let $\msf{X}_{P, [s, r]} = \{ h \in \msf{G}_{r, \ov{k}} \mid \sigma(h) h^{-1} \in \msf{U}_{[s, r]} \}$ be an affine fibration of $\msf{Y}_{P, [s, r]}$ and let $Z \subset \msf{X}_{P, [s, r]}$ be the inverse image of $\msf{Y}_{P, [s+, r]}$. 
    Then we have 
    \[
        Z = \{ h \in \msf{X}_{P, [s, r]} \mid \sigma(h)h^{-1} \bmod{\msf{U}_{[s+, r]}} \in \Img(\iota_s) \}. 
    \]
\end{lem}
\begin{proof}
    The claim is immediate from construction. 
\end{proof}

We would like to show that the $\rho \otimes \phi$-isotypic part vanishes on $\msf{X}_{P, [s, r]} - Z$. In fact, one can show that $\phi\vert_{\msf{m}_{x, r}}$-isotypic part already vanishes. Let 
\[
    \msf{X}_{P, 2s, s} = \{ h \in \msf{G}_{2s, s, \ov{k}} \mid \sigma(h) h^{-1} \in \msf{U}_{[s, r]} \}. 
\]
As in \Cref{lem:induction_to_Gr}, $\msf{X}_{P, [s, r]} = \msf{X}_{P, 2s, s} \times^{\msf{G}_{2s, s}(k)} \msf{G}_r(k)$. Thus, it is enough to study the $\phi\vert_{\msf{m}_{x, r}}$-isotypic part on $\msf{X}_{P, 2s, s} - Z$. Let $\pr \colon \msf{X}_{P, 2s, s} \to \msf{U}_{[s, (r-s)-]}$ be the map sending $h$ to $\sigma(h) h^{-1}$ modulo $\msf{U}_{[r-s, r]}$. By \Cref{lem:description_inv_image}, we would like to show that
\begin{equation} \label{eq:sheaf_in_issue_s>0}
    \pr_! \Qla \mathop{\otimes}_{\msf{m}_{x, r}} \phi 
\end{equation} 
vanishes at each $u \in \msf{U}_{[s, (r-s)-]}$ that does not lie in $\Img(\iota_s)$ modulo $\msf{U}_{[s+, (r-s)-]}$. By the equivariance under $\msf{U}_{[s, r]} \cap \sigma^{-1}(\msf{U}_{[s, r]})$, the support of \eqref{eq:sheaf_in_issue_s>0} is stable under its $\sigma$-conjugation. Since $\msf{u}_s^+ = \Img(\iota_s) \oplus (\msf{u}^+_s \cap \sigma(\msf{u}^-_s))$ by \Cref{lem:twisted_isomorphism}, it is enough to prove the following. 

\begin{prop} \label{prop:vanishing_outside_Z}
    For every closed point $u \in \msf{U}_{[s, r]}$ that maps to $\msf{u}^+_s \cap \sigma(\msf{u}^-_s) - \{0\}$ modulo $\msf{U}_{[s+, r]}$, we have
    \[
        (\pr_! \Qla)_u \mathop{\otimes}_{\msf{m}_{x, r}} \phi  = 0. 
    \]
\end{prop}
\begin{proof}
    Fix $h \in \msf{G}_{2s, s}$ such that $\sigma(h) h^{-1} = u$. Let 
    \[
        Z_h = \{ z \in \msf{m}_{x, r}, \; a \in \msf{m}^\perp_{[r-s, r]} \mid \sigma(zah) (zah)^{-1} \in \msf{U}_{[r-s, r]} \cdot u \} \subset \pr^{-1}(u).  
    \]
    Here, the condition is rewritten as 
    \[
        (\sigma(z) - z + [u, a^{-1}]) \cdot (\sigma(a) - a) \in \msf{U}_{[r-s, r]}. 
    \]
    Since $\msf{U}_{[r-s, r]} \cong \msf{u}^+_{[r-s, r]} \subset \msf{m}^\perp_{[r-s, r]}$, $Z_h$ is finite \'{e}tale over $\msf{U}_{[r-s, r]}$. In particular, $Z_h \subset \pr^{-1}(u)$ is a closed and open subspace, so it is enough to show $R\Gamma_c(Z_h, \Qla) \mathop{\otimes}_{\msf{m}_{x, r}} \phi = 0 $. Here, the $\msf{m}_{x, r}$-action on $Z_h$ is given by addition to $z$. 

\begin{lem} \label{lem:reduce_wtd_Z_h}
    Let $\pi_{r} \colon \msf{g}_r \to \msf{m}_{x, r}$ be the natural projection and let 
    \[
        \wtd{Z}_h = \{ z \in \msf{m}_{x, r}, \; a \in \msf{m}^\perp_{r-s} \mid \sigma(z) - z = \pi_{r}([u, a]), \quad \sigma(a) - a \in \msf{u}^+_{r-s} \}. 
    \]
    Here, $\msf{m}_{x, r}$ acts by addition to $z$. The projection $\msf{m}^\perp_{[r-s, r]} \to \msf{m}^\perp_{r-s}$ induces an affine fibration $Z_h \to \wtd{Z}_h$ equivariant under $\msf{m}_{x, r}$. 
\end{lem}
\begin{proof}
    Let $(z, a) \in \wtd{Z}_h$ be a point valued in an arbitrary $\ov{k}$-algebra and fix a lift $\wtd{a} \in \msf{m}^\perp_{[r-s, r]}$ of $a$. Then, the fiber of $(z, a) \in \wtd{Z}_h$ in $Z_h$ classifies $a_+ \in \msf{m}^\perp_{[(r-s)+, r]}$ such that 
    \[
        \sigma(a_+) - a_+ \in (\id-\pi_r)([u, a^{-1}]) - (\sigma(\wtd{a}) - \wtd{a}) + \msf{u}^+_{[(r-s)+, r]}. 
    \]
    By the same proof as in \Cref{lem:pi_U_isom}, $a_+$ is uniquely determined by its projection to $\msf{u}^+_{[(r-s)+, r]}$ since $I_0 = \phi$. Since $\msf{u}^+_{[(r-s)+, r]}$ is an affine space, we get the claim. 
\end{proof}

    As a corollary, it is enough to prove $R\Gamma_c(\wtd{Z}_h, \Qla) \mathop{\otimes}_{\msf{m}_{x, r}} \phi = 0$. Let $\wtd{\msf{u}}^+ = \wtd{Z}_h / \msf{m}_{x, r}$. By the defining equation in \Cref{lem:reduce_wtd_Z_h}, 
    \[
        \wtd{\msf{u}}^+ \cong \{ a \in \msf{m}^\perp_{r-s} \mid \sigma(a) - a \in \msf{u}^+_{r-s} \}. 
    \]
    and $\wtd{Z}_h \times^{\msf{m}_{x, r}, \phi} \Qla \cong f^* \cl{L}_\psi$ where $f \colon \wtd{\msf{u}}^+ \to \bb{A}^1$ sends $a$ to $\phi(\pi_r([u, a]))$. Here, we implicitly rely on the choice of $X$. 
    Then, it is enough to show  $R\Gamma_c(\wtd{\msf{u}}^+, f^*\cl{L}_\psi) = 0$.  

    Let $u_s \in \msf{u}^+_s \cap \sigma(\msf{u}^-_s)$ be the reduction of $u$. As in the proof of \Cref{lem:pi_U_isom}, 
    for some maximal torus $T \subset M$, take a weight decomposition
    \[
        \msf{m}^\perp_{r-s} = \bigoplus_{\alpha \in \Phi^M_{r-s}} \msf{u}_{\alpha, r-s}. 
    \]
    Here, $\Phi^M_{r-s} \subset \Phi^M$ is the subset of roots $\alpha$ with $\dim \msf{u}_{\alpha, r-s} = 1$ and we have a decomposition $\Phi^M_{r-s} = \Phi^{M+}_{r-s} \sqcup \Phi^{M-}_{r-s}$. Take a basis $v_{\alpha, r-s} \in \msf{u}_{\alpha, r-s}$ so that $\sigma(v_{\alpha, r-s}) = v_{\sigma(\alpha), r-s}$. Then, 
    \[
        \wtd{\msf{u}}_+ \cong (\bb{A}^1)^{\Phi_{r-s}^{M+}} ,\quad 
        \sum u_\alpha v_{\alpha, r-s} \mapsto (u_\alpha)_{\alpha \in \Phi^{M+}_{r-s}}. 
    \]
    This is because $u_\alpha = u_{\sigma^{-d_\alpha}(\alpha)}^{q^{d_\alpha}}$ for each $\alpha \in \Phi^M_{r-s}$, where $d_\alpha \geq 0$ is the minimum integer with $\sigma^{-d_\alpha}(\alpha) \in \Phi^{M+}_{r-s}$. Now, take a similar basis $v_{- \alpha, s} \in \msf{u}_{- \alpha, s}$ for each $\alpha \in \Phi^M_{r-s}$ and write $u_s = \sum c_{-\alpha} v_{-\alpha, s}$. Then, $f$ is given by
    \[
        (\bb{A}^1)^{\Phi_{r-s}^{M+}} \to \bb{A}^1 , \quad 
        (u_\alpha)_{\alpha \in \Phi^{M+}_{r-s}} \mapsto \sum_{\alpha \in \Phi^M_{r-s}} c_{-\alpha} \phi([v_{-\alpha, s}, v_{\alpha, r-s}]) u_{\sigma^{-d_\alpha}(\alpha)}^{q^{d_\alpha}}. 
    \]
    Here, $\phi([v_{\alpha, r-s}, v_{-\alpha, s}]) \neq 0$ due to the $(G, M)$-genericity of $\phi$. Since $c_{-\alpha} = 0$ unless $-\alpha \in \Phi^{M+}_{s} \cap \sigma(\Phi^{M-}_s)$, $f$ is decomposed into a sum $\sum_{\alpha \in \Phi^{M+}_{r-s} \cap \sigma^{-1}(\Phi^{M-}_{r-s})} c'_\alpha u_\alpha^q$ and $c'_\alpha \neq 0$ for at least one $\alpha$. Then, it is easy to see $R\Gamma_c((\bb{A}^1)^{\Phi_{r-s}^{M+}}, f^*\cl{L}_\psi) = 0$ by the K\"{u}nneth formula. 
\end{proof}

As a summary, we get the following induction step for $s > 0$. 

\begin{cor} \label{cor:induction_step_positive}
    For every $0 < s < r/2$ and $i \geq 0$, the restriction map 
    \[
        H_c^i(\msf{Y}_{P, [s, r]}, \Qla) \mathop{\otimes}\limits_{\msf{M}_r(k)} (\rho \otimes \phi) \to H_c^i(\msf{Y}_{P, [s+, r]}, \Qla) \mathop{\otimes}\limits_{\msf{M}_r(k)} (\rho \otimes \phi) 
    \]
    is an isomorphism. In particular, we have
    \[
        H_c^i(\msf{Y}_{P, [s, r]}, \Qla) \mathop{\otimes}\limits_{\msf{M}_r(k)} (\rho \otimes \phi) = \left\{
            \begin{alignedat}{4}
                & 0 & \quad & (i \neq \dim \msf{Y}_{P, r}) \\
                & \Ind_{r, \phi}(\rho) & \quad & (i = \dim \msf{Y}_{P, r}). 
            \end{alignedat}
        \right. 
    \]
\end{cor}
\begin{proof}
    By the above argument and \Cref{prop:vanishing_outside_Z}, it follows that
    \[
        H_c^i(\msf{Y}_{P, [s, r]} - \msf{Y}_{P, [s+, r]}, \Qla) \mathop{\otimes}\limits_{\msf{M}_r(k)} (\rho \otimes \phi) = 0.
    \]
    Thus, we get the first claim. Then, the second claim follows from \Cref{prop:geometric_Induction} and the induction on $s$. 
\end{proof}

\subsection{Comparison at depth zero}

In this section, we study the contribution 
\[
    H_c^i(\msf{Y}_{P, r} - \msf{Y}_{P, [0+, r]}, \Qla) \mathop{\otimes}\limits_{\msf{M}_r(k)} (\rho \otimes \phi)
\]
outside $\msf{Y}_{P, [0+, r]}$. It turns out that this analysis is more complicated than the positive-depth induction step. We follow the approach in \cite{IN25} to show this vanishing when $P$ is \textit{convex}. 

First, we proceed with the previous strategy and point out the difficulty. Let $\msf{U}_0 = \msf{U}_{[0, r]} / \msf{U}_{[0+, r]}$ and let $\iota_0 \colon \msf{U}_0 \cap \sigma^{-1}(\msf{U}_0) \to \msf{U}_0$ be the closed immersion sending $u$ to $\sigma(u) u^{-1}$. 

\begin{lem} \label{lem:description_inv_image_s0}
    Let $Z \subset \msf{X}_{P, r}$ be the inverse image of $\msf{Y}_{P, [0+, r]}$. 
    Then we have 
    \[
        Z = \{ h \in \msf{X}_{P, r} \mid \sigma(h)h^{-1} \bmod{\msf{U}_{[0+, r]}} \in \Img(\iota_0) \}. 
    \]
\end{lem}
\begin{proof}
    The claim is immediate from construction. 
\end{proof}

We would like to show that the $\rho \otimes \phi$-isotypic part vanishes on $\msf{X}_{P, r} - Z$. Let 
\[
    \msf{X}^0_{P, r} = \{ h \in \msf{G}^0_{r, \ov{k}} \mid \sigma(h) h^{-1} \in \msf{U}_{[0, r]} \}. 
\]
We have $\msf{X}_{P, r} = \msf{X}_{P, r}^0 \times^{\msf{G}^0_{r}(k)} \msf{G}_r(k)$. Thus, it is enough if the $\phi\vert_{\msf{m}_{x, r}}$-isotypic part vanishes on $\msf{X}_{P, r}^0 - Z$. Let $\pr \colon \msf{X}_{P, r}^0 \to \msf{U}_{[0, r-]}$ be the map sending $h$ to $\sigma(h) h^{-1}$ modulo $\msf{u}^+_r$. By \Cref{lem:description_inv_image_s0}, we would like to show that
\begin{equation} \label{eq:sheaf_in_issue_s=0}
    \pr_! \Qla \mathop{\otimes}_{\msf{m}_{x, r}} \phi
\end{equation} 
vanishes at each $u \in \msf{U}_{[0, r-]}$ that does not lie in $\Img(\iota_0)$ modulo $\msf{U}_{[0+, r-]}$. 

We will compute the stalk $(\pr_! \Qla)_u \mathop{\otimes}_{\msf{m}_{x, r}} \phi$ as in the proof of \Cref{prop:vanishing_outside_Z}. Fix $h \in \msf{G}_r^0$ such that $\sigma(h) h^{-1} = u$. Let 
\[
    Z_h = \{ g \in \msf{g}_{x, r} \mid \sigma(gh) (gh)^{-1} \in \msf{u}^+_{r} \cdot u \} \subset \pr^{-1}(u).  
\]
Here, the condition is rewritten as 
\begin{equation} \label{eq:def_eq_Zh}
    \sigma(g) - u(g) \in \msf{u}_r^+. 
\end{equation}
In particular, $Z_h$ is finite \'{e}tale over $\msf{u}_{r}^+$ and $Z_h \subset \pr^{-1}(u)$ is a closed and open subspace. Since $r \in \bb{Z}$, $\msf{g}_r \cong \msf{g}_0$, so it is enough to show the following. 

\begin{ques} \label{conj:obstruction_for_depth_0_induction}
    For each $u \in \msf{U}_0$, let 
    \[
        Z_u = \{ g \in \msf{\ov{g}}_0 \mid \sigma(g) - ug \in \msf{u}_0^+ \}
    \]
    and let $\msf{m}_{x, 0}$ act on $Z_u$ by addition. For every $M$-stable $0$-generic map $\phi \colon \msf{m}_{x, 0} \to k$, determine if we always have
    \[
        R\Gamma_c(Z_u, \Qla) \mathop{\otimes}_{\msf{m}_{x, 0}} \phi = 0 
    \]
    for $u \notin \Img(\iota_0)$. Here, $\iota_0 \colon \msf{U}_0 \cap \sigma^{-1}(\msf{U}_0) \to \msf{U}_0$ is the closed immersion sending $h$ to $\sigma(h) h^{-1}$. 
\end{ques}
\begin{rmk}
    Here, we say that $\phi$ is $0$-generic if $\phi(\bar{H}_\alpha) \neq 0$ for every $\alpha \in \Phi^{\msf{M}}$ where $\bar{H}_\alpha \in \msf{g} \otimes \ov{k}$ is the reduction of $H_\alpha \in \breve{\mfr{g}}$. Since $Z_u$ is a commutative unipotent group scheme, its identity component is an affine space. In particular, the above conjecture only concerns the action $\msf{m}_{x, 0} \circlearrowright \pi_0(Z_u)$. 
\end{rmk}

We see several difficulties, such as the failure of group analogue of \Cref{lem:twisted_isomorphism} and the complicated relation between $u$ and $\sigma$. The convexity assumption made in \cite{IN25} is a clean way to simplify these points. Here, we define the convexity as follows. 

\begin{defi} \label{defi:convex_parabolic}
    We say that $P \subset \breve{G}$ is convex if there are total orders $\prec$ on $\Phi^{M+} \cap \sigma(\Phi^{M+})$ and $\Phi^{M-} \cap \sigma(\Phi^{M-})$ with the following conditions for every $\alpha, \beta \in \Phi^M$. 
    \begin{enumerate}
        \item When $\alpha, \beta \in \Phi^{M+} \cap \sigma(\Phi^{M+})$, 
        \begin{enumerate}
            \item $\sigma^{-1}(\alpha) \prec \alpha$ if $\sigma^{-1}(\alpha) \in \Phi^{M+} \cap \sigma(\Phi^{M+})$, 
            \item $\beta \prec \alpha$ if $\beta - \alpha \in \Phi^{M+} \cap \sigma(\Phi^{M-})$, and
            \item $\alpha + \beta \prec \alpha$ or $\alpha + \beta \prec \beta$ if $\alpha + \beta \in \Phi^M$. 
        \end{enumerate}
        \item When $\alpha, \beta \in \Phi^{M-} \cap \sigma(\Phi^{M-})$,
        \begin{enumerate}
            \item $\sigma(\alpha) \prec \alpha$ if $\sigma(\alpha) \in \Phi^{M-} \cap \sigma(\Phi^{M-})$, 
            \item $\beta \prec \alpha$ if $\beta - \alpha \in \Phi^{M+} \cap \sigma(\Phi^{M-})$. 
        \end{enumerate}
    \end{enumerate}
\end{defi}

\begin{prop}\textup{\cite[Proposition 3.5]{IN25}} \label{prop:existence_convex}
    For every $M \subset G$, there exists a convex parabolic subgroup $P \subset \breve{G}$ such that $\breve{M} \subset P$ is a Levi subgroup. 
\end{prop}
\begin{proof}
    Since $Z_M / Z_G$ is anisotropic, we can take an unramified elliptic maximal torus $T$ of $G$ inside $M$. Then, the claim follows from \cite[Proposition 3.5]{IN25} since the convexity condition introduced in \cite{NTY25} provides desired total orders by \cite[Lemma 3.4]{IN25} in terms of the function $n_\bullet$. Here, we use the Bruhat order to break ties. 
\end{proof}

The essential benefit of the convexity is the following. 

\begin{prop} \label{prop:benefit_convexity}
    When $P$ is convex, the following hold. 
    \begin{enumerate}
        \item The map
        \[
            \msf{U}_0 \cap \sigma^{-1}(\msf{U}_0) \times \msf{U}_0 \cap \sigma(\msf{U}^-_0) \to \msf{U}_0, \quad (u, a) \mapsto \sigma(u) a u^{-1}
        \]
        is an isomorphism. 
        \item For every $u \in \msf{U}_0 \cap \sigma(\msf{U}_0^-)$, the natural projection $\msf{g}_0 \twoheadrightarrow \msf{u}^+_0$ induces $Z_u / \msf{m}_{x, 0} \cong \msf{u}_0^+$. 
    \end{enumerate}
\end{prop}
\begin{proof}
    For (1), it is essentially proved in \cite[Theorem 0.1]{NTY25}. Here, we provide a proof to explain how to use condition (1) in \Cref{defi:convex_parabolic}. Fix an isomorphism 
    \[
        \msf{U}_0 \cap \sigma(\msf{U}_0) \times \msf{U}_0 \cap \sigma(\msf{U}_0^-) \cong \msf{U}_0, \quad (u, a) \mapsto ua. 
    \]
    For each $\alpha \in \Phi_0^{M+} \cap \sigma(\Phi_0^{M+})$, let 
    \[
        (\msf{U}_0 \cap \sigma(\msf{U}_0))_{\preceq \alpha} = \prod_{\substack{\beta \in \Phi_0^{M+} \cap \sigma(\Phi_0^{M+}) \\ \beta \preceq \alpha}} \msf{U}_{\beta, 0}. 
    \]
    The product order does not matter by condition (1)--(c). 
    Now, conditions (1)--(a) and (1)--(b) imply that for every $u \in (\msf{U}_0 \cap \sigma(\msf{U}_0))_{\preceq \alpha} \times  \msf{U}_0 \cap \sigma(\msf{U}_0^-)$, there is a unique $u_\alpha \in \msf{U}_{\alpha, 0}$ such that $u_\alpha u \sigma^{-1}(u_\alpha)^{-1} \in (\msf{U}_0 \cap \sigma(\msf{U}_0))_{\prec \alpha} \times  \msf{U}_0 \cap \sigma(\msf{U}_0^-)$. The claim follows by induction on $\alpha$. 

    For (2), the projection $\msf{g}_0 \to \msf{m}_{x, 0}^\perp$ induces an isomorphism
    \begin{equation} \label{eq:Zu_mod_m_x0}
        Z_u / \msf{m}_{x, 0} \cong \{ g \in \msf{m}_{x, 0}^\perp\mid \sigma(g) - ug \in \msf{\ov{m}}_{x, 0} \oplus \msf{u}^+ \}. 
    \end{equation}
    since $u - 1$ maps $\msf{\ov{m}}_{x, 0} \oplus \msf{u}^+_0$ into $\msf{u}^+_0$. First, we show that the projection $\msf{\ov{m}}_{x, 0}^\perp \to \msf{u}^+_0$ induces an injection $Z_u / \msf{m}_{x, 0} \to \msf{u}^+_0$. Since it is additive, it is enough to show that the kernel is trivial. 

    Let $g \in \msf{u}^-_0$ be an element such that $\sigma(g) - ug \in \msf{\ov{m}}_{x, 0} \oplus \msf{u}^+$. For each $\alpha \in \Phi^{M-}_0$, fix a basis $v_\alpha \in \msf{u}_{\alpha, 0}$ and let $g = \sum_{\alpha \in \Phi^{M-}_0} c_\alpha v_\alpha$. We inductively show $c_\alpha = 0$. 
    
    First, we show $c_\alpha = 0$ for $\alpha \in \Phi^{M-}_0 \cap \sigma(\Phi^{M+}_0)$. Since $\sigma^{-1}(\alpha) \in \Phi^{M+}_0$, the $\alpha$-component of $\sigma(g)$ is zero. Moreover, the $\alpha$-component of $ug - g$ is a linear combination of $c_\beta$ for $\alpha - \beta \in \Phi^{M+}_0 \cap \sigma(\Phi^{M-}_0)$. Then, $\beta$ also lies in $\Phi^{M-}_0 \cap \sigma(\Phi^{M+}_0)$, so we get $c_\alpha = 0$ by induction on $\alpha$ along the Bruhat order in $\Phi^{M-}_0 \cap \sigma(\Phi^{M+}_0)$. 
    
    Next, we show $c_\alpha = 0$ for $\alpha \in \Phi^{M-}_0 \cap \sigma(\Phi^{M-}_0)$. Suppose $c_\beta = 0$ for every $\beta \in \Phi^{M-}_0 \cap \sigma(\Phi^{M-}_0)$ with $\alpha \prec \beta$. By condition (2)--(a) and the case for $\Phi^{M-}_0 \cap \sigma(\Phi^{M+}_0)$, the $\alpha$-component of $\sigma(g)$ is zero. Moreover, the $\alpha$-component of $ug - g$ is a linear combination of $c_\beta$ for $\alpha - \beta \in \Phi^{M+}_0 \cap \sigma(\Phi^{M-}_0)$. By condition (2)--(b), we have $\alpha \prec \beta$. By the induction hypothesis, $c_\beta = 0$ and we get $c_\alpha = 0$.  

    Now, it follows that $Z_u / \msf{m}_{x, 0} \to \msf{u}^+_0$ is an injective additive morphism. We show that it is perfectly proper by checking the valuative criterion as in \cite[Proposition A.20]{Zhu17}. Let $R$ be a perfect valuation ring over $\ov{k}$ and let $K$ be its fraction field. Let $g \in (Z_u / \msf{m}_{x, 0})(K)$ be an element such that its projection to $\msf{u}_0^+$ lies in $\msf{u}_0^+(R)$. For the valuative criterion, it is enough to show that the coefficient $c_\alpha$ of the $\alpha$-component of $g$ lies in $R$ for every  $\alpha \in \Phi^{M-}_0$. This follows from the same induction argument as previously.
    
    Now, $Z_u / \msf{m}_{x, 0} \to \msf{u}^+_0$ is an injective perfectly proper additive morphism. By definition, $Z_u$ is finite \'{e}tale over $\msf{u}_0^+$, so $\dim (Z_u / \msf{m}_{x, 0}) = \dim Z_u = \dim \msf{u}^+_0$. Thus, $Z_u / \msf{m}_{x, 0} \to \msf{u}^+_0$ is bijective and universally closed, so it is universally homeomorphic. By \cite[Corollary A.16]{Zhu17}, it is an isomorphism. 
\end{proof}

\begin{thm} \label{thm:positive_depth_DL_induction}
    When $P$ is convex, \Cref{conj:obstruction_for_depth_0_induction} is affirmative. In particular, for every $i \geq 0$, the restriction map 
    \[
        H_c^i(\msf{Y}_{P, r}, \Qla) \mathop{\otimes}\limits_{\msf{M}_r(k)} (\rho \otimes \phi) \to H_c^i(\msf{Y}_{P, [0+, r]}, \Qla) \mathop{\otimes}\limits_{\msf{M}_r(k)} (\rho \otimes \phi) 
    \]
    is an isomorphism and we have
    \[
        R\Gamma_c(\msf{Y}_{P, r}, \IC_{\msf{Y}, r}) \mathop{\otimes}\limits_{\msf{M}_r(k)} (\rho \otimes \phi) = \Ind_{r, \phi}(\rho). 
    \]
\end{thm}
\begin{proof}
    By the equivariance of $\msf{X}^0_{P, r}$ under $\msf{U}_{[0, r]} \cap \sigma^{-1}(\msf{U}_{[0, r]})$, the support of \eqref{eq:sheaf_in_issue_s=0} is stable under its $\sigma$-conjugation. By \Cref{prop:benefit_convexity} (1), it is enough to treat $u \in \msf{U}_0 \cap \sigma(\msf{U}_0^-) - \{ 0 \}$ in \Cref{conj:obstruction_for_depth_0_induction}. First, we have a Cartesian diagram
    \begin{center}
        \begin{tikzcd}
            Z_u \ar[r] \ar[d] & \msf{\ov{m}}_{x, 0} \ar[d, "m \mapsto \sigma(m) - m"] \\
            Z_u / \msf{m}_{x, 0} \ar[r, "\pi"] & \msf{\ov{m}}_{x, 0}, 
        \end{tikzcd}
    \end{center} 
    where the bottom map $\pi$ sends $g \in Z_u / \msf{m}_{x, 0}$ in \eqref{eq:Zu_mod_m_x0} to the projection of $\sigma(g) - ug$ to $\msf{\ov{m}}_{x, 0}$. Regard $\phi$ as a $k$-linear form on $\msf{m}_{x, 0}$ via the identification $\msf{m}_{x, 0} \cong \msf{m}_{x, r}$ and let $f = \phi \circ \pi$. Then
    \[
        Z_u \times^{\msf{m}_{x, 0}, \phi} \Qla \cong f^* \cl{L}_\psi
    \]
    as a rank $1$ local system on $Z_u / \msf{m}_{x, 0}$. Now, fix an identification $Z_u / \msf{m}_{x, 0} \cong \msf{u}_0^+$ as in \Cref{prop:benefit_convexity} (2) and regard $f$ as a polynomial map $\msf{u}_0^+ \to \bb{A}^1$. Let $\wtd{\msf{u}}_0^+$ be the following Cartesian product
    \begin{center}
        \begin{tikzcd}
            \wtd{\msf{u}}_0^+ \ar[r] \ar[d] & \bb{A}^1 \ar[d, "a \mapsto a^q - a"] \\
            \msf{u}_0^+ \ar[r, "f"] & \bb{A}^1.  
        \end{tikzcd}
    \end{center} 
    Since $\wtd{\msf{u}}_0^+$ is a commutative unipotent group scheme, its identity component is an affine space. Then $R\Gamma_c(\wtd{\msf{u}}_0^+, \Qla) \cong \Qla^{\oplus \pi_0(\wtd{\msf{u}}_0^+)}[-2 \dim \msf{u}_0^+]$ and it is enough to show that $\wtd{\msf{u}}_0^+$ is connected.

    Fix a total order of $\Phi^{M+}_0 \cap \sigma(\Phi^{M-}_0)$ that is compatible with the Bruhat order and take a decomposition $\msf{U}_0 \cap \sigma(\msf{U}^-_0) \cong \prod_{\alpha \in \Phi^{M+}_0 \cap \sigma(\Phi^{M-}_0)} \msf{U}_{\alpha, 0}$ with respect to this order. Let $\alpha \in \Phi^{M+}_0 \cap \sigma(\Phi^{M-}_0)$ be the minimal element such that the $\alpha$-component of $u$ is nontrivial. 
    
    Consider the restriction of $f$ to $\msf{u}_{-\sigma^{-1}(\alpha), 0} \subset \msf{u}_0^+$. Let $t$ be the standard coordinate of $\msf{u}_{-\sigma^{-1}(\alpha), 0} \cong \bb{A}^1$. For each $\beta \in \Phi^{M}_0$, fix a basis $v_\beta \in \msf{u}_{\beta, 0}$ and let $g = \sum_{\beta \in \Phi^{M}_0} c_\beta v_\beta$ be a section of $\msf{u}_{-\sigma^{-1}(\alpha), 0} \subset Z_u / \msf{m}_{x, 0}$ via \eqref{eq:Zu_mod_m_x0}. Since $c_\beta = 0$ for $\beta \in \Phi^{M+}_0 \backslash \{-\sigma^{-1}(\alpha)\}$ and $c_\beta = t$ for $\beta = -\sigma^{-1}(\alpha)$, $\sigma(g) - ug \in \msf{\ov{m}}_{x, 0} \oplus \msf{u}^+$ implies
    \[
        c_\beta = \left\{
            \begin{alignedat}{4}
                & t^q & \quad & (\beta  = -\alpha) \\
                & 0 & \quad & (\beta \in \Phi^{M-}_0 \cap \sigma(\Phi^{M+}_0) \backslash \{-\alpha\})
            \end{alignedat}
        \right.
    \]
    by the minimality of $\alpha$. Then $f(g) = c t^q$ for some $c \in \ov{k}^\times$ by the $(G, M)$-genericity of $\phi$. Thus, the restriction of $\wtd{\msf{u}}_0^+$ to $\msf{u}_{-\sigma^{-1}(\alpha), 0}$ is connected. As $\wtd{\msf{u}}_0^+ \to \msf{u}_0^+$ is a Galois cover, it follows that $\wtd{\msf{u}}_0^+$ is also connected. 

    Thus, \Cref{conj:obstruction_for_depth_0_induction} is affirmative. Then, \eqref{eq:sheaf_in_issue_s=0} is concentrated on $\Img(\iota_0)$, so 
    \[
        R\Gamma_c(\msf{Y}_{P, r} - \msf{Y}_{P, [0+, r]}, \IC_{\msf{Y}, r}) \mathop{\otimes}\limits_{\msf{M}_r(k)} (\rho \otimes \phi) = 0. 
    \]
    Thus, the remaining claims follow from \Cref{cor:induction_step_positive}. 
\end{proof}

\begin{rmk}
    In particular, we get the middle concentration of the $\rho \otimes \phi$-isotypic part in $\msf{Y}_{P, r}$. It would be interesting to see if the natural map
    \[
        R\Gamma_c(\msf{Y}_{P, r}, \IC_{\msf{Y}, r}) \mathop{\otimes}\limits_{\msf{M}_r(k)} (\rho \otimes \phi) \to R\Gamma(\msf{Y}_{P, r}, \IC_{\msf{Y}, r}) \mathop{\otimes}\limits_{\msf{M}_r(k)} (\rho \otimes \phi) 
    \]
    is an isomorphism as in the case of $\msf{W}_{I, r}$ (see \Cref{lem:ass_meaning}) and the usual Deligne-Lusztig varieties (see \cite[Th\'{e}or\`{e}me 11.7]{BR03}). 
\end{rmk}

\begin{rmk}
    In particular, we verify \cite[Conjecture 5.12]{CO25} for convex parabolic subgroups, which conjectures the Euler characteristic formula of $R\Gamma_c(\msf{Y}_{P, r}, \IC_{\msf{Y}, r}) \otimes (\rho \otimes \phi)$. As a consequence of \cite[Proposition 6.4]{Cha25}, \cite[Theorem 5.6]{CO25} is recovered for suitably chosen Borel subgroups without technical conditions on residual characteristic. 
\end{rmk}


\renewcommand\bibfont{\footnotesize}
\printbibliography

@misc{Tak26_special_affinoids,
      title={Geometry around unramified CM points}, 
      author={Yuta Takaya},
      year={2026},
}

@misc{Cha25,
      title={The scalar product formula for parahoric Deligne--Lusztig induction}, 
      author={Charlotte Chan},
      year={2025},
      eprint={2405.00671},
      archivePrefix={arXiv},
      primaryClass={math.RT},
      url={https://arxiv.org/abs/2405.00671}, 
}

@misc{IN25,
      title={Convex elements and cohomology of deep level Deligne-Lusztig varieties}, 
      author={Alexander B. Ivanov and Sian Nie},
      year={2025},
      eprint={2503.13412},
      archivePrefix={arXiv},
      primaryClass={math.AG},
      url={https://arxiv.org/abs/2503.13412}, 
}

@article {NTY25,
    AUTHOR = {Nie, Sian and Tan, Panjun and Yu, Qingchao},
     TITLE = {Convex elements and {S}teinberg's cross-sections},
   JOURNAL = {Adv. Math.},
  FJOURNAL = {Advances in Mathematics},
    VOLUME = {475},
      YEAR = {2025},
     PAGES = {Paper No. 110346, 15},
      ISSN = {0001-8708,1090-2082},
   MRCLASS = {20G99 (20F55)},
  MRNUMBER = {4905421},
       DOI = {10.1016/j.aim.2025.110346},
       URL = {https://doi-org.utokyo.idm.oclc.org/10.1016/j.aim.2025.110346},
}

@misc{Nie24,
      title={Decomposition of higher Deligne-Lusztig representations}, 
      author={Sian Nie},
      year={2024},
      eprint={2406.06430},
      archivePrefix={arXiv},
      primaryClass={math.RT},
      url={https://arxiv.org/abs/2406.06430}, 
}

@article {CO25_JEMS,
    AUTHOR = {Chan, Charlotte and Oi, Masao},
     TITLE = {Geometric {$L$}-packets of {H}owe-unramified toral
              supercuspidal representations},
   JOURNAL = {J. Eur. Math. Soc. (JEMS)},
  FJOURNAL = {Journal of the European Mathematical Society (JEMS)},
    VOLUME = {27},
      YEAR = {2025},
    NUMBER = {4},
     PAGES = {1465--1526},
      ISSN = {1435-9855,1435-9863},
   MRCLASS = {22E50 (11F70 11S37)},
  MRNUMBER = {4875610},
       DOI = {10.4171/jems/1396},
       URL = {https://doi-org.utokyo.idm.oclc.org/10.4171/jems/1396},
}

@misc{Kal21,
      title={Supercuspidal L-packets}, 
      author={Tasho Kaletha},
      year={2021},
      eprint={1912.03274},
      archivePrefix={arXiv},
      primaryClass={math.RT},
      url={https://arxiv.org/abs/1912.03274}, 
}

@misc{IN26,
      title={Deep level Deligne--Lusztig induction for tamely ramified tori}, 
      author={Alexander B. Ivanov and Sian Nie},
      year={2026},
      eprint={2601.06929},
      archivePrefix={arXiv},
      primaryClass={math.RT},
      url={https://arxiv.org/abs/2601.06929}, 
}

@misc{CO25,
      title={Green functions for positive-depth Deligne--Lusztig induction}, 
      author={Charlotte Chan and Masao Oi},
      year={2025},
      eprint={2506.04449},
      archivePrefix={arXiv},
      primaryClass={math.RT},
      url={https://arxiv.org/abs/2506.04449}, 
}

@article {FKS23,
    AUTHOR = {Fintzen, Jessica and Kaletha, Tasho and Spice, Loren},
     TITLE = {A twisted {Y}u construction, {H}arish-{C}handra characters,
              and endoscopy},
   JOURNAL = {Duke Math. J.},
  FJOURNAL = {Duke Mathematical Journal},
    VOLUME = {172},
      YEAR = {2023},
    NUMBER = {12},
     PAGES = {2241--2301},
      ISSN = {0012-7094,1547-7398},
   MRCLASS = {22E50 (11S37)},
  MRNUMBER = {4654051},
MRREVIEWER = {Yongqi\ Feng},
       DOI = {10.1215/00127094-2022-0080},
       URL = {https://doi-org.utokyo.idm.oclc.org/10.1215/00127094-2022-0080},
}

@article {IT23,
    AUTHOR = {Imai, Naoki and Tsushima, Takahiro},
     TITLE = {Geometric construction of {H}eisenberg-{W}eil representations
              for finite unitary groups and {H}owe correspondences},
   JOURNAL = {Eur. J. Math.},
  FJOURNAL = {European Journal of Mathematics},
    VOLUME = {9},
      YEAR = {2023},
    NUMBER = {2},
     PAGES = {Paper No. 31, 34},
      ISSN = {2199-675X,2199-6768},
   MRCLASS = {20C33 (11F27)},
  MRNUMBER = {4581884},
MRREVIEWER = {Layla\ Sorkatti},
       DOI = {10.1007/s40879-023-00620-5},
       URL = {https://doi-org.utokyo.idm.oclc.org/10.1007/s40879-023-00620-5},
}

@misc{FS25,
      title={Construction of tame supercuspidal representations in arbitrary residue characteristic}, 
      author={Jessica Fintzen and David Schwein},
      year={2025},
      eprint={2501.18553},
      archivePrefix={arXiv},
      primaryClass={math.RT},
      url={https://arxiv.org/abs/2501.18553}, 
}

@article {Ger77,
    AUTHOR = {G\'{e}rardin, Paul},
     TITLE = {Weil representations associated to finite fields},
   JOURNAL = {J. Algebra},
  FJOURNAL = {Journal of Algebra},
    VOLUME = {46},
      YEAR = {1977},
    NUMBER = {1},
     PAGES = {54--101},
      ISSN = {0021-8693},
   MRCLASS = {20G05},
  MRNUMBER = {460477},
MRREVIEWER = {G.\ I.\ Lehrer},
       DOI = {10.1016/0021-8693(77)90394-5},
       URL = {https://doi-org.utokyo.idm.oclc.org/10.1016/0021-8693(77)90394-5},
       shorthand = {G\'{e}r77}
}

@article {Fin21,
    AUTHOR = {Fintzen, Jessica},
     TITLE = {On the construction of tame supercuspidal representations},
   JOURNAL = {Compos. Math.},
  FJOURNAL = {Compositio Mathematica},
    VOLUME = {157},
      YEAR = {2021},
    NUMBER = {12},
     PAGES = {2733--2746},
      ISSN = {0010-437X,1570-5846},
   MRCLASS = {22E50 (20G25)},
  MRNUMBER = {4357723},
MRREVIEWER = {Yongqi\ Feng},
       DOI = {10.1112/S0010437X21007636},
       URL = {https://doi-org.utokyo.idm.oclc.org/10.1112/S0010437X21007636},
}

@article {Kal19,
    AUTHOR = {Kaletha, Tasho},
     TITLE = {Regular supercuspidal representations},
   JOURNAL = {J. Amer. Math. Soc.},
  FJOURNAL = {Journal of the American Mathematical Society},
    VOLUME = {32},
      YEAR = {2019},
    NUMBER = {4},
     PAGES = {1071--1170},
      ISSN = {0894-0347,1088-6834},
   MRCLASS = {22E50 (11F70 11S37)},
  MRNUMBER = {4013740},
MRREVIEWER = {Amiya\ Kumar\ Mondal},
       DOI = {10.1090/jams/925},
       URL = {https://doi-org.utokyo.idm.oclc.org/10.1090/jams/925},
}

@article {Yu01,
    AUTHOR = {Yu, Jiu-Kang},
     TITLE = {Construction of tame supercuspidal representations},
   JOURNAL = {J. Amer. Math. Soc.},
  FJOURNAL = {Journal of the American Mathematical Society},
    VOLUME = {14},
      YEAR = {2001},
    NUMBER = {3},
     PAGES = {579--622},
      ISSN = {0894-0347,1088-6834},
   MRCLASS = {22E50},
  MRNUMBER = {1824988},
MRREVIEWER = {Bertrand\ Lemaire},
       DOI = {10.1090/S0894-0347-01-00363-0},
       URL = {https://doi.org/10.1090/S0894-0347-01-00363-0},
}

@book {SGA41/2,
    AUTHOR = {Deligne, Pierre},
     TITLE = {Cohomologie \'etale},
    SERIES = {Lecture Notes in Mathematics},
    VOLUME = {569},
      NOTE = {S\'eminaire de g\'eom\'etrie alg\'ebrique du Bois-Marie SGA
              $4\frac{1}{2}$},
 PUBLISHER = {Springer-Verlag, Berlin},
      YEAR = {1977},
     PAGES = {iv+312},
      ISBN = {3-540-08066-X; 0-387-08066-X},
   MRCLASS = {14F20},
  MRNUMBER = {463174},
MRREVIEWER = {J.\ S.\ Milne},
       DOI = {10.1007/BFb0091526},
       URL = {https://doi.org/10.1007/BFb0091526},
 shorthand = {SGA$4\frac{1}{2}$}
}

@article {Del74,
    AUTHOR = {Deligne, Pierre},
     TITLE = {La conjecture de {W}eil. {I}},
   JOURNAL = {Inst. Hautes \'Etudes Sci. Publ. Math.},
  FJOURNAL = {Institut des Hautes \'Etudes Scientifiques. Publications
              Math\'ematiques},
    NUMBER = {43},
      YEAR = {1974},
     PAGES = {273--307},
      ISSN = {0073-8301,1618-1913},
   MRCLASS = {14G13},
  MRNUMBER = {340258},
MRREVIEWER = {Nicholas\ M.\ Katz},
       URL = {http://www.numdam.org/item?id=PMIHES_1974__43__273_0},
}

@article {BW16,
    AUTHOR = {Boyarchenko, Mitya and Weinstein, Jared},
     TITLE = {Maximal varieties and the local {L}anglands correspondence for
              {$GL(n)$}},
   JOURNAL = {J. Amer. Math. Soc.},
  FJOURNAL = {Journal of the American Mathematical Society},
    VOLUME = {29},
      YEAR = {2016},
    NUMBER = {1},
     PAGES = {177--236},
      ISSN = {0894-0347,1088-6834},
   MRCLASS = {11S37 (11G18 11G25 14G22)},
  MRNUMBER = {3402698},
MRREVIEWER = {Charlotte\ On\ Chan},
       DOI = {10.1090/jams826},
       URL = {https://doi.org/10.1090/jams826},
}

@article {DL76,
    AUTHOR = {Deligne, Pierre and Lusztig, George},
     TITLE = {Representations of reductive groups over finite fields},
   JOURNAL = {Ann. of Math. (2)},
  FJOURNAL = {Annals of Mathematics. Second Series},
    VOLUME = {103},
      YEAR = {1976},
    NUMBER = {1},
     PAGES = {103--161},
      ISSN = {0003-486X},
   MRCLASS = {20G05 (14M15)},
  MRNUMBER = {393266},
MRREVIEWER = {S.\ I.\ Gel\cprime fand},
       DOI = {10.2307/1971021},
       URL = {https://doi.org/10.2307/1971021},
}

@misc{Mie16,
      title={Geometric approach to the explicit local {L}anglands correspondence},
      author={Mieda, Yoichi},
      year={2016},
      eprint={1605.00511},
      archivePrefix={arXiv},
      primaryClass={math.NT},
      url={https://arxiv.org/abs/1605.00511},
}

@article {BR03,
    AUTHOR = {Bonnaf\'e, C\'edric and Rouquier, Rapha\"el},
     TITLE = {Cat\'egories d\'eriv\'ees et vari\'et\'es de
              {D}eligne-{L}usztig},
   JOURNAL = {Publ. Math. Inst. Hautes \'Etudes Sci.},
  FJOURNAL = {Publications Math\'ematiques. Institut de Hautes \'Etudes
              Scientifiques},
    NUMBER = {97},
      YEAR = {2003},
     PAGES = {1--59},
      ISSN = {0073-8301,1618-1913},
   MRCLASS = {20G05 (18E30 20G40)},
  MRNUMBER = {2010739},
MRREVIEWER = {Jean\ Michel},
       DOI = {10.1007/s10240-003-0013-3},
       URL = {https://doi.org/10.1007/s10240-003-0013-3},
}

@book {KP23,
    AUTHOR = {Kaletha, Tasho and Prasad, Gopal},
     TITLE = {Bruhat-{T}its theory---a new approach},
    SERIES = {New Mathematical Monographs},
    VOLUME = {44},
 PUBLISHER = {Cambridge University Press, Cambridge},
      YEAR = {2023},
     PAGES = {xxx+718},
      ISBN = {978-1-108-83196-3},
   MRCLASS = {20E42 (11F70 20G25 22E50)},
  MRNUMBER = {4520154},
}

@article {Zhu17,
    AUTHOR = {Zhu, Xinwen},
     TITLE = {Affine {G}rassmannians and the geometric {S}atake in mixed
              characteristic},
   JOURNAL = {Ann. of Math. (2)},
  FJOURNAL = {Annals of Mathematics. Second Series},
    VOLUME = {185},
      YEAR = {2017},
    NUMBER = {2},
     PAGES = {403--492},
      ISSN = {0003-486X},
   MRCLASS = {14D24 (14L35 14M15 20G25)},
  MRNUMBER = {3612002},
MRREVIEWER = {Rolf Berndt},
       DOI = {10.4007/annals.2017.185.2.2},
       URL = {https://doi-org.utokyo.idm.oclc.org/10.4007/annals.2017.185.2.2},
}

\end{document}